\documentclass[11pt,a4paper]{article}
\usepackage{indentfirst}
\usepackage{fancyhdr}
\usepackage{color}
\usepackage{mathrsfs}
\usepackage{amsfonts,amssymb,amsmath,amsthm}
\usepackage{authblk}
\usepackage[unicode,pdftex]{hyperref}
\usepackage[dvipsnames, svgnames, x11names]{xcolor}  

\usepackage{tikz}  
\usetikzlibrary{positioning} 

\hypersetup{
	colorlinks = true,
	citecolor = green,
	anchorcolor = blue,
	linkcolor = blue}

\newtheorem{thm}{Theorem}[section]
\newtheorem{lemma}[thm]{Lemma}

\newtheorem{prop}[thm]{Proposition}
\newtheorem{defi}[thm]{Definition}
\newtheoremstyle{rem}{10pt}{10pt}{\rmfamily}{}{\bfseries}{.}{.5em}{} 
\theoremstyle{rem}
\newtheorem{rem}[thm]{Remark}

\numberwithin{equation}{section} 

\title{Local well-posedness for the Schr\"{o}dinger-KdV system in $H^{s_1}\times H^{s_2}$}
\author[1]{Yingzhe Ban}
\author[2]{Jie Chen}
\author[3,1]{Ying Zhang}
\affil[1]{\scriptsize \textit{The Graduate School of China Academy of Engineering Physics, Beijing, 100088, P.R. China}}
\affil[2]{\scriptsize \textit{School of Sciences, Jimei University, Xiamen 361021, P.R. China}}
\affil[3]{\scriptsize \textit{Academy of Mathematics and Systems Science, CAS, Beijing 100190, P.R. China}}
\date{}

\begin{document}
	\maketitle
	\begin{abstract}
		In this paper, we study local well-posedness theory of the Cauchy problem for Schr\"{o}dinger-KdV system in Sobolev spaces $H^{s_1}\times H^{s_2}$. We obtain the local well-posedness when $s_1\geq 0$, $\max\{-3/4,s_1-3\}\leq s_2\leq \min\{4s_1,s_1+2\}$. The result is sharp in some sense and improves previous one by Corcho-Linares \cite{corcho2007well}. The endpoint case $(s_1,s_2) = (0,-3/4)$ has been solved in \cite{guo2010well,wang2011cauchy}. We show the necessary and sufficient conditions for related estimates in Bourgain spaces. To solve the borderline cases, we use the $U^p-V^p$ spaces introduced by Koch-Tataru \cite{kochtataru} and function spaces constructed by Guo-Wang \cite{guo2010well}. We also use normal form argument to control the nonresonant interaction.
	\end{abstract}
	\section{Introduction}		
	We study the Cauchy problem for the Schr\"{o}dinger-KdV system
	\begin{equation*}\label{model}\tag{S-KdV}
		\left\{
		\begin{aligned}
			&i\partial_t u+\partial_{xx}u = \alpha uv+\beta |u|^2u,\ \  t,x\in \mathbb{R}\\
			&\partial_t v+\partial_{xxx}v+v\partial_x v = \gamma\partial_x(|u|^2),\\
			&(u,v)|_{t = 0} = (u_0,v_0)\in H^{s_1}\times H^{s_2}
		\end{aligned}
		\right.
	\end{equation*}
	where $\alpha, \beta, \gamma$ are real-valued constants, $u$ is complex-valued and $v$ is real-valued. In fact, for local wellposedness theory, the argument also works for complex-valued $v$ and $\alpha,\beta,\gamma$.
	This system arises in fluid mechanics as well as plasma physics. See \cite{corcho2007well} and reference therein.
	
	We recall some early results on this system. Tsutsumi \cite{tsutsumi1993well} obtained global well-posedness for initial data $(u_0,v_0)\in H^{s+1/2}\times H^s$ with $s\in \mathbb{N}$. For the case of $\beta = 0$, Guo-Miao \cite{guo1999well} obtained global well-posedness in $H^s\times H^s$, $s\in\mathbb{N}$. Bekiranov-Ogawa-Ponce \cite{bekiranov1997weak} obtained local well-posedness in $H^s\times H^{s-1/2}$, $s\geq 0$. Corcho-Linares \cite{corcho2007well} obtained local well-posedness in $H^{s_1}\times H^{s_2}$, 
	\begin{align*}
		&s_2>-3/4,~s_1-1\leq s_2\leq 2s_1-1/2,~s_1\in [0,1/2]~~\mathrm{or}\\
		&s_1>1/2,~s_1-1\leq s_2\leq s_1+1/2.
	\end{align*}
	Guo-Wang \cite{guo2010well} obtained local well-posedness in $L^2\times H^{-3/4}$. The same result was also showed in \cite{wang2011cauchy}. 
	There are stronger well-posedness results for KdV and  $1d$ cubic Schr\"{o}dinger equation. See \cite{killip2019kdv,harrop2020sharp}. These results rely heavily on the complete integrability of equations. However, Schr\"{o}digner--KdV system is not completely integrable. See for example \cite{bekiranov1997weak}. To approach lower regularity well-posedness than $L^2\times H^{-3/4}$ for \eqref{model} seems to be difficult.
	
	The main results in this paper are as follows:
	\begin{thm}\label{illpore}
		Let $\alpha,\beta,\gamma\neq 0$. The solution mapping of \eqref{model} from $\mathcal{S}\times \mathcal{S}$ to $C([0,T];H^{s_1}\times H^{s_2})$ can not be $C^3$ with $H^{s_1}\times H^{s_2}$ topology except 
		$$s_1\geq 0,~\max\{-3/4,s_1-3\}\leq s_2\leq \min\{4s_1,s_1+2\}.$$
		
		If $\beta = 0$, $\alpha\gamma\neq 0$, The solution mapping of \eqref{model} from $\mathcal{S}\times \mathcal{S}$ to $C([0,T];H^{s_1}\times H^{s_2})$ can not be $C^2$ with $H^{s_1}\times H^{s_2}$ topology except 
		$$s_1\geq -3/16,~\max\{-3/4,s_1-3\}\leq s_2\leq \min\{4s_1,s_1+2\}.$$
	\end{thm}
	
	\begin{thm}\label{mainlocalresults}
		\eqref{model} is local well-posed in $H^{s_1}\times H^{s_2}$ for
		$$s_1\geq 0,~ \max\{-3/4,s_1-3\}\leq s_2\leq \min\{4s_1,s_1+2\}.$$
	\end{thm}
%
%
%
%

	\begin{rem}
		We show Theorem \ref{illpore} by disproving some multi-linear estimates. One may obtain stronger ill-posedness result.
		For example, the solution mapping may be not uniformly continuous. Here, we only concern the region of $(s_1,s_2)$ that one can not construct solution by contraction mapping argument. Thus it is reasonable to say that Theorem \ref{mainlocalresults} is sharp.
	\end{rem}		
	\begin{rem}
		If $\beta = 0$, \eqref{model} is also local well-posed in $H^{s_1}\times H^{s_2}$ for
		$$s_1\geq -3/16,~ \max\{-3/4,s_1-3\}\leq s_2\leq \min\{4s_1,s_1+2\}.$$
		Since the case $(s_1,s_2) = (-3/16,-3/4)$ needs some extra efforts, we would study this problem in a forthcoming paper.
	\end{rem}
	
	The paper is organized as follows. In Section \ref{multiesinsobo}, we show the necessary and sufficient conditions for the boundedness of second Picard iteration in Sobolev spaces $H^{s_1}\times H^{s_2}$. Combining the ill-posedness result for single equation, we obtain Theorem \ref{illpore}. In Section \ref{multiinbo}, we establish some multi-linear estimates in Bourgain spaces which implies local well-posedness of \eqref{model} in $s_1\geq 0$, $\max\{-3/4,s_1/2-11/8,s_1-5/2\}<s_2<\min\{4s_1,s_1+1\}$ (the region $B$ in Figure \ref{figure}). In Section \ref{borderlinecase}, we first recall the definition and some basic properties of $U^p-V^p$ introduced by Koch-Tataru \cite{kochtataru}.  In Subsection \ref{cases2=4s1}, we consider the case $s_2 = \min\{4s_1,s_1+1\}$. In Subsection \ref{borderlinecaselow}, we consider the case $s_2 = -3/4$, $0\leq s_1<5/4$ by using the function spaces constructed by Guo-Wang \cite{guo2010well}. We use normal form argument to improve the estimates for nonresonant interaction in Section \ref{normalform}. In Subsection \ref{upperreg}, we consider the case 
	$4/3<s_1+1<s_2\leq\max\{4s_1, s_1+2\}$ (the region $A$ in Figure \ref{figure}). In Subsection \ref{lowerreg}, we consider the case 
	$s_2\geq -3/4$, $s_2+2<s_1\leq s_2+3$ ( containing the region $C$ in Figure \ref{figure}). Combining all these cases, we obtain Theorem \ref{mainlocalresults}.
	
		{\centering
		\begin{tikzpicture}[scale=1.6]
		\draw[->,thick](-0.5,0)--(4.5,0)node[below]{$s_1$};
		\draw[->,thick](0,-1)--(0,4)node[left]{$s_2$};
		\node[below left]at(0,0){$O$};
		\node[left]at(0,1){\tiny{$1$}};
		\node[below]at(1,0){\tiny{$1$}};
	
		\draw[gray, thin, loosely dotted,fill=yellow!10](2,4)--(0.67,2.67)--(0.33,1.33)--(3,4)--(2,4);
		\draw[gray, thin,  loosely dotted,fill=cyan!10](3,4)--(0.33,1.33)--(0,0)--(0,-0.75)--(1.25,-0.75)--(2.25,-0.25)--(4.5,2)--(4.5,4)--(3,4);
		\draw[gray, thin, loosely dotted,fill=pink!10](4.5,2)--(2.25,-0.25)--(1.25,-0.75)--(2.25,-0.75)--(2.25,-0.75)--(4.5,1.5)--(4.5,2);
		\draw[step=0.5cm, gray, ultra thin,loosely dotted] (-0.5,-1) grid (4.5,4); 

		\draw[thick](-0.2,0)--(4,0);
		\draw[gray, thin, densely dashed](0,-0.75)--(2.25,-0.75); 	\node[left] at(0,-0.75){\tiny{$-3/4$}}; 
		\draw[gray, thin, densely dashed](0,1.33)--(0.33,1.33); \node[left] at(0,1.34){\tiny{$4/3$}};
		\draw[gray, thin, densely dashed](0,2.66)--(0.66,2.66);  \node[left] at(0,2.66){\tiny{$8/3$}};
		
		\draw[gray, thin, densely dashed](0,0)--(1,4); \node at(1,4.1){\tiny{$4s_1$}}; 
		\draw[gray, thin, densely dashed](1.25,-0.75)--(2.25,-0.25)--(4.5,2);
		\node[right] at(4.5,2){\tiny{$s_1-5/2$}};
		\draw[gray,thin, densely dashed](0,0)--(0.33,1.33)--(3,4);\node[right] at(3,4.1){\tiny{$s_1+1$}};
		\draw[blue!60,thick](2,4)--(0.66,2.66)--(0,0)--(0,-0.75)--(2.25,-0.75)--(4.5,1.5); 
		\node[right] at(2,4.1){\tiny{$s_1+2$}};
		\node[right] at(4.5,1.5){\tiny{$s_1-3$}};
		
		\filldraw [gray] (1,0) circle (0.5pt);\node[below] at(1,0){\tiny{$1$}};
		\filldraw [gray] (2.25,-0.25) circle (0.5pt);\node[below] at(2.25,-0.25){\tiny{$(9/4,-1/4)$}};
		\filldraw [gray] (2.25,-0.75) circle (0.5pt);\node[below] at(2.25,-0.75){\tiny{$(9/4,-3/4)$}};
		\filldraw [gray] (1.25,-0.75) circle (0.5pt);\node[below] at(1.25,-0.75){\tiny{$(5/4,-3/4)$}};

		\node[below] at(2.3,1.8){$B$};
		\node[right] at(1.2,2.8){$A$};
		\node[below right] at(3,0.55){$C$};
	\end{tikzpicture}
	
	Figure 1\label{figure}

	}

	\textbf{Notations.}  For $a,b\in\mathbb{R}^+$, $a\lesssim b$ means that there exists $C>0$ such that $a\leq Cb$ and $a\sim b$ means $a\lesssim b\lesssim a$. $C$ may depends on $s_1,s_2$ or some fixed parameters. We use $\langle \xi\rangle$ to denote $(1+\xi^2)^{1/2}$. For $r\in [1,\infty]$, we denote the conjugate number of $r$ by $r'$.		
	We use $\chi_{\Omega}$ to denote the indicator of set $\Omega$.
	
	Let $\varphi$ be a even, smooth function and $\varphi|_{[-1,1]} = 1$, $\varphi|_{[-5/4,5/4]^c} = 0$. Define $\varphi_N=\varphi(\cdot/N)$, $\psi_N = \varphi_N-\varphi_{N/2}$. We always use $N,L$ to denote a dyadic number larger that $1$. In this article, we use inhomogeneous decomposition. Thus we define Littlewood-Paley projections $P_N$ by
	\begin{align*}
		P_1 f = \mathscr{F}^{-1}\varphi\mathscr{F}f,\quad P_Nf = \mathscr{F}^{-1}\psi_N\mathscr{F}f,~\forall~N\in 2^{\mathbb{N}}
	\end{align*}
	where we use $\mathscr{F}$ to denote the Fourier transform
	$$\mathscr{F}f(\xi) = \hat{f}(\xi) = \frac{1}{(2\pi)^{1/2}}\int_{\mathbb{R}}f(x)e^{-ix\xi}~d\xi.$$
	We use $J$ to denote the Fourier multiplier operator $\mathscr{F}^{-1}\langle \xi\rangle\mathscr{F}$. $\|u\|_{H^s}$ norm is defined by $\|J^su\|_{L^2}$.
	
	We use $P_{>N}$ to denote $\sum_{M>N}P_M $ and $P_{\leq N} = I-P_{>N}$. Sometimes we use $P_{\ll N}, P_{\sim N}, P_{\gg N}$ to denote $P_{\leq N/C}$, $P_{CN}-P_{N/C}$, $P_{>CN}$ where $C$ is a sufficiently large but fixed dyadic number.
	
	In many situations, we will do the Fourier transform for time variable. Thus we use $\mathscr{F}_x$, $\mathscr{F}_t$, $\mathscr{F}_{t,x}$ and $\mathscr{F}^{-1}_\xi$, $\mathscr{F}^{-1}_\tau$, $\mathscr{F}^{-1}_{\tau,\xi}$ to distinguish which variable that we perform Fourier transform on. With a little abuse of notation, we use $\hat{f}$ to denote $\mathscr{F}_{t,x}f$ if $f$ is a function on $\mathbb{R}_{t,x}$, and also use $\hat{f}$ to denote $\mathscr{F}_{x}f$ if $f$ is just a function on $\mathbb{R}_{x}$.
	
	We define $L_t^qL_x^r$, $L_x^rL_t^q$ by
	$\|u\|_{L_t^q L_x^r}:=\|\|u\|_{L_x^r}\|_{L^q_t}$, $\|u\|_{L_x^rL_t^q }:=\|\|u\|_{L^q_t}\|_{L_x^r}$ respectively. Also we define $L_T^qL_x^r$ by $\|u\|_{L_T^qL_x^r} = \|\|u\|_{L_x^r}\|_{L^q((0,T))}$ and similarly for $L_x^r L_T^q$. For a general function space on $\mathbb{R}^2$, we usually use $X_T$ to denote the space $X$ restricted on $[0,T]\times\mathbb{R}_x$.
	
	Let $S(t):=e^{it\partial_{xx}}$, $K(t):=e^{-t\partial_{xxx}}$ and
	\begin{align*}
		\mathscr{A}(f)(t)=\int_0^t S(t-t')f(t')~dt',\quad \mathscr{B}(g)(t)=\int_0^t\partial_x K(t-t')g(t')~dt'.
	\end{align*}
	
	We use the same notation as Corcho-Linares \cite{corcho2007well} to define Bourgain spaces.
	$$\|u\|_{X^{s,b}}:=\|\langle\xi\rangle^s\langle\tau+\xi^2\rangle^b\hat{u}\|_{L^2_{\tau,\xi}},~~\|v\|_{Y^{s,b}}:=\|\langle\xi\rangle^s\langle\tau-\xi^3\rangle^b\hat{v}\|_{L^2_{\tau,\xi}}.$$
	For the definition of $U^p-V^p$ and other notations, see Subsection \ref{defiupvp}. 
	
	In this paper we focus on local well-posedness. Without loss of generality, we assume that $\alpha = \beta = \gamma = 1$.
	
	\section{Multi-linear estimates}\label{multiesinsobo}
	In this section, we consider the boundedness of second Picard iteration. We evaluate the multi-linear terms on time $t = 1$.
	\begin{lemma}\label{ssstos}
		The inequalities
		\begin{align}
			\left\|\mathscr{A}(|S(t)u_0|^2S(t)u_0)(1)\right\|_{H^{s_1}}&\lesssim \|u_0\|_{H^{s_1}}^3;\label{ssstosin}\\
			\left\|\mathscr{A}(S(t)u_0 K(t)v_0)(1)\right\|_{H^{s_1}}&\lesssim \|u_0\|_{H^{s_1}}\|v_0\|_{H^{s_2}};\label{schkdvtosh}\\
			\left\|\mathscr{B}((K(t)v_0)^2)(1)\right\|_{H^{s_2}}&\lesssim \|v_0\|_{H^{s_2}}^2;\label{kdvvv}\\
			\left\|\mathscr{B}(|S(t)u_0|^2)(1)\right\|_{H^{s_2}}&\lesssim \|u_0\|_{H^{s_1}}^2\label{sstokdv}
		\end{align}
		hold if and only if $s_1\geq 0$; $s_2\geq -1$; $s_2\geq \max\{-1,-3+|s_1|\}$; $s_1\geq -1/2,s_2\leq \min\{4s_1,s_1+2\}$ respectively.
	\end{lemma}
	\begin{proof}[\textbf{Proof}]
		By the local well-posedness theory of \eqref{model} in $L^2\times H^{-3/4}$, inequalities \eqref{ssstosin}-- \eqref{sstokdv} hold if $s_1\geq 0$, $s_2\geq -3/4$.
		
		\textbf{Necessary part of \eqref{ssstosin}.} For $s<0$ one has
		\begin{align*}
			&\quad\left\|\mathscr{A}(|S(t)u_0|^2S(t)u_0)(1)\right\|_{H^s}\\
			&\sim \left\|\langle \xi\rangle^s\int_{\mathbb{R}^2}\int_0^1 e^{2it'(\xi_1-\xi_2)(\xi-\xi_1)}\widehat{u_0}(\xi_1)\overline{\widehat{u_0}}(\xi_2)\widehat{u_0}(\xi-\xi_1+\xi_2)~d\xi_1d\xi_2dt'\right\|_{L^2}.
		\end{align*}
		Let $N\gg 1$, $0<c\ll 1$, $\widehat{u_0}(\xi) = \chi_{[N,N+c]}(\xi)$. $\forall$ $\xi \in [N,N+c/2]$, one has
		\begin{align*}
			\left|\langle \xi\rangle^s\int_{\mathbb{R}^2}\int_0^1 e^{2it'(\xi_1-\xi_2)(\xi-\xi_1)}\widehat{u_0}(\xi_1)\overline{\widehat{u_0}}(\xi_2)\widehat{u_0}(\xi-\xi_1+\xi_2)~d\xi_1d\xi_2dt'\right|\sim N^{s}.
		\end{align*}
		Thus one has $\left\|\mathscr{A}(|S(t)u_0|^2S(t)u_0)(1)\right\|_{H^s}\gtrsim N^{s}$.
		Note that $\|u_0\|_{H^s}^3\sim N^{3s}$. Then \eqref{ssstosin} can not hold for $s<0$.
		Let $f(\xi) = \widehat{u_0}(\xi)\langle\xi\rangle^{s_1}$, $g(\xi) = \widehat{v_0}(\xi)\langle\xi\rangle^{s_2}$. \eqref{schkdvtosh} is equivalent to
		\begin{align*}
			\left\|\langle\xi\rangle^{s_1}\int_{\mathbb{R}}\frac{1-e^{i(\xi^2-\xi_1^2+(\xi-\xi_1)^3)}}{\xi^2-\xi_1^2+(\xi-\xi_1)^3}\frac{f(\xi_1)g(\xi-\xi_1)}{\langle\xi_1\rangle^{s_1}\langle\xi-\xi_1\rangle^{s_2}}~d\xi_1\right\|_{L^2}\lesssim \|f\|_{L^2}\|g\|_{L^2}
		\end{align*}
		\textbf{Necessary part of \eqref{schkdvtosh}.}
		Let $N\gg 1$, $f=\chi_{[1,2]}$, $g = \chi_{[N,N+4]}$. $\forall~\xi\in [N+2,N+3]$ one has
		\begin{align*}
			\left|\langle\xi\rangle^{s_1}\int_{\mathbb{R}}\frac{1-e^{i(\xi^2-\xi_1^2+(\xi-\xi_1)^3)}}{\xi^2-\xi_1^2+(\xi-\xi_1)^3}\frac{f(\xi_1)g(\xi-\xi_1)}{\langle\xi_1\rangle^{s_1}\langle\xi-\xi_1\rangle^{s_2}}~d\xi_1\right|\sim N^{s_1-s_2-3}.
		\end{align*}
		Note that $\|f\|_{L^2}\|g\|_{L^2}\sim 1$. Thus, $s_1-s_2-3\leq 0$ which means $s_2\geq -3+s_1$.
		
		Let $f = \chi_{[N,N+1]}$, $g = \chi_{[-N-4,-N+4]}$. $\forall$ $\xi\in[1,2]$, one has 
		\begin{align*}
			\left|\langle\xi\rangle^{s_1}\int_{\mathbb{R}}\frac{1-e^{i(\xi^2-\xi_1^2+(\xi-\xi_1)^3)}}{\xi^2-\xi_1^2+(\xi-\xi_1)^3}\frac{f(\xi_1)g(\xi-\xi_1)}{\langle\xi_1\rangle^{s_1}\langle\xi-\xi_1\rangle^{s_2}}~d\xi_1\right|\sim N^{-3-s_1-s_2}
		\end{align*}
		Thus, $-3-s_1-s_2\leq 0$ which means $s_2\geq -3-s_1$.
		
		For $0<c\ll 1$, let $f = \chi_{[-N-cN^{-1},-N]}$,
		$$g = \chi_{[a-N^{-1},a+N^{-1}]}, \quad a = -\frac{1}{2}+\sqrt{\frac{3}{2}N-\frac{1}{4}}.$$
		Then, $\forall$ $\xi\in [a-N,a-N+cN^{-1}]$, one has
		$
		|\xi^2-\xi_1^2+(\xi-\xi_1)^3|\ll 1
		$
		and then
		\begin{align*}
			\left|\langle\xi\rangle^{s_1}\int_{\mathbb{R}}\frac{1-e^{i(\xi^2-\xi_1^2+(\xi-\xi_1)^3)}}{\xi^2-\xi_1^2+(\xi-\xi_1)^3}\frac{f(\xi_1)g(\xi-\xi_1)}{\langle\xi_1\rangle^{s_1}\langle\xi-\xi_1\rangle^{s_2}}~d\xi_1\right|\sim N^{-1-\frac{s_2}{2}}.
		\end{align*}
		Thus, 
		\begin{align*}
			\left\|\langle\xi\rangle^{s_1}\int_{\mathbb{R}}\frac{1-e^{i(\xi^2-\xi_1^2+(\xi-\xi_1)^3)}}{\xi^2-\xi_1^2+(\xi-\xi_1)^3}\frac{f(\xi_1)g(\xi-\xi_1)}{\langle\xi_1\rangle^{s_1}\langle\xi-\xi_1\rangle^{s_2}}~d\xi_1\right\|_{L^2}\gtrsim N^{-1-\frac{s_2}{2}-\frac{1}{2}}
		\end{align*}
		
		Note that $\|f\|_{L^2}\|g\|_{L^2}\sim N^{-1}$. Thus, one has $s_2\geq -1$.
		
		\noindent\textbf{Sufficient part of \eqref{schkdvtosh}.} We decompose the integration into two parts.
		
		\noindent{\hypertarget{cs12}{\bf{Case 1}.}} $\Omega_1$: $\langle \xi^2-\xi_1^2+(\xi-\xi_1)^3\rangle\gtrsim \langle\xi-\xi_1\rangle^3$.
		
		\noindent{\hypertarget{cs22}{\bf{Case 2}.}} $\Omega_2$: $\langle \xi^2-\xi_1^2+(\xi-\xi_1)^3\rangle\ll \langle\xi-\xi_1\rangle^3$.
		
		\hyperlink{cs12}{\textcolor[RGB]{0,0,0}{{\bf{Case 1}}}}. By the definition of $\Omega_1$,  $s_2+3\geq \max\{|s_1|,2\}$ and Cauchy-Schwarz inequality, one has
		\begin{align*}
			&\quad\left\|\langle\xi\rangle^{s_1}\int_{\mathbb{R}}\chi_{\Omega_1}\frac{1-e^{i(\xi^2-\xi_1^2+(\xi-\xi_1)^3)}}{\xi^2-\xi_1^2+(\xi-\xi_1)^3}\frac{f(\xi_1)g(\xi-\xi_1)}{\langle\xi_1\rangle^{s_1}\langle\xi-\xi_1\rangle^{s_2}}~d\xi_1\right\|_{L^2}\\
			&\lesssim \left\|\int_{\mathbb{R}}\frac{\langle\xi\rangle^{s_1}}{\langle\xi-\xi_1\rangle^{3+s_2}\langle\xi_1\rangle^{s_1}}|f(\xi_1)||g(\xi-\xi_1)|~d\xi_1\right\|_{L^2}\\
			&\lesssim \left\|\int_{\mathbb{R}}(\langle\xi\rangle^{-2}+\langle\xi_1\rangle^{-2})|f(\xi_1)||g(\xi-\xi_1)|~d\xi_1\right\|_{L^2}\\    
			&\lesssim \|f\|_{L^2}\|g\|_{L^2}\|\langle\xi\rangle^{-2}\|_{L^2_\xi}+\|\langle\xi_1\rangle^{-2}\|_{L^2_{\xi_1}}\|f(\xi_1)g(\xi-\xi_1)\|_{L^2_{\xi,\xi_1}}\\
			&\lesssim \|f\|_{L^2}\|g\|_{L^2}.
		\end{align*}				
		
		\hyperlink{cs22}{\textcolor[RGB]{0,0,0}{{\bf{Case 2}}}}. By the definition of $\Omega_2$, we have $1\ll|\xi-\xi_1|\ll |\xi|\sim |\xi_1|$. By Cauchy-Schwarz inequality, we only need to control
		\begin{align*}
			\left\|\langle\xi\rangle^{s_1}\chi_{\Omega_2}\frac{1-e^{i(\xi^2-\xi_1^2+(\xi-\xi_1)^3)}}{\xi^2-\xi_1^2+(\xi-\xi_1)^3}\frac{1}{\langle\xi_1\rangle^{s_1}\langle\xi-\xi_1\rangle^{s_2}}\right\|_{L^\infty_\xi L^2_{\xi_1}}
		\end{align*}
		Due to $s_2\geq -1$ and the support set property of $\Omega_2$, one has
		\begin{align*}
			&\quad\left\|\langle\xi\rangle^{s_1}\chi_{\Omega_2}\frac{1-e^{i(\xi^2-\xi_1^2+(\xi-\xi_1)^3)}}{\xi^2-\xi_1^2+(\xi-\xi_1)^3}\frac{1}{\langle\xi_1\rangle^{s_1}\langle\xi-\xi_1\rangle^{s_2}}\right\|_{L^\infty_\xi L^2_{\xi_1}}\\
			&\lesssim\left\|\chi_{\Omega_2}\frac{1-e^{i(\xi^2-\xi_1^2+(\xi-\xi_1)^3)}}{\xi+\xi_1+(\xi-\xi_1)^2}\right\|_{L^\infty_\xi L^2_{\xi_1}}\lesssim\left\|\chi_{1\ll |\xi_2|\ll |\xi|}\frac{1-e^{i\xi_2(2\xi-\xi_2+\xi_2^2)}}{2\xi-\xi_2+\xi_2^2}\right\|_{L^\infty_\xi L^2_{\xi_2}}
		\end{align*}
		Let $y(\xi_2) = 2\xi-\xi_2+\xi_2^2$. Since $1\ll |\xi_2|$, one has $|dy|\sim |\xi_2d\xi_2|$. It is easy to control the integral for $|y(\xi_2)|\gtrsim 1$. For $|y(\xi_2)|\ll 1$, one has $|dy|\sim |\xi|^{1/2}|d\xi_2|$, $|1-e^{i\xi_2(2\xi-\xi_2+\xi_2^2)}|\lesssim \min\{1,|\xi|^{1/2}|y(\xi_2)|\}$. Thus, for $|\xi|\gg 1$, one has
		\begin{align*}
			\left\|\chi_{|y(\xi_2)|\ll 1}\frac{1-e^{i\xi_2(2\xi-\xi_2+\xi_2^2)}}{2\xi-\xi_2+\xi_2^2}\right\|_{L^2_{|\xi_2|\gg 1}}^2\lesssim \int_{|y|\ll 1}\frac{\min\{1,|\xi|y^2\}}{y^2}~\frac{dy}{|\xi|^{1/2}}\lesssim 1.
		\end{align*}
		We finish the proof for \eqref{schkdvtosh}.
		
		By polarization \eqref{kdvvv} is equivalent to
		\begin{align}\label{bilinearkdv}
			\left\|\mathscr{B}((K(t)u_0)K(t)v_0)(1)\right\|_{H^s}\lesssim \|u_0\|_{H^s}\|v_0\|_{H^s}.
		\end{align}
		By the definition of $H^s$-norm, one has
		\begin{align*}
			&\quad\left\|\mathscr{B}((K(t)u_0)K(t)v_0)(1)\right\|_{H^s}\\
			&\sim \left\|\langle\xi\rangle^s\xi\int_{\mathbb{R}}\int_0^1e^{-3it'\xi\xi_1(\xi-\xi_1)}\widehat{u_0}(\xi_1)\widehat{v_0}(\xi-\xi_1)~dt'd\xi_1\right\|_{L^2}
		\end{align*}
		\textbf{Necessary part of \eqref{kdvvv}.} Let 
		$N\gg 1,~~u_0 = \chi_{[N,N+1]},~~v_0 = \chi_{[-N-2,-N+2]}$.  
		$\forall$ $\xi\in [1,2]$, one has
		\begin{align*}
			\left|\langle\xi\rangle^s\xi\int_{\mathbb{R}}\int_0^1e^{-3it'\xi\xi_1(\xi-\xi_1)}\widehat{u_0}(\xi_1)\widehat{v_0}(\xi-\xi_1)~dt'd\xi_1\right|\sim  \frac{1}{N^2}.
		\end{align*}
		Thus, 
		\begin{align*}
			\left\|\int_0^1 e^{t'\partial_{xxx}}\partial_x(e^{-t'\partial_{xxx}}u_0e^{-t'\partial_{xxx}}v_0)~dt'\right\|_{H^s}\gtrsim N^{-2}.
		\end{align*}
		Since $\|u_0\|_{H^s}\|v_0\|_{H^s}\sim N^{2s}$, we obtain $s\geq -1$.	
		
		\noindent\textbf{Sufficient part of \eqref{kdvvv}.} For $s \geq -1$, one has
		\begin{align*}
			&\quad \left\|\mathscr{B}((K(t)u_0)K(t)v_0)(1)\right\|_{H^{-1}}\\
			&\sim \left\|\langle\xi\rangle^{s}\int_{\mathbb{R}}\frac{(1-e^{-3i\xi\xi_1(\xi-\xi_1)})\langle\xi_1\rangle^s\langle\xi-\xi_1\rangle^s}{\xi_1(\xi-\xi_1)}\frac{\widehat{u_0}(\xi_1)\widehat{v_0}(\xi-\xi_1)}{\langle\xi_1\rangle^s\langle\xi-\xi_1\rangle^s}~d\xi_1\right\|_{L^2}\\
			&\lesssim \left\|\int_{\mathbb{R}}\frac{\min\{1,|\xi\xi_1(\xi-\xi_1)|\}\langle\xi_1\rangle^s\langle\xi-\xi_1\rangle^s}{\langle\xi\rangle|\xi_1(\xi-\xi_1)|}\frac{|\widehat{u_0}(\xi_1)\widehat{v_0}(\xi-\xi_1)|}{\langle\xi_1\rangle^{-1}\langle\xi-\xi_1\rangle^{-1}}~d\xi_1\right\|_{L^2}
		\end{align*}
		Let $f(\xi) = \langle\xi\rangle^{s}\widehat{u_0}(\xi)$, $g(\xi) = \langle\xi\rangle^{s}\widehat{v_0}(\xi)$. We only need to prove
		\begin{align*}
			&\quad\left\|\int_{\mathbb{R}}\frac{\min\{1,|\xi\xi_1(\xi-\xi_1)|\}\langle\xi_1\rangle\langle\xi-\xi_1\rangle}{\langle\xi\rangle|\xi_1(\xi-\xi_1)|}f(\xi_1)g(\xi-\xi_1)~d\xi_1\right\|_{L^2}\\
			&\lesssim \|f\|_{L^2}\|g\|_{L^2}
		\end{align*}
		
		We decompose the integration into two parts.
		
		\noindent{\hypertarget{cs11}{\bf{Case 1}.}} $\Omega_1$: $|\xi_1|\gtrsim 1$, $|\xi-\xi_1|\gtrsim 1$ or  $|\xi_1|\ll 1$, $|\xi-\xi_1|\lesssim 1$.
		
		\noindent{\hypertarget{cs21}{\bf{Case 2}.}} $\Omega_2$: $|\xi_1|\ll 1\ll|\xi-\xi_1|$ or $|\xi_1|\gg 1\gg|\xi-\xi_1|$.
		
		For \hyperlink{cs11}{\textcolor[RGB]{0,0,0}{{\bf{Case 1}}}}, by Cauchy-Schwarz inequality, one has
		\begin{align*}
			&\quad\left\|\int_{\mathbb{R}}\frac{\min\{1,|\xi\xi_1(\xi-\xi_1)|\}\langle\xi_1\rangle\langle\xi-\xi_1\rangle}{\langle\xi\rangle|\xi_1(\xi-\xi_1)|}\chi_{\Omega_1}f(\xi_1)g(\xi-\xi_1)~d\xi_1\right\|_{L^2}\\
			&\lesssim \left\|\int_{\mathbb{R}}\frac{|f(\xi_1)g(\xi-\xi_1)|}{\langle\xi\rangle}~d\xi_1\right\|_{L^2}\lesssim \|f\|_{L^2} \|g\|_{L^2}\|\langle\xi\rangle^{-1}\|_{L^2}\\
			&\lesssim \|f\|_{L^2} \|g\|_{L^2}.
		\end{align*}
		For \hyperlink{cs21}{\textcolor[RGB]{0,0,0}{{\bf{Case 2}}}}, by symmetry, we only consider $|\xi_1|\ll 1\ll |\xi-\xi_1|$.	
		By Cauchy-Schwarz inequality, we have
		\begin{align*}
			&\quad\left\|\int_{\mathbb{R}}\frac{\min\{1,|\xi\xi_1(\xi-\xi_1)|\}\langle\xi_1\rangle\langle\xi-\xi_1\rangle}{\langle\xi\rangle|\xi_1(\xi-\xi_1)|}\chi_{|\xi_1|\ll 1\ll|\xi-\xi_1|}f(\xi_1)g(\xi-\xi_1)~d\xi_1\right\|_{L^2}\\
			&\lesssim \left\|\frac{\min\{1,\xi^2\xi_1\}}{\xi\xi_1}\right\|_{L^\infty_{|\xi|\gg 1}L^2_{\xi_1}} \|f(\xi_1)g(\xi-\xi_1)\|_{L^2_{\xi,\xi_1}}\lesssim \|f\|_{L^2}\|g\|_{L^2}
		\end{align*}
		Thus, \eqref{bilinearkdv} holds for $s = -1$. By interpolation, \eqref{bilinearkdv} holds for $s \geq -1$. We finish the proof of \eqref{kdvvv}.

		Similar to the argument for \eqref{schkdvtosh}-\eqref{kdvvv}, \eqref{sstokdv} is equivalent to
		\begin{align}\label{sstokdvv}
			\left\|\langle\xi\rangle^{s_2}\int_{\mathbb{R}}\frac{1-e^{i\xi(\xi-2\xi_1-\xi^2)}}{\xi-2\xi_1-\xi^2}\frac{f(\xi_1)g(\xi-\xi_1)}{\langle\xi_1\rangle^{s_1}\langle\xi-\xi_1\rangle^{s_1}}~d\xi_1\right\|_{L^2}\lesssim \|f\|_{L^2}\|g\|_{L^2}.
		\end{align}
		\textbf{Necessary part of \eqref{sstokdv}.} Let $1\ll M\ll N$, $f = \chi_{[N,N+2M]}$, $g = \chi_{[-N-2M,-N+2M]}$. 
		$\forall$ $\xi\in [M,2M]$, one has
		\begin{align*}
			\left|\langle\xi\rangle^{s_2}\int_{\mathbb{R}}\frac{1-e^{i\xi(\xi-2\xi_1-\xi^2)}}{\xi-2\xi_1-\xi^2}\frac{f(\xi_1)g(\xi-\xi_1)}{\langle\xi_1\rangle^{s_1}\langle\xi-\xi_1\rangle^{s_1}}~d\xi_1\right|\sim M^{s_2+1}N^{-2s_1+1}.
		\end{align*}
		Since $\|f\|_{L^2}\|g\|_{L^2}\sim M$, we obtain $s_1\geq -1/2$.
		
		Let $f = \chi_{[N,N+1]}$, $g = \chi_{[-1,1]}$. For $\xi\in [N,N+1]$, one has
		\begin{align*}
			\left|\langle\xi\rangle^{s_2}\int_{\mathbb{R}}\frac{1-e^{i\xi(\xi-2\xi_1-\xi^2)}}{\xi-2\xi_1-\xi^2}\frac{f(\xi_1)g(\xi-\xi_1)}{\langle\xi_1\rangle^{s_1}\langle\xi-\xi_1\rangle^{s_1}}~d\xi_1\right|\sim N^{s_2-s_1-2}.
		\end{align*}
		Since $\|f\|_{L^2}\|g\|_{L^2}\sim 1$, we obtain $s_2\leq s_1+2$.
		
		Let $f = \chi_{[-N,-N+N^{-1}]}$, $a = 1/2+\sqrt{2N+1/4}$, $0<c\ll 1$. Then, $\forall$ $\xi\in [a-cN^{-1},a+cN^{-1}]$, one has
		$|\xi||\xi-2\xi_1-\xi^2|\ll 1$, $ \forall~\xi_1\in\mathrm{supp}(f)$. 
		Let $g = \chi_{[a+N-N^{-1},a+N+N^{-1}]}$. Then, $\forall~|\xi-a|\leq cN^{-1}$, one has
		$$\left|\langle\xi\rangle^{s_2}\int_{\mathbb{R}}\frac{1-e^{i\xi(\xi-2\xi_1-\xi^2)}}{\xi-2\xi_1-\xi^2}\frac{f(\xi_1)g(\xi-\xi_1)}{\langle\xi_1\rangle^{s_1}\langle\xi-\xi_1\rangle^{s_1}}~d\xi_1\right|\sim N^{\frac{1}{2}+\frac{s_2}{2}-2s_1-1}.$$
		Since $\|f\|_{L^2}\|g\|_{L^2}\sim N^{-1}$, we obtain $1/2+s_2/2-2s_1-1-1/2\leq -1$. Thus, $s_2\leq 4s_1$.
		
		\textbf{Sufficient part of \eqref{sstokdv}.} By interpolation, we only need to show \eqref{sstokdvv} for $(s_1,s_2) = (-1/2,-2)$ and $(s,s+2)$, $\forall~s\geq 2/3$. 
		If $|\xi|\lesssim 1$, by Cauchy-Schwarz inequality one has
		\begin{align*}
			&\quad\left|\langle\xi\rangle^{s_2}\int_{\mathbb{R}}\frac{1-e^{i\xi(\xi-2\xi_1-\xi^2)}}{\xi-2\xi_1-\xi^2}\frac{f(\xi_1)g(\xi-\xi_1)}{\langle\xi_1\rangle^{s_1}\langle\xi-\xi_1\rangle^{s_1}}~d\xi_1\right|\\
			&\lesssim \left|\int_{\mathbb{R}}\frac{\min\{1,|\xi-2\xi_1-\xi^2|\}}{|\xi-2\xi_1-\xi^2|\langle\xi_1\rangle^{-1/2}\langle\xi-\xi_1\rangle^{-1/2}}|f(\xi_1)||g(\xi-\xi_1)|~d\xi_1\right|\\
			&\lesssim \int_{\mathbb{R}}|f(\xi_1)||g(\xi-\xi_1)|~d\xi_1\lesssim \|f\|_{L^2}\|g\|_{L^2}.
		\end{align*}
		Thus,
		\begin{align*}
			\left\|\langle\xi\rangle^{s_2}\int_{\mathbb{R}}\frac{1-e^{i\xi(\xi-2\xi_1-\xi^2)}}{\xi-2\xi_1-\xi^2}\frac{f(\xi_1)g(\xi-\xi_1)}{\langle\xi_1\rangle^{s_1}\langle\xi-\xi_1\rangle^{s_1}}~d\xi_1\right\|_{L^2_{|\xi|\lesssim 1}}\lesssim \|f\|_{L^2}\|g\|_{L^2}.
		\end{align*}
		
		Now, we consider $|\xi|\gg 1$.
		For point $(s_1,s_2) = (-1/2,-2)$, we decompose the integration into two parts.
		
		\noindent{\hypertarget{cs13}{\bf{Case 1.1}.}} $\Omega_1$: $|\xi-2\xi_1-\xi^2|\gtrsim \langle\xi_1\rangle^{1/2}\langle\xi-\xi_1\rangle^{1/2}$.
		
		\noindent{\hypertarget{cs23}{\bf{Case 1.2}.}} $\Omega_2$: $|\xi-2\xi_1-\xi^2|\ll \langle\xi_1\rangle^{1/2}\langle\xi-\xi_1\rangle^{1/2}$.
		
		For \hyperlink{cs13}{\textcolor[RGB]{0,0,0}{{\bf{Case 1.1}}}}, by Cauchy-Schwarz inequality, one has
		\begin{align*}
			&\quad\left\|\int_{\mathbb{R}}\frac{\chi_{\Omega_1}}{\langle\xi\rangle^2}\frac{1-e^{i\xi(\xi-2\xi_1-\xi^2)}}{\xi-2\xi_1-\xi^2}\frac{f(\xi_1)g(\xi-\xi_1)}{\langle\xi_1\rangle^{-1/2}\langle\xi-\xi_1\rangle^{-1/2}}~d\xi_1\right\|_{L^2_{|\xi|\gg 1}}\\
			&\lesssim \|\langle\xi\rangle^{-2}\|_{L^2_{\xi}}\|f\|_{L^2}\|g\|_{L^2}\lesssim \|f\|_{L^2}\|g\|_{L^2}.
		\end{align*}			
		For \hyperlink{cs23}{\textcolor[RGB]{0,0,0}{{\bf{Case 1.2}}}}, one has $|\xi-\xi_1|\sim|\xi_1|\sim |\xi|^2$. By Cauchy-Schwarz inequality one has
		\begin{align*}
			&\quad\left|\int_{\mathbb{R}}\chi_{|\xi|\gg 1}\int_{\mathbb{R}}\frac{\chi_{\Omega_2}}{\langle\xi\rangle^{2}}\frac{1-e^{i\xi(\xi-2\xi_1-\xi^2)}}{\xi-2\xi_1-\xi^2}\frac{f(\xi_1)g(\xi-\xi_1)}{\langle\xi_1\rangle^{-1/2}\langle\xi-\xi_1\rangle^{-1/2}}~d\xi_1h(\xi)~d\xi\right|\\
			&\lesssim \int_{\mathbb{R}}|f(\xi_1)| \|g(\xi-\xi_1)h(\xi)\|_{L^2_{\xi}}\left\|\frac{\chi_{1\ll |\xi|^2\sim |\xi_1|}\min\{1,|\xi(\xi-2\xi_1-\xi^2)|\}}{|\xi-2\xi_1-\xi^2|}\right\|_{L^2_\xi}\\
			&\lesssim \|f\|_{L^2}\|g\|_{L^2}\|h\|_{L^2}\left\|\frac{\chi_{1\ll |\xi|^2\sim |\xi_1|}\min\{1,||\xi_1|^{\frac{1}{2}}(\xi-2\xi_1-\xi^2)|\}}{|\xi-2\xi_1-\xi^2|}\right\|_{L^\infty_{\xi_1}L^2_\xi}
		\end{align*}
		Let $y(\xi) = \xi-2\xi_1-\xi^2$. For $1\ll |\xi|^2\sim |\xi_1|$, one has $|dy(\xi)|\sim |(1-2\xi)d\xi|\sim |\xi_1|^{1/2}|d\xi|$. Thus,
		\begin{align*}
			\left\|\frac{\chi_{1\ll |\xi|^2\sim |\xi_1|}\min\{1,||\xi_1|^{\frac{1}{2}}(\xi-2\xi_1-\xi^2)|\}}{|\xi-2\xi_1-\xi^2|}\right\|_{L^2_\xi}^2&\lesssim \int_{\mathbb{R}}\frac{\min\{1,|\xi_1|y^2\}}{y^2}~\frac{dy}{|\xi_1|^{1/2}}\\
			&\lesssim 1.
		\end{align*}
		By duality one has
		\begin{align*}
			\left\|\int_{\mathbb{R}}\frac{\chi_{\Omega_2}}{\langle\xi\rangle^{2}}\frac{1-e^{i\xi(\xi-2\xi_1-\xi^2)}}{\xi-2\xi_1-\xi^2}\frac{f(\xi_1)g(\xi-\xi_1)}{\langle\xi_1\rangle^{-1/2}\langle\xi-\xi_1\rangle^{-1/2}}~d\xi_1\right\|_{L^2_{|\xi|\gg 1}}\lesssim \|f\|_{L^2}\|g\|_{L^2}.
		\end{align*}
		For points $(s,s+2), s\geq 2/3$, we decompose the integration into two parts.
		
		\noindent{\hypertarget{cs14}{\bf{Case 2.1}.}} $\Omega_1$: $|\xi-2\xi_1-\xi^2|\gtrsim |\xi|^2$.
		
		\noindent{\hypertarget{cs24}{\bf{Case 2.2}.}} $\Omega_2$: $|\xi-2\xi_1-\xi^2|\ll |\xi|^2$.
		
		For \hyperlink{cs13}{\textcolor[RGB]{0,0,0}{{\bf{Case 2.1}}}}, by Cauchy-Schwarz inequality, for $s\geq 2/3$, one has
		\begin{align*}
			&\quad\left\|\langle\xi\rangle^{s+2}\int_{\mathbb{R}}\chi_{\Omega_1}\frac{1-e^{i\xi(\xi-2\xi_1-\xi^2)}}{\xi-2\xi_1-\xi^2}\frac{f(\xi_1)g(\xi-\xi_1)}{\langle\xi_1\rangle^{s}\langle\xi-\xi_1\rangle^{s}}~d\xi_1\right\|_{L^2_{|\xi|\gg 1}}\\
			&\lesssim \left\|\langle\xi\rangle^{s}\int_{\mathbb{R}}\frac{|f(\xi_1)||g(\xi-\xi_1)|}{\langle\xi_1\rangle^{s}\langle\xi-\xi_1\rangle^{s}}~d\xi_1\right\|_{L^2}\\ &\lesssim \left\|\int_{\mathbb{R}}(\langle\xi_1\rangle^{-s}+\langle\xi-\xi_1\rangle^{-s})|f(\xi_1)||g(\xi-\xi_1)|~d\xi_1\right\|_{L^2}\\
			&\lesssim \|\langle\xi\rangle^{-s}\|_{L^2_{\xi}}\|f\|_{L^2}\|g\|_{L^2}\lesssim \|f\|_{L^2}\|g\|_{L^2}.
		\end{align*}
		
		For \hyperlink{cs23}{\textcolor[RGB]{0,0,0}{{\bf{Case 2.2}}}}, recall that we always assume $|\xi|\gg 1$. Thus, $|\xi-\xi_1|\sim|\xi_1|\sim |\xi|^2$. Then,
		$$\frac{\langle\xi\rangle^{s+2}}{\langle\xi_1\rangle^{s}\langle\xi-\xi_1\rangle^{s}}\sim \langle\xi\rangle^{2-3s}\lesssim 1,\quad \forall~s\geq \frac{2}{3}.$$			
		By the same argument for \hyperlink{cs23}{\textcolor[RGB]{0,0,0}{{\bf{Case 1.2}}}}, we obtain the control for this case.			
		Combining all these cases, we finish the proof of this lemma.
	\end{proof}		
	By \eqref{schkdvtosh}--\eqref{sstokdv} and Theorems 1.1, 1.4 in \cite{kenig2001ill}, one has
	Theorem \ref{illpore}.
	
	\section{Multi-linear estimates in Bourgain space}\label{multiinbo}
	In this section, our main result is:
	\begin{prop}\label{conclusion}
		If there exists $1/2<b_1<\tilde{b}_1$, $1/2<b_2<\tilde{b}_2$ such that
		\begin{align}
			\|u\bar{v}w\|_{X^{s_1,\tilde{b}_1-1}}&\lesssim \|u\|_{X^{s_1,b_1}}\|v\|_{X^{s_1,b_1}}\|w\|_{X^{s_1,b_1}},\label{mainxsbre}\\
			\|uv\|_{X^{s_1,\tilde{b}_1-1}}&\lesssim \|u\|_{X^{s_1,b_1}}\|v\|_{Y^{s_2,b_2}},\label{mainsks}\\
			\|\partial_x(u\bar{v})\|_{Y^{s_2,\tilde{b}_2-1}}&\lesssim \|u\|_{X^{s_1,b_1}}\|v\|_{X^{s_1,b_1}},\label{mainssk}\\
			\|\partial_x(uv)\|_{Y^{s_2,\tilde{b}_2-1}}&\lesssim \|u\|_{Y^{s_2,b_2}}\|v\|_{Y^{s_2,b_2}}\label{mainkkk}
		\end{align}
		then $s_1\geq 0$, $\max\{-3/4,s_1/2-11/8,s_1-5/2\}<s_2<\min\{4s_1,s_1+1\}$. Conversely, if $s_1\geq 0$, $\max\{-3/4,s_1/2-11/8,s_1-5/2\}<s_2<\min\{4s_1,s_1+1\}$, there exists $1/2<b_1<\tilde{b}_1$, $1/2<b_2<\tilde{b}_2$ such that \eqref{mainxsbre}--\eqref{mainkkk} hold.
	\end{prop}
	We will show \eqref{mainsks}--\eqref{mainkkk} in Lemmas \ref{sktosb}--\ref{purekdv} respectively.
	\begin{lemma}\label{sktosb}
		If there exists $b_1>1/2$, $b_2\in\mathbb{R}$ such that 
		\begin{align}\label{counter}
			\|uv\|_{X^{s_1,b_1-1}}\lesssim \|u\|_{X^{s_1,b_1}}\|v\|_{Y^{s_2,b_2}},
		\end{align}
		then $s_2> s_1-5/2$. 
		
		Let $s_1\geq 0$, $s_2\geq -1$. If there exists $\epsilon_0(s_1,s_2)>0$ such that $\forall~0<\varepsilon<\epsilon_0$, 
		\begin{align}\label{sktosbc}
			\|uv\|_{X^{s_1,-1/2+\varepsilon}}\lesssim \|u\|_{X^{s_1,1/2+\varepsilon}}\|v\|_{Y^{s_2,1/2+\varepsilon}},
		\end{align}
		then $s_2\geq s_1-2$. Conversely, if $s_2\geq s_1-2$, then there exists $\varepsilon_0(s_1,s_2)>0$ such that  $\forall~0<\varepsilon<\varepsilon_0$, one has
		\begin{align}\label{sktosbb}
			\|uv\|_{X^{s_1,-1/2+2\varepsilon}}\lesssim \|u\|_{X^{s_1,1/2+\varepsilon}}\|v\|_{Y^{s_2,1/2+\varepsilon}}.
		\end{align}
	\end{lemma}
	\begin{proof}[\textbf{Proof}]
		By the definition of $X^{s,b},Y^{s,b}$ and duality, \eqref{counter} is equivalent to
		\begin{equation}\label{countereq}
			\begin{aligned}
				&\quad\int_{\substack{\xi_1+\xi_2 = \xi,\\\tau_1+\tau_2 = \tau}}\frac{\langle\xi\rangle^{s_1}f(\tau_1,\xi_1)g(\tau_2,\xi_2)h(\tau,\xi)}{\langle\xi_1\rangle^{s_1}\langle\xi_2\rangle^{s_2}\langle\tau_1+\xi_1^2\rangle^{b_1}\langle\tau_2-\xi_2^3\rangle^{b_2}\langle\tau+\xi^2\rangle^{1-b_1}}\\
				&\lesssim \|f\|_{L^2}\|g\|_{L^2}\|h\|_{L^2},\quad \forall~f,g,h\geq 0.
			\end{aligned}
		\end{equation}	
		Let $f(\tau,\xi) =\chi_{[-2,2]^2}(\tau,\xi)$, $g(\tau,\xi) = \chi_{[N,N+1]}(\xi)\chi_{[0,1]}(\tau-\xi^3)$, and $h(\tau,\xi) =\chi_{[N^3-10N^2,N^3+10N^2]}(\tau) \chi_{[N-10,N+10]}(\xi)$. Then, $\|f\|_{L^2}\|g\|_{L^2}\|h\|_{L^2}\sim N$.
		\begin{align*}
			&\quad\int_{\substack{\xi_1+\xi_2 = \xi,\\\tau_1+\tau_2 = \tau}}\frac{\langle\xi\rangle^{s_1}f(\tau_1,\xi_1)g(\tau_2,\xi_2)h(\tau,\xi)}{\langle\xi_1\rangle^{s_1}\langle\xi_2\rangle^{s_2}\langle\tau_1+\xi_1^2\rangle^{b_1}\langle\tau_2-\xi_2^3\rangle^{b_2}\langle\tau+\xi^2\rangle^{1-b_1}}\\
			&\gtrsim \frac{1}{N^{s_2-s_1+3(1-b)}}\int_{\mathbb{R}^4}f(\tau_1,\xi_1)g(\tau_2,\xi_2)~d\tau_1d\tau_2 d\xi_1 d\xi_2 \sim N^{-s_2+s_1-3+3b}.
		\end{align*}
		Thus, $-s_2+s_1-3+3b_1\leq 1$ which means $s_2> s_1-5/2$.
		
		For \eqref{sktosbc}, it is equivalent to 
		\begin{equation}\label{dualsktosd}
			\begin{aligned}
				&\quad\int_{\substack{\xi_1+\xi_2 = \xi,\\\tau_1+\tau_2 = \tau}}\frac{\langle\xi\rangle^{s_1}f(\tau_1,\xi_1)g(\tau_2,\xi_2)h(\tau,\xi)}{\langle\xi_1\rangle^{s_1}\langle\xi_2\rangle^{s_2}\langle\tau_1+\xi_1^2\rangle^{1/2+\varepsilon}\langle\tau_2-\xi_2^3\rangle^{1/2+\varepsilon}\langle\tau+\xi^2\rangle^{1/2-\varepsilon}}\\
				&\lesssim \|f\|_{L^2}\|g\|_{L^2}\|h\|_{L^2},\quad \forall~f,g,h\geq 0.
			\end{aligned}
		\end{equation}
		Let $f(\tau,\xi) = \chi_{[-10,10]^2}(\tau,\xi)$, $h(\tau,\xi) = \chi_{[N-10,N+10]}(\xi)\chi_{[-100,100]}(\tau+\xi^2)$,
		\begin{align*}
			g(\tau,\xi) = \int_{\mathbb{R}^2}f(\tau_1,\xi_1)h(\tau_1+\tau,\xi_1+\xi)~d\tau_1d\xi_1.
		\end{align*}
		$\forall~\xi \in [N-1,N+1]$, $|\tau+\xi^2|\leq N$, one has $g(\tau,\xi)\gtrsim N^{-1}$. Thus, $\|g\|_{L^2}\gtrsim N^{-1/2}$.
		\begin{equation}
			\begin{aligned}
				&\quad\int_{\substack{\xi_1+\xi_2 = \xi,\\\tau_1+\tau_2 = \tau}}\frac{\langle\xi\rangle^{s_1}f(\tau_1,\xi_1)g(\tau_2,\xi_2)h(\tau,\xi)}{\langle\xi_1\rangle^{s_1}\langle\xi_2\rangle^{s_2}\langle\tau_1+\xi_1^2\rangle^{1/2+\varepsilon}\langle\tau_2-\xi_2^3\rangle^{1/2+\varepsilon}\langle\tau+\xi^2\rangle^{1/2-\varepsilon}}\\
				&\gtrsim N^{s_1-s_2-3/2-3\varepsilon} \int_{\mathbb{R}^4}f(\tau_1,\xi_2)h(\tau_1+\tau_2,\xi_1+\xi_2)g(\tau_2,\xi_2)~d\tau_1 d\tau_2 d\xi_1 d\xi_2\\
				&\sim  N^{s_1-s_2-3/2-3\varepsilon}\|g\|_{L^2}^2 \\
				&\gtrsim N^{s_1-s_2-2-3\varepsilon}\|g\|_{L^2}\sim N^{s_1-s_2-2-3\varepsilon}\|f\|_{L^2}\|g\|_{L^2}\|h\|_{L^2}.
			\end{aligned}
		\end{equation}
		Thus, $s_2\geq s_1-2$.
		
		\eqref{sktosbb} is equivalent to
		\begin{equation}\label{dualsktos}
			\begin{aligned}
				&\quad\int_{\substack{\xi_1+\xi_2 = \xi,\\\tau_1+\tau_2 = \tau}}\frac{\langle\xi\rangle^{s_1}f(\tau_1,\xi_1)g(\tau_2,\xi_2)h(\tau,\xi)}{\langle\xi_1\rangle^{s_1}\langle\xi_2\rangle^{s_2}\langle\tau_1+\xi_1^2\rangle^{1/2+\varepsilon}\langle\tau_2-\xi_2^3\rangle^{1/2+\varepsilon}\langle\tau+\xi^2\rangle^{1/2-2\varepsilon}}\\
				&\lesssim \|f\|_{L^2}\|g\|_{L^2}\|h\|_{L^2},\quad \forall~f,g,h\geq 0.
			\end{aligned}
		\end{equation}			
		By Lemma 3.1 in \cite{corcho2007well}, \eqref{dualsktos} holds for $(s_1,s_2) = (0,-1)$. If $\eqref{dualsktos}$ holds for $(s_1,s_2)$, then it also holds for $(s_1,\tilde{s}_2)$, $\forall~\tilde{s}_2\geq s_2$ and $(s_1+a,s_2+a)$, $\forall~a>0$. By multi-interpolation, we only need to show the inequality for $(s_1,s_2) = (1,-1)$. 
		
		We decompose the integration into two parts.
		
		\noindent{\hypertarget{cs15}{\bf{Case 1}.}} $\Omega_1$: $\langle\xi\rangle\lesssim \langle\xi_1\rangle$.
		
		\noindent{\hypertarget{cs25}{\bf{Case 2}.}} $\Omega_2$: $\langle\xi\rangle\gg \langle\xi_1\rangle$.
		
		For \hyperlink{cs15}{\textcolor[RGB]{0,0,0}{{\bf{Case 1}}}}, the inequality reduces to $(s_1,s_2) = (0,-1)$ which has been shown in \cite{corcho2007well}. For \hyperlink{cs25}{\textcolor[RGB]{0,0,0}{{\bf{Case 2}}}}, one has $|\xi_2|\sim |\xi|\gg \max\{1,|\xi_1|\}$. Then,
		\begin{align*}
			\langle\tau_1+\xi_1^2\rangle+\langle\tau_2-\xi_2^3\rangle+\langle\tau+\xi^2\rangle\gtrsim |\xi_1^2-\xi_2^3-\xi^2|\sim |\xi|^3.
		\end{align*}
		If $\langle\tau+\xi^2\rangle\gtrsim |\xi|^3$, one has
		\begin{equation*}
			\begin{aligned}
				&\quad\int_{\substack{\xi_1+\xi_2 = \xi,\\\tau_1+\tau_2 = \tau}}\frac{\chi_{\langle\xi\rangle\gg \langle\xi_1\rangle, \langle\tau+\xi^2\rangle\gtrsim |\xi|^3}\langle\xi\rangle f(\tau_1,\xi_1)g(\tau_2,\xi_2)h(\tau,\xi)}{\langle\xi_1\rangle\langle\xi_2\rangle^{-1}\langle\tau_1+\xi_1^2\rangle^{1/2+\varepsilon}\langle\tau_2-\xi_2^3\rangle^{1/2+\varepsilon}\langle\tau+\xi^2\rangle^{1/2-2\varepsilon}}\\
				&\lesssim \int_{\substack{\xi_1+\xi_2 = \xi,\\\tau_1+\tau_2 = \tau}}\frac{\chi_{|\xi|\gg \langle\xi_1\rangle}|\xi|^{1/2+6\varepsilon}f(\tau_1,\xi_1)g(\tau_2,\xi_2)h(\tau,\xi)}{\langle\xi_1\rangle\langle\tau_1+\xi_1^2\rangle^{1/2+\varepsilon}\langle\tau_2-\xi_2^3\rangle^{1/2+\varepsilon}}\\
				&\lesssim \left\|\frac{\chi_{|\xi|\gg \langle\xi_1\rangle}|\xi|^{1/2+6\varepsilon}}{\langle\xi_1\rangle\langle\tau_1+\xi_1^2\rangle^{1/2+\varepsilon}\langle\tau-\tau_1-(\xi-\xi_1)^3\rangle^{1/2+\varepsilon}}\right\|_{L^\infty_{\tau,\xi}L^2_{\tau_1,\xi_1}}\|f\|_{L^2}\|g\|_{L^2}\|h\|_{L^2}\\
				&\lesssim \left\|\frac{\chi_{|\xi|\gg \langle\xi_1\rangle}|\xi|^{1/2+6\varepsilon}}{\langle\tau+\xi_1^2-(\xi-\xi_1)^3\rangle^{1/2+\varepsilon}}\right\|_{L^\infty_{\tau,\xi}L^2_{\xi_1}}\|f\|_{L^2}\|g\|_{L^2}\|h\|_{L^2}\\
				&\lesssim \left\|\frac{\chi_{|\xi|\gg \langle\xi_1\rangle}|\xi|^{1/2+6\varepsilon}}{\langle y\rangle^{1/2+\varepsilon}|\xi|}\right\|_{L^\infty_{\xi}L^2_{y}}\|f\|_{L^2}\|g\|_{L^2}\|h\|_{L^2}\lesssim \|f\|_{L^2}\|g\|_{L^2}\|h\|_{L^2}.
			\end{aligned}
		\end{equation*}		
		
		If $\langle\tau_1+\xi_1^2\rangle\gtrsim \max\{\langle\tau+\xi^2\rangle,|\xi|^3\}$, then 
		$$\langle\tau_1+\xi_1^2\rangle^{1/2+\varepsilon}\langle\tau+\xi^2\rangle^{1/2-2\varepsilon}\gtrsim \langle\tau_1+\xi_1^2\rangle^{1/2-2\varepsilon}\langle\tau+\xi^2\rangle^{1/2+\varepsilon}.$$
		By the symmetry between $\xi_1$ and $\xi$, one has the result.
		
		If $\langle\tau_2-\xi_2^3\rangle\gtrsim |\xi|^3$, $\langle\tau+\xi^2\rangle\gtrsim |\xi|$, one has
		\begin{equation*}
			\begin{aligned}
				&\quad\int_{\substack{\xi_1+\xi_2 = \xi,\\\tau_1+\tau_2 = \tau}}\frac{\chi_{\langle\xi\rangle\gg \langle\xi_1\rangle, \langle\tau+\xi^2\rangle\gtrsim |\xi|, \langle\tau_2-\xi_2^3\rangle\gtrsim |\xi|^3}\langle\xi\rangle f(\tau_1,\xi_1)g(\tau_2,\xi_2)h(\tau,\xi)}{\langle\xi_1\rangle\langle\xi_2\rangle^{-1}\langle\tau_1+\xi_1^2\rangle^{1/2+\varepsilon}\langle\tau_2-\xi_2^3\rangle^{1/2+\varepsilon}\langle\tau+\xi^2\rangle^{1/2-2\varepsilon}}\\
				&\lesssim \int_{\substack{\xi_1+\xi_2 = \xi,\\\tau_1+\tau_2 = \tau}}\frac{\langle\xi\rangle^{-\varepsilon}f(\tau_1,\xi_1)g(\tau_2,\xi_2)h(\tau,\xi)}{\langle\xi_1\rangle\langle\tau_1+\xi_1^2\rangle^{1/2+\varepsilon}}\\
				&\lesssim \int_{\mathbb{R}^2}\frac{f(\tau_1,\xi_1)}{\langle\xi_1\rangle\langle\tau_1+\xi_1^2\rangle^{1/2+\varepsilon}}~d\tau_1 d\xi_1\|g\|_{L^2}\|h\|_{L^2}\\
				&\lesssim \|f\|_{L^2}\|g\|_{L^2}\|h\|_{L^2}.
			\end{aligned}
		\end{equation*}	
		
		If $\langle\tau_2-\xi_2^3\rangle\gtrsim |\xi|^3$, $\langle\tau+\xi^2\rangle\ll |\xi|$, one has
		\begin{equation*}
			\begin{aligned}
				&\quad\int_{\substack{\xi_1+\xi_2 = \xi,\\\tau_1+\tau_2 = \tau}}\frac{\chi_{\langle\xi\rangle\gg \langle\xi_1\rangle, \langle\tau+\xi^2\rangle\ll |\xi|, \langle\tau_2-\xi_2^3\rangle\gtrsim |\xi|^3}\langle\xi\rangle f(\tau_1,\xi_1)g(\tau_2,\xi_2)h(\tau,\xi)}{\langle\xi_1\rangle\langle\xi_2\rangle^{-1}\langle\tau_1+\xi_1^2\rangle^{1/2+\varepsilon}\langle\tau_2-\xi_2^3\rangle^{1/2+\varepsilon}\langle\tau+\xi^2\rangle^{1/2-2\varepsilon}}\\
				&\lesssim \int_{\substack{\xi_1+\xi_2 = \xi,\\\tau_1+\tau_2 = \tau}}\frac{\chi_{|\xi|\gg \langle\xi_1\rangle}|\xi|^{1/2}f(\tau_1,\xi_1)g(\tau_2,\xi_2)h(\tau,\xi)}{\langle\xi_1\rangle\langle\tau_1+\xi_1^2\rangle^{1/2+\varepsilon}\langle\tau+\xi^2\rangle^{1/2+\varepsilon}}\\
				&\lesssim \left\|\frac{\chi_{|\xi|\sim|\xi_2|\gg 1}|\xi_2|^{1/2}}{\langle\tau-\tau_2+(\xi-\xi_2)^2\rangle^{1/2+\varepsilon}\langle\tau+\xi^2\rangle^{1/2+\varepsilon}}\right\|_{L^\infty_{\tau_2,\xi_2}L^2_{\xi ,\tau}} \|f\|_{L^2}\|g\|_{L^2}\|h\|_{L^2}\\
				&\lesssim \left\|\frac{\chi_{|\xi|\sim|\xi_2|\gg 1}|\xi_2|^{1/2}}{\langle-\xi^2-\tau_2+(\xi-\xi_2)^2\rangle^{1/2+\varepsilon}}\right\|_{L^\infty_{\tau_2,\xi_2}L^2_{\xi}} \|f\|_{L^2}\|g\|_{L^2}\|h\|_{L^2}\\
				&\lesssim \left\|\frac{\chi_{|\xi_2|\gg 1}|\xi_2|^{1/2}}{|\xi_2|^{1/2}\langle y\rangle^{1/2+\varepsilon}}\right\|_{L^\infty_{\xi_2}L^2_{y}} \|f\|_{L^2}\|g\|_{L^2}\|h\|_{L^2}\lesssim \|f\|_{L^2}\|g\|_{L^2}\|h\|_{L^2}.
			\end{aligned}
		\end{equation*}	
		We finish the proof of this lemma.
	\end{proof}
	
	\begin{lemma}\label{sstkdv}
		Let $s_1\geq 0$. If there exists $b_2>1/2$, $b_1\in\mathbb{R}$ such that
		\begin{align}\label{conutsstokdv}
			\|\partial_x(u\bar{v})\|_{Y^{s_2,b_2-1}}\lesssim \|u\|_{X^{s_1,b_1}}\|v\|_{X^{s_1,b_1}}
		\end{align}
		then $s_2<\min\{4s_1,s_1+1\}$. Conversely, if $s_2<\min\{4s_1,s_1+1\}$, there exists $\varepsilon_0(s_1,s_2)>0$ such that  $\forall~0<\varepsilon<\varepsilon_0$, one has
		\begin{align}\label{sstokdvbou}
			\|\partial_x(u\bar{v})\|_{Y^{s_2,-1/2+2\varepsilon}}\lesssim \|u\|_{X^{s_1,1/2+\varepsilon}}\|v\|_{X^{s_1,1/2+\varepsilon}}.
		\end{align}
	\end{lemma}
	\begin{proof}[\textbf{Proof}]
		By the definition of $X^{s,b},Y^{s,b}$ and duality, \eqref{conutsstokdv} is equivalent to
		\begin{equation*}
			\begin{aligned}
				&\quad\int_{\substack{\xi_1+\xi_2 = \xi,\\\tau_1+\tau_2 = \tau}}\frac{|\xi|\langle\xi\rangle^{s_2}f(\tau_1,\xi_1)g(\tau_2,\xi_2)h(\tau,\xi)}{\langle\xi_1\rangle^{s_1}\langle\xi_2\rangle^{s_1}\langle\tau_1+\xi_1^2\rangle^{b_1}\langle\tau_2-\xi_2^2\rangle^{b_1}\langle\tau-\xi^3\rangle^{1-b_2}}\\
				&\lesssim \|f\|_{L^2}\|g\|_{L^2}\|h\|_{L^2},\quad \forall~f,g,h\geq 0.
			\end{aligned}
		\end{equation*}	
		Let $f(\tau,\xi)=\chi_{[-10,10]^2}(\tau,\xi)$, $g(\tau,\xi) = \chi_{[N-10,N+10]}(\xi)\chi_{[-100,100]}(\tau-\xi^2)$, $h = f*g$. Then,
		\begin{align*}
			&\quad\int_{\substack{\xi_1+\xi_2 = \xi,\\\tau_1+\tau_2 = \tau}}\frac{|\xi|\langle\xi\rangle^{s_2}f(\tau_1,\xi_1)g(\tau_2,\xi_2)h(\tau,\xi)}{\langle\xi_1\rangle^{s_1}\langle\xi_2\rangle^{s_1}\langle\tau_1+\xi_1^2\rangle^{b_1}\langle\tau_2-\xi_2^2\rangle^{b_1}\langle\tau-\xi^3\rangle^{1-b_2}}\\
			&\sim N^{1+s_2-s_1-3(1-b_2)}\int_{\substack{\xi_1+\xi_2 = \xi,\\\tau_1+\tau_2 = \tau}}f(\tau_1,\xi_1)g(\tau_2,\xi_2)h(\tau,\xi)\\
			&\sim N^{s_2-s_1-2+3b_2}\|h\|_{L^2}^2.
		\end{align*}
		Since $\|h\|_{L^2}\gtrsim N^{-1/2}$ and  $\|f\|_{L^2}\sim \|g\|_{L^2}\sim 1$, we obtain $$N^{s_2-s_1-2+3b_2}\|h\|_{L^2}^2\gtrsim N^{s_2-s_1+3b_2-5/2}\|f\|_{L^2}\|g\|_{L^2}\|h\|_{L^2}.$$
		Thus, $s_2<s_1+1$.
		
		Let $f(\tau,\xi) = \chi_{[(-N^2+N)/2-10N^{-1},(-N^2+N)/2+10N^{-1}]}(\xi)\chi_{[-10,10]}(\tau+\xi^2)$, $g(\tau,\xi) = \chi_{[-100,100]}(\tau-\xi^2)\chi_{[(N^2+N)/2-10N^{-1},(N^2+N)/2+10N^{-1}]}(\xi)$, $h = f*g$. For $|\xi-N|\leq N^{-1}$, $|\tau-\xi^2-N^2\xi|\leq 1$, one has
		\begin{align*}
			h(\tau,\xi)\gtrsim N^{-1}.
		\end{align*}			
		Thus, $\|h\|_{L^2}\gtrsim N^{-3/2}$. On the support of $h(\tau,\xi)$, one has$\langle\tau-\xi^3\rangle\lesssim N$.
		\begin{align*}
			&\quad\int_{\substack{\xi_1+\xi_2 = \xi,\\\tau_1+\tau_2 = \tau}}\frac{|\xi|\langle\xi\rangle^{s_2}f(\tau_1,\xi_1)g(\tau_2,\xi_2)h(\tau,\xi)}{\langle\xi_1\rangle^{s_1}\langle\xi_2\rangle^{s_1}\langle\tau_1+\xi_1^2\rangle^{b_1}\langle\tau_2-\xi_2^2\rangle^{b_1}\langle\tau-\xi^3\rangle^{1-b_2}}\\
			&\gtrsim N^{1+s_2-4s_1-(1-b_2)}\int_{\substack{\xi_1+\xi_2 = \xi,\\\tau_1+\tau_2 = \tau}}f(\tau_1,\xi_1)g(\tau_2,\xi_2)h(\tau,\xi)\\
			&\sim N^{s_2-4s_1+b_2}\|h\|_{L^2}^2
			\\
			&\gtrsim N^{s_2-4s_1-3/2+b_2}\|h\|_{L^2}\sim N^{s_2-4s_1-1/2+b_2}\|f\|_{L^2}\|g\|_{L^2}\|h\|_{L^2}.
		\end{align*}
		Thus, $s_2<4s_1$ since $b_2>1/2$.
		
		If $\eqref{sstokdvbou}$ holds for $(s_1,s_2)$, then it also holds for $(s_1,\tilde{s}_2)$, $\forall~\tilde{s}_2\leq s_2$ and $(s_1+a,s_2+a)$, $\forall~a>0$. By multi-interpolation, we only need to show the inequality for $(s_1,s_2) = (0,-\delta)$ and $(1/3, 4/3-\delta)$ where $0<\delta\ll 1$. 
		
		For $(s_1,s_2) = (0,-\delta)$, \eqref{sstokdvbou} is equivalent to 
		\begin{align*}
			&\quad\int_{\substack{\xi_1+\xi_2 = \xi,\\\tau_1+\tau_2 = \tau}}\frac{|\xi|f(\tau_1,\xi_1)g(\tau_2,\xi_2)h(\tau,\xi)}{\langle\xi\rangle^{\delta}\langle\tau_1+\xi_1^2\rangle^{1/2+\varepsilon}\langle\tau_2-\xi_2^2\rangle^{1/2+\varepsilon}\langle\tau-\xi^3\rangle^{1/2-2\varepsilon}}\\
			&\lesssim \|f\|_{L^2}\|g\|_{L^2}\|h\|_{L^2},\quad \forall~f,g,h\geq 0.
		\end{align*}
		For $|\xi|\lesssim 1$, one has
		\begin{align*}
			&\quad\int_{\substack{\xi_1+\xi_2 = \xi,\\\tau_1+\tau_2 = \tau}}\frac{\chi_{|\xi|\lesssim 1}|\xi|f(\tau_1,\xi_1)g(\tau_2,\xi_2)h(\tau,\xi)}{\langle\xi\rangle^{\delta}\langle\tau_1+\xi_1^2\rangle^{1/2+\varepsilon}\langle\tau_2-\xi_2^2\rangle^{1/2+\varepsilon}\langle\tau-\xi^3\rangle^{1/2-2\varepsilon}}\\
			&\lesssim \int_{\substack{\xi_1+\xi_2 = \xi,\\\tau_1+\tau_2 = \tau}}\frac{f(\tau_1,\xi_1)g(\tau_2,\xi_2)h(\tau,\xi)}{\langle\tau_1+\xi_1^2\rangle^{1/2+\varepsilon}\langle\tau_2-\xi_2^2\rangle^{1/2+\varepsilon}}\\
			&\lesssim \|\mathscr{F}^{-1}(f(\tau,\xi)/\langle\tau+\xi^2\rangle^{1/2+\varepsilon})\|_{L^4_{t,x}}\|\mathscr{F}^{-1}(g(\tau,\xi)/\langle\tau-\xi^2\rangle^{1/2+\varepsilon})\|_{L^4_{t,x}}\|h\|_{L^2}\\
			&\lesssim \|f\|_{L^2}\|g\|_{L^2}\|h\|_{L^2}.
		\end{align*}
		If $|\xi|\gg 1$ and $\max\{\langle\tau_1+\xi_1^2\rangle,\langle\tau_2-\xi_2^2\rangle,\langle\tau-\xi^3\rangle\}\gtrsim |\xi|^2$,
		then for $\varepsilon\leq \delta/4$, we have
		\begin{align*}
			&\frac{|\xi|}{\langle\xi\rangle^{\delta}\langle\tau_1+\xi_1^2\rangle^{1/2+\varepsilon}\langle\tau_2-\xi_2^2\rangle^{1/2+\varepsilon}\langle\tau-\xi^3\rangle^{1/2-2\varepsilon}}\\
			&\lesssim \frac{\chi_{|\xi|\gg 1}(\langle\tau_1+\xi_1^2\rangle^{1/2+\varepsilon}+\langle\tau_2-\xi_2^2\rangle^{\frac{1}{2}+\varepsilon}+\langle\tau-\xi^3\rangle^{1/2+\varepsilon})}{\langle\tau_1+\xi_1^2\rangle^{1/2+\varepsilon}\langle\tau_2-\xi_2^2\rangle^{1/2+\varepsilon}\langle\tau-\xi^3\rangle^{1/2+\varepsilon}}
		\end{align*}
		By Strichartz estimate, we obtain the result.
		
		If $|\xi|\gg 1$ and $\max\{\langle\tau_1+\xi_1^2\rangle,\langle\tau_2-\xi_2^2\rangle,\langle\tau-\xi^3\rangle\}\ll |\xi|^2$, one has
		$$|\xi||\xi_1-\xi_2+\xi^2| = |\xi_1^2-\xi_2^2+\xi^3|\lesssim \max\{\langle\tau_1+\xi_1^2\rangle,\langle\tau_2-\xi_2^2\rangle,\langle\tau-\xi^3\rangle\}\ll |\xi|^2.$$
		Thus, $|\xi_1-\xi_2+\xi^2|\ll |\xi|$ which means $|\xi_1|\sim |\xi_2|\sim |\xi|^2$. For $\varepsilon\leq \delta/6$, one has
		\begin{align*}
			&\quad\int_{\substack{\xi_1+\xi_2 = \xi,\\\tau_1+\tau_2 = \tau}}\frac{\chi_{|\xi|\gg 1,\max\{\langle\tau_1+\xi_1^2\rangle,\langle\tau_2-\xi_2^2\rangle,\langle\tau-\xi^3\rangle\}\ll |\xi|^2 }|\xi|f(\tau_1,\xi_1)g(\tau_2,\xi_2)h(\tau,\xi)}{\langle\xi\rangle^{\delta}\langle\tau_1+\xi_1^2\rangle^{1/2+\varepsilon}\langle\tau_2-\xi_2^2\rangle^{1/2+\varepsilon}\langle\tau-\xi^3\rangle^{1/2-2\varepsilon}}\\
			&\lesssim \int_{\substack{\xi_1+\xi_2 = \xi,\\\tau_1+\tau_2 = \tau}}\frac{\chi_{\max\{1,|\xi_1-\xi_2+\xi^2|\}\ll |\xi|}|\xi|^{1-\delta+6\varepsilon}f(\tau_1,\xi_1)g(\tau_2,\xi_2)h(\tau,\xi)}{\langle\tau_1+\xi_1^2\rangle^{1/2+\varepsilon}\langle\tau_2-\xi_2^2\rangle^{1/2+\varepsilon}\langle\tau-\xi^3\rangle^{1/2+\varepsilon}}\\
			&\lesssim \int_{\mathbb{R}^2}f(\tau_1,\xi_1) \|g(\tau-\tau_1,\xi-\xi_1)h(\tau,\xi)\|_{L^2_{\tau,\xi}}~d\tau_1 d\xi_1\\
			&\quad \cdot \left\|\frac{|\xi_1|^{(1-\delta+6\varepsilon)/2} \chi_{\max\{1,|2\xi_1-\xi+\xi^2|\}\ll |\xi|}}{\langle\tau-\tau_1-(\xi-\xi_1)^2\rangle^{1/2+\varepsilon}\langle\tau-\xi^3\rangle^{1/2+\varepsilon}}\right\|_{L^\infty_{\tau_1,\xi_1}L^2_{\tau,\xi}}\\
			&\lesssim \left\|\frac{|\xi_1|^{(1-\delta+6\varepsilon)/2}\chi_{\max\{1,|2\xi_1-\xi+\xi^2|\}\ll |\xi|}}{\langle\xi^3-\tau_1-(\xi-\xi_1)^2\rangle^{1/2+\varepsilon}}\right\|_{L^\infty_{\tau_1,\xi_1}L^2_{\xi}}\|f\|_{L^2}\|g\|_{L^2}\|h\|_{L^2} \\
			&\lesssim \left\|\frac{|\xi_1|^{(1-\delta+6\varepsilon)/2}\chi_{|\xi_1|\gg 1}}{\langle y\rangle^{1/2+\varepsilon}|\xi_1|^{1/2}}\right\|_{L^\infty_{\xi_1}L^2_{y}}\|f\|_{L^2}\|g\|_{L^2}\|h\|_{L^2} \lesssim \|f\|_{L^2}\|g\|_{L^2}\|h\|_{L^2}.
		\end{align*}
		
		For $(s_1,s_2) = (1/3,4/3-\delta)$, if $\langle\xi\rangle^4\lesssim \langle\xi_1\rangle\langle\xi_2\rangle$, one can reduce this case to $(s_1,s_2) = (0,-\delta)$. Thus, we only consider the case $\langle\xi\rangle^4\gg \langle\xi_1\rangle\langle\xi_2\rangle$. Thus, $|\xi|\gg \max\{1,|\xi_1|^{1/2}\}$.
		$$ \max\{\langle\tau_1+\xi_1^2\rangle,\langle\tau_2-\xi_2^2\rangle,\langle\tau-\xi^3\rangle\}\gtrsim |\xi||2\xi_1-\xi+\xi^2|\sim |\xi|^3.$$
		If $\langle\tau-\xi^3\rangle = \max\{\langle\tau_1+\xi_1^2\rangle,\langle\tau_2-\xi_2^2\rangle,\langle\tau-\xi^3\rangle\}$, one has
		\begin{align*}
			&\quad\chi_{\langle\xi\rangle^4\gg \langle\xi_1\rangle\langle\xi_2\rangle}\frac{|\xi|\langle\xi\rangle^{4/3-\delta}}{\langle\xi_1\rangle^{1/3}\langle\xi_2\rangle^{1/3}}\frac{1}{\langle\tau_1+\xi_1^2\rangle^{1/2+\varepsilon}\langle\tau_2-\xi_2^2\rangle^{1/2+\varepsilon}\langle\tau-\xi^3\rangle^{1/2-2\varepsilon}}\\
			&\lesssim \chi_{|\xi|\gg \max\{1,|\xi_1|^{1/2}\}}\frac{|\xi|^{1/2-\delta+6\varepsilon}}{\langle\tau_1+\xi_1^2\rangle^{1/2+\varepsilon}\langle\tau_2-\xi_2^2\rangle^{1/2+\varepsilon}}.
		\end{align*}
		For $\varepsilon\leq \delta/6$, by Cauchy-Schwarz inequality, one can control this part by
		\begin{align*}
			&\quad\left\| \frac{\chi_{|\xi|\gg \max\{1,|\xi_1|^{1/2}\}}|\xi|^{1/2-\delta+6\varepsilon}}{\langle\tau_1+\xi_1^2\rangle^{1/2+\varepsilon}\langle\tau-\tau_1-(\xi-\xi_1)^2\rangle^{1/2+\varepsilon}}\right\|_{L^\infty_{\tau,\xi}L^2_{\xi_1,\tau_1}}\|f\|_{L^2}\|g\|_{L^2}\|h\|_{L^2}\\
			&\lesssim \left\| \frac{\chi_{|\xi|\gg \max\{1,|\xi_1|^{1/2}\}}|\xi|^{1/2-\delta+6\varepsilon}}{\langle\tau+\xi_1^2-(\xi-\xi_1)^2\rangle^{1/2+\varepsilon}}\right\|_{L^\infty_{\tau,\xi}L^2_{\xi_1}}\|f\|_{L^2}\|g\|_{L^2}\|h\|_{L^2} \\
			&\lesssim \left\| \frac{\chi_{|\xi|\gg1 }|\xi|^{1/2-\delta+6\varepsilon}}{\langle y\rangle^{1/2+\varepsilon}|\xi|^{1/2}}\right\|_{L^\infty_{\xi}L^2_{y}}\|f\|_{L^2}\|g\|_{L^2}\|h\|_{L^2} \\
			&\lesssim\|f\|_{L^2}\|g\|_{L^2}\|h\|_{L^2} .
		\end{align*}
		If $\langle \tau_1+\xi_1^2\rangle = \max\{\langle\tau_1+\xi_1^2\rangle,\langle\tau_2-\xi_2^2\rangle,\langle\tau-\xi^3\rangle\}$, one has
		\begin{align*}
			&\quad\chi_{\langle\xi\rangle^4\gg \langle\xi_1\rangle\langle\xi_2\rangle}\frac{|\xi|\langle\xi\rangle^{4/3-\delta}}{\langle\xi_1\rangle^{1/3}\langle\xi_2\rangle^{1/3}}\frac{1}{\langle\tau_1+\xi_1^2\rangle^{1/2+\varepsilon}\langle\tau_2-\xi_2^2\rangle^{1/2+\varepsilon}\langle\tau-\xi^3\rangle^{1/2-2\varepsilon}}\\
			&\lesssim\chi_{\langle\xi\rangle^4\gg \langle\xi_1\rangle\langle\xi_2\rangle}\frac{|\xi|\langle\xi\rangle^{4/3-\delta}}{\langle\xi_1\rangle^{1/3}\langle\xi_2\rangle^{1/3}}\frac{1}{\langle\tau_1+\xi_1^2\rangle^{\frac{1}{2}-2\varepsilon}\langle\tau_2-\xi_2^2\rangle^{1/2+\varepsilon}\langle\tau-\xi^3\rangle^{1/2+\varepsilon}}\\
			&\lesssim \chi_{|\xi|\gg \max\{1,|\xi_1|^{1/2}\}}\frac{|\xi|^{1/2-\delta+6\varepsilon}}{\langle\tau_2-\xi_2^2\rangle^{1/2+\varepsilon}\langle\tau-\xi^3\rangle^{1/2+\varepsilon}}.
		\end{align*}
		Let $y(\xi) = \xi^3-\tau_1-(\xi-\xi_1)^2$. Then, $|dy| = |(3\xi^2-2(\xi-\xi_1))d\xi|\sim |\xi|^2|d\xi|$ since $|\xi|\gg \max\{1,|\xi_1|^{1/2}\}$. Let $\xi(y)$ be the inverse function of $y(\xi)$.  By Cauchy-Schwarz inequality, one can control this part by
		\begin{align*}
			&\quad\left\| \frac{\chi_{|\xi|\gg \max\{1,|\xi_1|^{1/2}\}}|\xi|^{1/2-\delta+6\varepsilon}}{\langle\tau-\tau_1-(\xi-\xi_1)^2\rangle^{1/2+\varepsilon}\langle\tau-\xi^3\rangle^{1/2+\varepsilon}}\right\|_{L^\infty_{\tau_1,\xi_1}L^2_{\xi,\tau}}\|f\|_{L^2}\|g\|_{L^2}\|h\|_{L^2}\\
			&\lesssim \left\| \frac{\chi_{|\xi|\gg \max\{1,|\xi_1|^{1/2}\}}|\xi|^{1/2-\delta+6\varepsilon}}{\langle\xi^3-\tau_1-(\xi-\xi_1)^2\rangle^{1/2+\varepsilon}}\right\|_{L^\infty_{\tau_1,\xi_1}L^2_{\xi}}\|f\|_{L^2}\|g\|_{L^2}\|h\|_{L^2} \\
			&\lesssim \left\| \frac{\chi_{|\xi(y)|\gg1 }|\xi(y)|^{1/2-\delta+6\varepsilon}}{\langle y\rangle^{1/2+\varepsilon}|\xi(y)|}\right\|_{L^\infty_{\xi}L^2_{y}}\|f\|_{L^2}\|g\|_{L^2}\|h\|_{L^2} \\
			&\lesssim\|f\|_{L^2}\|g\|_{L^2}\|h\|_{L^2} .
		\end{align*}
		Due to the symmtry between $\xi_1$ and $\xi_2$, the same argument works when $\langle \tau_2-\xi_2^2\rangle = \max\{\langle\tau_1+\xi_1^2\rangle,\langle\tau_2-\xi_2^2\rangle,\langle\tau-\xi^3\rangle\}$.		
		We finish the proof.
	\end{proof}
	
	\begin{lemma}\label{purekdv}
		If there exists $1/2<b<\tilde{b}$ such that
		\begin{align}\label{kkk}
			\|\partial_x(uv)\|_{Y^{s,\tilde{b}-1}}\lesssim \|u\|_{Y^{s,b}}\|v\|_{Y^{s,b}}
		\end{align}
		then $s>-3/4$, $b<3/4+s/3$. Conversely, if $s>-3/4$, $1/2<b<\min\{3/4+s/3,3/4\}$, then there exists $\tilde{b}>b$ such that \eqref{kkk} holds.
	\end{lemma}
	\begin{proof}[\textbf{Proof}]
		\eqref{kkk} is equivalent to
		\begin{align*}
			&\quad\int_{\substack{\xi_1+\xi_2 = \xi,\\\tau_1+\tau_2 = \tau}}\frac{|\xi|\langle\xi\rangle^sf(\tau_1,\xi_1)g(\tau_2,\xi_2)h(\tau,\xi)}{\langle\xi_1\rangle^{s}\langle\xi_2\rangle^{s}\langle\tau_1-\xi_1^3\rangle^{b}\langle\tau_2-\xi_2^3\rangle^{b}\langle\tau-\xi^3\rangle^{1-\tilde{b}}}\\
			&\lesssim \|f\|_{L^2}\|g\|_{L^2}\|h\|_{L^2},\quad \forall~f,g,h\geq 0.
		\end{align*}		
		Let  $f(\tau,\xi) = \chi_{[N-2N^{-1/2}, N+2N^{-1/2}]}(\xi)\chi_{[-2,2]}(\tau-\xi^3)$, $g(\tau,\xi) = \chi_{[-100,100]}(\tau-\xi^3)\chi_{[N-10N^{-1/2}, N+10N^{-1/2}]}(\xi)$, $h = f*g$. For $\xi\in [2N-N^{-1/2},2N+N^{-1/2}]$, $|\tau-\xi^3/4|\leq 1$, one has $h(\tau,\xi)\gtrsim N^{-1/2}$. Thus, $\|h\|_{L^2}\gtrsim N^{-3/4}$. Since
		\begin{align*}
			&\quad\int_{\substack{\xi_1+\xi_2 = \xi,\\\tau_1+\tau_2 = \tau}}\frac{|\xi|\langle\xi\rangle^sf(\tau_1,\xi_1)g(\tau_2,\xi_2)h(\tau,\xi)}{\langle\xi_1\rangle^{s}\langle\xi_2\rangle^{s}\langle\tau_1-\xi_1^3\rangle^{b}\langle\tau_2-\xi_2^3\rangle^{b}\langle\tau-\xi^3\rangle^{1-\tilde{b}}}\\
			&\gtrsim N^{1-s-3(1-\tilde{b})}\|h\|_{L^2}^2\sim N^{-9/4-s+3\tilde{b}}\|f\|_{L^2}\|g\|_{L^2}\|h\|_{L^2},
		\end{align*}
		we obtain $b<\tilde{b}\leq 3/4+s/3$.
		
		By multi-linear interpolation, we only need to consider the case $s = 0$, $1/2<b<3/4$ and $0<s+3/4\ll 1$, $1/2<b<3/4+s/3$. By Theorem 2.2 in \cite{kenig1996bilinear}, \eqref{kkk} holds in such region.
	\end{proof}
	
	\begin{proof}[\textbf{Proof of Proposition \ref{conclusion}}]
		For \eqref{mainxsbre}, one needs $s_1\geq 0$. For \eqref{mainsks}, by the proof of Lemma \ref{sktosb}, one needs $s_2>s_1-5/2$, $s_1-s_2-3b_2-1/2\leq 0$. For \eqref{mainssk}, by Lemmas \ref{sstkdv}, \ref{purekdv}, one needs $s_2<\max\{4s_1,s_1+1\}$. For \eqref{mainkkk}, one needs $s_2>-3/4$, $b_2<3/4+s_2/3$. Thus, $s_2\geq s_1-3b_2-1/2>s_1-3(3/4+s_2/3)-1/2$ which means $s_2>s_1/2-11/8$.
		
		Combining Lemmas \ref{sktosb}--\ref{sstkdv}, we have \eqref{mainxsbre}--\eqref{mainkkk} for some $1/2<b_1<\tilde{b}_1$, $1/2<b_2<\tilde{b_2}$ when $s_1\geq 0$, $\max\{-3/4,s_1-2\}<s_2<\min\{4s_1,s_1+1\}$. For $s_1>0$, $\max\{-3/4,s_1/2-11/8,s_1-5/2\}<s_2\leq s_1-2$, one has \eqref{mainssk} with $\tilde{b}_2 \leq 1$, $b_1>1/2$. By choosing $b_2 = \min\{3/4+s/3,3/4\}-\delta$ for some $0<\delta\ll 1$, we have \eqref{mainkkk} by Lemma \ref{purekdv} and \eqref{mainsks} by slightly modifying the proof of Lemma \ref{sktosb}. We finish the proof of this lemma.
	\end{proof}		
	By Proposition \ref{conclusion}, we have
	\begin{prop}
		\eqref{model} is local well-posed in $H^{s_1}\times H^{s_2}$ where		
		$$s_1\geq 0,\quad \max\{-3/4,s_1/2-11/8,s_1-5/2\}<s_2<\min\{4s_1,s_1+1\}.$$
	\end{prop}
	
	\section{Borderline cases by using \texorpdfstring{$U^p-V^p$}{Up-Vp} spaces}\label{borderlinecase}
	To manipulate the borderline cases, we need the $U^p, V^p$ spaces which were introduced in \cite{kochtataru}. Most of the materials can be found in \cite{hadac2009well}. For reader's convenience, we include the definitions and basic properties here.
	\subsection{\texorpdfstring{$U^p-V^p$}{Up-Vp} spaces}\label{defiupvp}
	
	\begin{defi}[Definition 2.1, 2.3 in \cite{hadac2009well}]
		Let $\mathcal{Z}$ be the set of finite partitions $-\infty=t_0<t_1<\cdots<t_K = \infty$. $1\leq p<\infty$. For $\{t_k\}_{k=0}^K\subset \mathcal{Z}$ and $\{\phi_k\}_{k=0}^{K-1}\subset L^2$ with $\sum_{k=0}^{K-1}\|\phi_k\|_{L^2}^p = 1$ and $\phi_0 = 0$, we call the function $a:\mathbb{R}\rightarrow L^2$ given by $a = \sum_{k=1}^K \chi_{[t_{k-1},t_k)}\phi_{k-1}$ a $U^p$-atom. Define the atomic space
		\begin{equation*}
			U^p:=\left\{u = \sum_{j=1}^\infty \lambda_ja_j: a_j~U^p\mbox{-atom}, \lambda_j\in \mathbb{C}~\mbox{such that}~\sum_{j=1}^\infty |\lambda_j|<\infty\right\}
		\end{equation*}
		with norm
		\begin{equation}
			\|u\|_{U^p}:=\inf\left\{\sum_{j=1}^\infty|\lambda_j|: u=\sum_{j=1}^\infty \lambda_j a_j, \lambda_j\in \mathbb{C}, a_j ~\mbox{is}~U^p\mbox{-atom}\right\}.
		\end{equation}
		Let $1\leq p<\infty$, the space $V^p$ is defined as the normed space of all functions $v:\mathbb{R}\rightarrow L^2$ such that $v(-\infty):=\lim_{t\rightarrow -\infty} v(t)$ exists and for which the norm
		\begin{equation*}
			\|v\|_{V^p}:= \sup_{\{t_k\}_{k=0}^K\in \mathcal{Z}} \left(\sum_{k=1}^K\|v(t_k)-v(t_{k-1})\|_{L^2}^p\right)^{1/p}
		\end{equation*}
		is finite, where we use the convention $v(\infty) = 0$. Let $V^p_{-, rc}$ denote all $v\in V^p$ which are right-continuous and $v(-\infty) = 0$.
	\end{defi}
	\begin{prop}[Basic properties, Proposition 2.2--2.5 and Corollary 2.6 in \cite{hadac2009well}]\label{basicproperty}
		Let $1\leq p<q<\infty$.
		\begin{itemize}
			\item[$\mathrm{(i)}$] $U^p, V^p$ are Banach spaces. $V^p_{-,rc}$ is a closed subspace of $V^p$.
			\item[$\mathrm{(ii)}$] The embedding $U^p\subset V_{-,rc}^p \subset L^\infty(\mathbb{R}, L^2)$ is continuous.
			\item[$\mathrm{(iii)}$] The embedding $V_{-,rc}^p\subset U^q$ is continuous.
		\end{itemize}
	\end{prop}
	\begin{prop}[Proposition 2.7 in \cite{hadac2009well}]
		Let $1<p<\infty$. For $u\in U^p$ and $v\in V^{p'}$ and a partition $\mathfrak{t}:=\{t_k\}_{k=0}^K\in \mathcal{Z}$, we define
		\begin{equation*}
			B_{\mathfrak{t}}(u,v):=\sum_{k=1}^K (u(t_{k-1}), v(t_k)-v(t_{k-1})).
		\end{equation*}
		$(\cdot,\cdot)$ denotes the $L^2$ inner product. There is a unique number $B(u,v)$ with the property that for all $\varepsilon>0$ there exists $\mathfrak{t}\in \mathcal{Z}$ such that for every $\mathfrak{t}'\supset \mathfrak{t}$ it holds that
		\begin{align*}
			|B_{\mathfrak{t}'}(u,v)-B(u,v)|<\varepsilon,
		\end{align*}
		and the associated bilinear form
		$$B:U^p\times V^{p'}:(u,v)\mapsto B(u,v)$$
		satisfies the estimate
		\begin{equation*}
			|B(u,v)|\leq \|u\|_{U^p}\|v\|_{V^{p'}}.
		\end{equation*}
	\end{prop}
	\begin{thm}[Theorem 2.8 in \cite{hadac2009well}]\label{dual}
		Let $1<p<\infty$. We have
		$$(U^p)^* = V^{p'}$$
		in the sense that
		$$T:V^{p'}\rightarrow (U^p)^*, \quad T(v)(u) = B(u,v)$$
		is an isometric isomorphism.
	\end{thm}
	For a real-valued function $h:\mathbb{R}\rightarrow \mathbb{R}$, we define 
	$$\|u\|_{U^p_h}:=\|e^{-ith(-i\partial_x)}u(t)\|_{U^p},\quad \|u\|_{V^p_h}:=\|e^{-ith(-i\partial_x)}u(t)\|_{V^p}.$$
	Let $Q_{h,\leq L}:=\mathscr{F}_{\tau,\xi}^{-1}\varphi((\tau-h(\xi))/L)\mathscr{F}_{t,x}$, $Q_{h,>L}:=I-Q_{h,\leq L}$.
	
	\begin{lemma}[Corollary 2.15 in \cite{hadac2009well}]\label{highmodu}
		We have
		\begin{equation*}
			\begin{array}{c}					
				\|Q_{h,>L}u\|_{L^2_{t,x}}\lesssim L^{-1/2}\|u\|_{V^2_h},\\
				\|Q_{h,\leq L}u\|_{V^p_h}\lesssim \|u\|_{V^p_h},~~ \|Q_{h,>L}u\|_{V^p_h}\lesssim \|u\|_{V^p_h},\\
				\|Q_{h,\leq L}u\|_{V^p_h}\lesssim \|u\|_{U^p_h},~~ \|Q_{h,>L}u\|_{V^p_h}\lesssim \|u\|_{U^p_h}.
			\end{array}
		\end{equation*}
	\end{lemma}
	In this paper, we define $$\|f\|_{U^p_S}:=\|S(-t)f(t)\|_{U^p},\quad\|f\|_{U^p_K}:=\|K(-t)f(t)\|_{U^p}$$ and similarly for $V^p_S,V^p_K$. Also, let $Q^S_{\leq L}:=\mathscr{F}_{\tau,\xi}^{-1}\varphi_L((\tau+\xi^2))\mathscr{F}_{t,x}$, $Q^S_{>L}:=I-Q^S_{\leq L}$, $Q^K_{\leq L}:=\mathscr{F}_{\tau,\xi}^{-1}\varphi_L((\tau-\xi^3))\mathscr{F}_{t,x}$, $Q_{>L}^K:=I-Q_{\leq L}^K$. In Subsection \ref{borderlinecaselow}, we also use $Q_{\leq L}^{S,\lambda}$ and $Q_{> L}^{S,\lambda}$ to denote $\mathscr{F}_{\tau,\xi}^{-1}\varphi_L((\tau+\lambda\xi^2))\mathscr{F}_{t,x}$ and $I-Q_{\leq L}^{S,\lambda}$ respectively. Replacing $\varphi_L$ with $\psi_L$, we define $Q^S_L$, $Q^K_L$, $Q^{S,\lambda}_L$ similarly.
	
	We have the following transversal estimates.
	\begin{lemma}\label{bilinearrefined}
		For $N\gg 1$, we have
		\begin{equation}\label{sstran}
			\|P_{N}(u_1\bar{u}_2)\|_{L^2_{t,x}}\lesssim N^{-{1/2}}\|u_1\|_{U^2_S}\|u_2\|_{U^2_S}.
		\end{equation}
		For $N_1\gg N_2^2$, we have
		\begin{align}\label{kdvtran1}
			\|P_{N_1}uP_{N_2}v\|_{L^2_{t,x}}\lesssim N_1^{-1/2} \|u\|_{U^2_S}\|v\|_{U^2_K}.
		\end{align}
		For $N_1\ll N_2^2$, we have
		\begin{align}\label{kdvtran2}
			\|P_{N_1}uP_{N_2}v\|_{L^2_{t,x}}\lesssim N_2^{-1} \|u\|_{U^2_S}\|v\|_{U^2_K}.
		\end{align}
	\end{lemma}
	\subsection{The case \texorpdfstring{$s_1\geq 0,s_2=\min\{4s_1,s_1+1\}$}{s1>=0,s2=min\{4s1,s1+1\}}}\label{cases2=4s1}
	Let $T>0$, $\varepsilon>0$. We denote the indicator of set $[0,T]\subset \mathbb{R}_t$ by  $\chi_T$. Then We define $\|u\|_{X^{s}}:=\|J^su\|_{U^2_S}$.
	\begin{align*}
		\|u\|_{X_{\varepsilon,T}^s}:=T^{-\varepsilon}\|u\|_{X^s},\quad \|v\|_{Y_T^{s}}:=\|J^s v\|_{V^2_K}+\|P_1v\|_{L_x^2 L_T^\infty}.
	\end{align*}
	Also, we define $
		\|u\|_{\mathscr{M}^s}=\|J^su\|_{V^2_S},\quad
		\|v\|_{\mathscr{N}^s}=\|J^sv\|_{U^2_K}$.
	We consider the linear and multi-linear estimates in $X^s_T$, $Y^s_T$.
	\begin{lemma}\label{linearupvp}
		Let $s\in\mathbb{R}$, $\varepsilon>0$. Then $\forall~T>0$, one has
		\begin{align*}
			\|\chi_{[0,\infty)}(t)S(t)u_0\|_{X_{\varepsilon,T}^s}= T^{-\varepsilon}\|u_0\|_{H^s},~
			\|\chi_{[0,\infty)}(t)K(t)v_0\|_{Y^s_T}\lesssim \langle T\rangle^{1/2}\|v_0\|_{H^s}.
		\end{align*}
	\end{lemma}
	\begin{proof}[\textbf{Proof}]
		By maximal function estimate, one has
		$$\|P_1\chi_{[0,\infty)}(t)K(t)v_0\|_{L_x^2L_T^\infty}\lesssim \langle T\rangle^{1/2}\|P_1v_0\|_{L^2}\lesssim \langle T\rangle^{1/2}\|v_0\|_{H^s}.$$
		It is easy to conclude the proof of this lemma by the definition of $U^2$, $V^2$. 
	\end{proof}
	
	\begin{lemma}\label{multilinearoutsch}
		Let $s_2\geq s_1\geq 0$, $\varepsilon>0$.  $\forall~T>0$ one has
		\begin{align*}
			\left\|\mathscr{A}(\chi_{T}u\bar{v} w)\right\|_{X_{\varepsilon,T}^{s_1}}\lesssim T^{1/2+2\varepsilon}\|u\|_{X_{\varepsilon,T}^{s_1}}\|v\|_{X_{\varepsilon,T}^{s_1}}\|w\|_{X_{\varepsilon,T}^{s_1}},
		\end{align*}
		and
		\begin{align*}
			\left\|\mathscr{A}(\chi_Tuv)\right\|_{X_{\varepsilon,T}^{s_1}}\lesssim T^{{13}/{16}}\|u\|_{X_{\varepsilon,T}^{s_1}}\|v\|_{Y^{s_2}_T}.
		\end{align*}
	\end{lemma}
	\begin{proof}[\textbf{Proof}]
		By duality, we only need to show
		\begin{align*}
			\left|\int_{\mathbb{R}}\int_0^Tu\bar{v} w \bar{f}~dtdx\right|&\lesssim T^{1/2+3\varepsilon}\|u\|_{X_{\varepsilon,T}^{s_1}}\|v\|_{X_{\varepsilon,T}^{s_1}}\|w\|_{X_{\varepsilon,T}^{s_1}}\|f\|_{\mathscr{M}^{-s_1}},\\
			\left|\int_{\mathbb{R}}\int_0^Tuv\bar{w}~dtdx\right|
			&\lesssim T^{{13}/{16}+\varepsilon}\|u\|_{X_{\varepsilon,T}^{s_1}}\|v\|_{Y_T^{s_2}}\|w\|_{\mathscr{M}^{-s_1}}.
		\end{align*}
		By H\"{o}lder inequality, fractional Leibniz rule, one has
		\begin{align*}
			\left|\int_{\mathbb{R}}u\bar{v}w\bar{f}~dx\right|&\leq \|J^{s_1}(u\bar{v} w)\|_{L^{4/3}_x}\|J^{-s_1}\bar{f}\|_{L^4_x}\\
			&\lesssim \|J^{s_1}u\|_{L^4_{x}}\|J^{s_1}v\|_{L^4_{x}}\|J^{s_1}w\|_{L^4_{x}}\|J^{-s_1}f\|_{L^4_x}.
		\end{align*}
		By H\"{o}lder inequality, $U^2_S\hookrightarrow V^2_S\hookrightarrow U^8_S\hookrightarrow L_t^8L_x^4$, we have
		\begin{align*}
			&\quad\left|\int_{\mathbb{R}}\int_0^Tu\bar{v} w \bar{f}~dtdx\right|\\
			&\lesssim T^{1/2}\|J^{s_1}u\|_{L_T^8L^4_x}\|J^{s_1}v\|_{L_T^8L^4_x}\|J^{s_1}w\|_{L_T^8L^4_x}\|J^{-s_1}f\|_{L_{T}^8L_x^4}\\
			&\lesssim T^{1/2}\|J^{s_1}u\|_{U^2_S}\|J^{s_1}v\|_{U^2_S}\|J^{s_1}w\|_{U^2_S}\|J^{-s_1}f\|_{V^2_S}\\
			&= T^{1/2+3\varepsilon}\|u\|_{X_{\varepsilon,T}^{s_1}}\|v\|_{X_{\varepsilon,T}^{s_1}}\|w\|_{X_{\varepsilon,T}^{s_1}}\|f\|_{\mathscr{M}^{-s_1}}.
		\end{align*}
		Similar to the former argument, 
		\begin{align*}
			\left|\int_{\mathbb{R}}u_1v\bar{u}_2~dx\right|&\leq \|J^{s_1}(u_1v)\|_{L_x^{16/9}}\|J^{-s_1} u_2\|_{L^{16/7}_x}\\
			&\lesssim \|J^{s_1}u_1\|_{L_x^{16/7}}\|J^{s_2}v\|_{L^8_x}\|J^{-s_1}u_2\|_{L^{16/7}_x}.
		\end{align*}
		By H\"{o}lder inequality and $V^2_K\hookrightarrow U^8_K\hookrightarrow L_{t,x}^8$, $U^2_S\hookrightarrow V^2_S\hookrightarrow U^{32}_S\hookrightarrow L_t^{32}L_x^{16/7}$, one has
		\begin{align*}
			&\quad\left|\int_{\mathbb{R}}\int_0^Tuv\bar{w}~dtdx\right|\\
			&\lesssim \|J^{s_1}u\|_{L_{x,T}^{16/7}}\|J^{s_2}v\|_{L_{x,T}^8}\|J^{-s_1} w\|_{L_{x,T}^{16/7}}\\
			&\lesssim T^{{13}/{16}}\|J^{s_1}u\|_{L_T^{32}L^{16/7}_x}\|J^{s_2}v\|_{L_{x,T}^8}\|J^{-s_1} w\|_{L_T^{32}L^{16/7}_x}\\
			&\lesssim T^{{13}/{16}} \|J^{s_1}u\|_{U^2_S}\|J^{s_2}v\|_{V^2_K}\|J^{-s_1} w\|_{V^2_S}\\
			&\lesssim T^{{13}/{16}+\varepsilon}\|u\|_{X_{\varepsilon,T}^{s_1}}\|v\|_{Y^{s_2}_T}\|w\|_{\mathscr{M}^{-s_1}}.
		\end{align*}
		We finish the proof of this lemma.
	\end{proof}
	
	\begin{lemma}\label{kdvlowfre}
		Let $0<T<1$, $s_1\geq -1/4$, $s_2\geq 0$. 
		\begin{align*}
			\left\|\mathscr{B}(\chi_TP_1(u\bar{v}))\right\|_{L^2_x L_T^\infty}&\lesssim T^{1/2+2\varepsilon}\|u\|_{X^{s_1}_{\varepsilon,T}}\|v\|_{X^{s_1}_{\varepsilon,T}},\\
			\left\|\mathscr{B}(\chi_TP_1(uv))\right\|_{L^2_x L_T^\infty}&\lesssim T^{5/6}\|u\|_{Y^{s_2}_T}\|v\|_{Y^{s_2}_T}.
		\end{align*}
	\end{lemma}
	\begin{proof}[\textbf{Proof}]
		By maximal function estimate and transversal estimates (lem \ref{bilinearrefined}) , one has
			\begin{align*}
			&	\quad\left\|\mathscr{B}(\chi_TP_1(u\bar{v}))\right\|_{L^2_x L_T^\infty}\\
			&\lesssim \|P_{\lesssim 1}u P_{\lesssim 1} \bar{v}\|_{L_{x}^2L_T^1}+\sum_{N\gg 1}\|P_N uP_{\sim N}\bar{v}\|_{L_x^2L_T^1}	\\
			&\lesssim T^{{1}/{2}}\|P_{\lesssim 1}u\|_{U^{2}_S}\|P_{\lesssim 1}\bar{v}\|_{U^{2}_S} + \sum_{N\gg 1} T^{{1}/{2}} N^{-{1}/{2}}\|P_{N}u\|_{U^2_S}\|P_{\sim N}\bar{v}\|_{U^2_S} \\
			&\lesssim T^{1/2+2\varepsilon}\|J^{-s_1}P_{\lesssim 1}u\|_{X^{s_1}_{\epsilon,T}}\|J^{-s_1}P_{\lesssim 1}\bar{v}\|_{X^{s_1}_{\epsilon,T}}\\
			&\qquad+ \sum_{N\gg 1}T^{1/2+2\varepsilon}N^{-\frac{1}{2}}\|J^{-s_1}P_Nu\|_{X^{s_1}_{\epsilon,T}}\|J^{-s_1}P_{\sim N}\bar{v}\|_{X^{s_1}_{\epsilon,T}}\\
			&\lesssim T^{1/2+2\varepsilon}\|u\|_{X^{s_1}_{\epsilon,T}}\|v\|_{X^{s_1}_{\epsilon,T}}+T^{1/2+2\varepsilon}\sum_{N\gg 1}N^{-1/2-2s_1}\|P_Nu\|_{X^{s_1}_{\epsilon,T}}\|P_{\sim N}\bar{v}\|_{X^{s_1}_{\epsilon,T}}\\
			&\lesssim T^{1/2+2\varepsilon}\|u\|_{X^{s_1}_{\varepsilon,T}}\|v\|_{X^{s_1}_{\varepsilon,T}}.
		\end{align*}
		Similarly, one has
		\begin{align*}
			&\quad\left\|\mathscr{B}(\chi_TP_1(uv))\right\|_{L^2_x L_T^\infty}\lesssim \|u\|_{L_T^2L_x^4}\|v\|_{L_T^2L_x^4}\\
			&\lesssim T^{5/6}\|u\|_{L_{t,x}^8}^{2/3}\|u\|_{L_T^\infty L_x^2}^{1/3}\|v\|_{L_{t,x}^8}^{2/3}\|v\|_{L_T^\infty L_x^2}^{1/3}\\
			&\lesssim T^{5/6}\|u\|_{Y^{s_2}_T}\|v\|_{Y^{s_2}_T}.
		\end{align*}
		We finish the proof of this lemma.
	\end{proof}
	
	\begin{lemma}\label{frequecylocbilinear}
		Let $0<T<1$, $N_1,N_2,N_3\geq 2$, $N_{\max}:=\max\{N_1,N_2,N_3\}$. Then for $s\geq 0$, we have
		\begin{align*}
			\left|\int_{\mathbb{R}}\int_0^T\partial_x(P_{N_1}uP_{N_2}v)P_{N_3}\bar{w}~dtdx\right|\lesssim T^{1/4} N_{\max}^{-1/{16}}\|u\|_{Y^s_T}\|v\|_{Y^s_T}\|w\|_{\mathscr{N}^{-s}}.
		\end{align*}
	\end{lemma}
	\begin{proof}[\textbf{Proof}]
		By replacing $u,v,w$ to $\chi_{[0,T)}(t)u$, $\chi_{[0,T)}(t)v$, $\chi_{[0,T)}(t)w$, we can assume that $u$, $v$, $w$ are supported on $[0,T]\times \mathbb{R}$.
		Firstly, by H\"{o}lder, Bernstein inequalities, we have
		\begin{equation}\label{trialone}
			\begin{aligned}
				&\quad\left|\int_{\mathbb{R}}\int_0^T\partial_x(P_{N_1}uP_{N_2}v)P_{N_3}\bar{w}~dtdx\right|\\
				&\lesssim T\|P_{N_1}u\|_{L_t^\infty L_x^3}\|P_{N_2}v\|_{L_t^\infty L_x^3}\|P_{N_3}w\|_{L_t^\infty L_x^3}\\
				&\lesssim TN_{\max}^{1/2}\|J^su\|_{L_t^\infty L_x^2}\|J^sv\|_{L_t^\infty L_x^2}\|J^{-s}w\|_{L_t^\infty L_x^2}\\  
				&\lesssim TN_{\max}^{1/2}\|u\|_{Y^s_T}\|v\|_{Y^s_T}\|w\|_{\mathscr{N}^{-s}}.
			\end{aligned}	
		\end{equation}
		By Lemma \ref{highmodu}, we have $\|Q^K_{>L}f\|_{L_{t,x}^2}\lesssim L^{-1/2}\|f\|_{V^2_K}$. Since
		\begin{align*}
			|\xi_1^3+\xi_2^3-\xi^3| = |-3\xi_1\xi_2\xi|\geq 3N_1N_2N_3/8,
		\end{align*}
		by choosing $L=N_1N_2N_3/8$, we have
		\begin{align*}
			\left|\int_{\mathbb{R}^2}\partial_x(Q^K_{\leq L}P_{N_1}uQ^K_{\leq L}P_{N_2}v)P_{N_3}\overline{Q^K_{\leq L}w}~dtdx\right| = 0.
		\end{align*}
		By Lemma \ref{highmodu} and H\"{o}lder inequality, we have
		\begin{align*}
			&\quad\left|\int_{\mathbb{R}^2}\partial_x(Q_{> L}^KP_{N_1}uP_{N_2}v)P_{N_3}\bar{w}~dtdx\right|\\
			&\leq \|Q_{>L}^KP_{N_1}u\|_{L_{t,x}^2}\|P_{N_2}v\|_{L_t^4 L_x^2}\|\partial_x P_{N_3}w\|_{L_t^4L_x^\infty}\\
			&\lesssim (N_1N_2N_3)^{-1/2}\|P_{N_1}u\|_{V^2_K}T^{1/4}\|P_{N_2}v\|_{V^2_K}N_3^{3/4}\|P_{N_3}w\|_{U^4_K}\\
			&\lesssim T^{1/4} N_{\max}^{-1/4}\|u\|_{Y^s_T}\|v\|_{Y^s_T}\|w\|_{\mathscr{N}^{-s}}
		\end{align*}
		Similarly, we have 
		\begin{align*}
			&\quad\left|\int_{\mathbb{R}^2}|P_{N_1}uQ^K_{> L}P_{N_2}v\partial_xP_{N_3}\bar{w}|+|P_{N_1}uP_{N_2}v\partial_xP_{N_3}\overline{Q_{> L}^Kw}|~dtdx\right|\\
			&\lesssim  T^{1/4} N_{\max}^{-1/4}\|u\|_{Y^s_T}\|v\|_{Y^s_T}\|w\|_{\mathscr{N}^{-s}}.
		\end{align*}
		By H\"{o}lder inequality, we have
		\begin{align*}
			&\quad \left|\int_{\mathbb{R}^2}\partial_x(Q_{>L}^KP_{N_1}uQ_{>L}^KP_{N_2}v)P_{N_3}\bar{w}~dtdx\right|\\
			&\leq \|Q_{>L}^KP_{N_1}u\|_{L_{t,x}^2}\|Q_{>L}^KP_{N_2}v\|_{L_{t,x}^2}\|\partial_xP_{N_3}w\|_{L_{t,x}^\infty}\\
			&\lesssim (N_1N_2N_3)^{-1}\|P_{N_1}u\|_{V^2_K}\|P_{N_2}v\|_{V^2_K}N_3^{3/2}\|P_{N_3}w\|_{U^2_K}\\
			&\lesssim N_{\max}^{-1/2}\|u\|_{Y^s_T}\|v\|_{Y^s_T}\|w\|_{\mathscr{N}^{-s}}.
		\end{align*}
		We can control other terms similarly. Thus,
		\begin{align*}
			&\quad\left|\int_{\mathbb{R}}\int_0^T\partial_x(P_{N_1}uP_{N_2}v)P_{N_3}\bar{w}~dtdx\right|\\
			&\lesssim (T^{1/4}N_{\max}^{-1/4}+N_{\max}^{-1/2})\|u\|_{Y^s_T}\|v\|_{Y^s_T}\|w\|_{\mathscr{N}^{-s}}.
		\end{align*}
		By interpolation with \eqref{trialone}, we have
		\begin{align*}
			&\quad\left|\int_{\mathbb{R}}\int_0^T\partial_x(P_{N_1}uP_{N_2}v)P_{N_3}\bar{w}~dtdx\right|\\
			&\lesssim (T^{1/4}N_{\max}^{-1/4}+N_{\max}^{-1/2})^{1-\delta}(TN_{\max}^{1/2})^\delta\|u\|_{Y^s_T}\|v\|_{Y^s_T}\|w\|_{\mathscr{N}^{-s}}\\
			&\lesssim T^\delta N_{\max}^{-1/4+3\delta/4}\|u\|_{Y^s_T}\|v\|_{Y^s_T}\|w\|_{\mathscr{N}^{-s}}.
		\end{align*}
		By choosing $\delta = 1/4$, we conclude the proof of the lemma.
	\end{proof}
	\begin{lemma}\label{kdvupvp}
		Let $0<T<1$, $s\geq 0$. Then,
		\begin{align*}
			\left\|\mathscr{B}(\chi_Tuv)\right\|_{Y_T^{s}}&\lesssim T^{1/4}\|u\|_{Y_T^{s}}\|v\|_{Y_T^{s}}.
		\end{align*}
	\end{lemma}
	\begin{proof}[\textbf{Proof}]
		By Lemma \ref{kdvlowfre} and duality, we only need to prove
		\begin{align*}
			\left|\int_{\mathbb{R}}\int_0^T\partial_x(uv)\bar{w}~dtdx\right|\lesssim T^{1/4}\|u\|_{Y_T^{s}}\|v\|_{Y_T^{s}}\|w\|_{\mathscr{N}^{-s}}
		\end{align*}
		By triangle inequality, one has
		\begin{align*}
			&\quad\left|\int_{\mathbb{R}}\int_0^T\partial_x(uv)\bar{w}~dtdx\right|\\
			&\lesssim  \sum_{N_1,N_2\geq 2,N_3\geq 1}\left|\int_{\mathbb{R}}\int_0^TP_{N_1}uP_{N_2}v\partial_xP_{N_3}\bar{w}~dtdx\right|\\
			&\quad +\left|\int_{\mathbb{R}}\int_0^T\partial_x(P_1uv)\bar{w}~dtdx\right|+\left|\int_{\mathbb{R}}\int_0^T\partial_x((I-P_1)v_1P_1v_2)\bar{v}_3~dtdx\right|.
		\end{align*}
		It is easy to control when $N_3 = 1$.   If $N_1 = 1$, we have
		\begin{align*}
			&\quad\left|\int_{\mathbb{R}}\int_0^T\partial_x(P_1uv)\bar{w}~dtdx\right|\\
			&\leq \|P_1 u\|_{L_x^2 L_T^\infty}\|J^s v\|_{L_T^2L_{x}^2}\|J^{-s}\partial_x w\|_{L_x^\infty L_t^2}\\ 
			&\lesssim T^{1/2}\|u\|_{Y^s_T}\|J^sv\|_{V^2_K}\|J^{-s}w\|_{U^2_K}\\
			&\lesssim T^{1/2}\|u\|_{Y^s_T}\|v\|_{Y^s_T}\|w\|_{\mathscr{N}^{-s}}.
		\end{align*}
		It is similar when $N_2 = 1$. Thus, we assume $N_1,N_2,N_3\geq 2$. By Lemma \ref{frequecylocbilinear}, we have
		\begin{align*}
			&\quad\sum_{N_1,N_2,N_3\geq 2}\left|\int_{\mathbb{R}}\int_0^TP_{N_1}uP_{N_2}v\partial_xP_{N_3}\bar{w}~dtdx\right|\\
			&\lesssim \sum_{N_1,N_2,N_3\geq 2}T^{1/4}N_{\max}^{-1/16}\|u\|_{Y^s_T}\|v\|_{Y^s_T}\|w\|_{\mathscr{N}^{-s}}\\
			&\lesssim T^{1/4} \|u\|_{Y^s_T}\|v\|_{Y^s_T}\|w\|_{\mathscr{N}^{-s}}.
		\end{align*}
		Since we assume $T<1$, the proof is completed.
	\end{proof}
	
	\begin{lemma}\label{frequecylocsbilinear}
		Let $0<T<1$, $N_1,N_2,N_3\geq 1$. If $N_3\lesssim 1$ or $N_3\gg N_1^{1/2}$, one has
		\begin{equation}\label{timenonreson}
			\begin{aligned}
				\left|\int_{\mathbb{R}}\int_0^T\partial_x(P_{N_1}uP_{N_2}\bar{v})P_{N_3}\bar{w}~dtdx\right|\lesssim N_3^{-1} \|u\|_{U^2_S}\|v\|_{U^2_S}\|w\|_{U^2_K}.
			\end{aligned}
		\end{equation} 
		If $N_3\ll N_1^{1/2}$, one has
		\begin{equation}\label{good}
			\begin{aligned}
				\left|\int_{\mathbb{R}}\int_0^T\partial_x(P_{N_1}uP_{N_2}\bar{v})P_{N_3}\bar{w}~dtdx\right|\lesssim N_1^{-1/2} \|u\|_{U^2_S}\|v\|_{U^2_S}\|w\|_{U^2_K}.
			\end{aligned}
		\end{equation}
		If $N_3\sim N_1^{1/2}\gg 1$, one has
		\begin{equation}\label{timeres}
			\begin{aligned}
				\left|\int_{\mathbb{R}}\int_0^T\partial_x(P_{N_1}uP_{N_2}\bar{v})P_{N_3\sim N_1^{1/2}}\bar{w}~dtdx\right|\lesssim T^{1/2}\|u\|_{U^2_S}\|v\|_{U^2_S}\|w\|_{U^2_K}.
			\end{aligned}
		\end{equation}
	\end{lemma}
	\begin{proof}[\textbf{Proof}]
		By replacing $u_1$, $u_2$, $v$ to $\chi_{[0,T)}(t)u_1$, $\chi_{[0,T)}(t)u_2$, $\chi_{[0,T)}(t)v$, we can assume that $u_1$, $u_2$, $v$ are supported on $[0,T]\times \mathbb{R}$. It is easy to obtain the estimate when $N_3\lesssim 1$. Thus, we assume $N_3\gg 1$. 
		By Lemma \ref{highmodu}, we have $\|Q^K_{>L}f\|_{L_{t,x}^2}\lesssim L^{-1/2}\|f\|_{V^2_K}$, $\|Q^S_{>L}f\|_{L_{t,x}^2}\lesssim L^{-1/2}\|f\|_{V^2_S}$.
		
		If $N_3\gg N_1^{1/2}$, by choosing $L = cN_3^3$ for some sufficiently small $c>0$, we have
		\begin{align*}
			\int_{\mathbb{R}^2}\partial_x(P_{N_1}Q^S_{\leq L}uP_{N_2}\overline{Q^S_{\leq L}v})P_{N_3}\overline{Q^K_{\leq L}w}~dtdx=0.
		\end{align*}
		By H\"{o}lder inequality, \eqref{kdvtran2},  and Lemma \ref{highmodu}, we have
		\begin{align*}
			&\quad\left|\int_{\mathbb{R}^2}\partial_x(Q^S_{> L}P_{N_1}uP_{N_2}\bar{v})P_{N_3}\bar{w}~dtdx\right|\\
			&\leq \|Q^S_{>L}P_{N_1}u\|_{L_{t,x}^2}\|P_{N_2}\bar{v}\partial_xP_{N_3}\bar{w}\|_{L_{t,x}^2}\\
			&\lesssim N_{3}^{-3/2}\|P_{N_1}u\|_{V^2_S}N_{3}^{-1}\|P_{N_2}v\|_{U^2_S}\|\partial_{x}P_{N_3}w\|_{U^2_K}\lesssim  N_{3}^{-3/2}\|u\|_{U^2_S}\|v\|_{U^2_S}\|w\|_{U^2_K}.
		\end{align*}
		By H\"{o}lder inequality and \eqref{sstran}, we have
		\begin{align*}
			&\quad \left|\int_{\mathbb{R}^2}\partial_x(P_{N_1}uP_{N_2}\bar{v})P_{N_3}\overline{Q_{K,> L}w}~dtdx\right|\\
			&\leq \|P_{N_3}(P_{N_1}uP_{N_2}\bar{v})\|_{L_{t,x}^2}\|\partial_xP_{N_3}Q^K_{> L}w\|_{L_{t,x}^2}\lesssim N_{3}^{-1}\|u\|_{U^2_S}\|v\|_{U^2_S}\|w\|_{U^2_K}.
		\end{align*}
		If $N_3\ll N_1^{1/2}$, then we have $N_2\sim N_1$. By choosing $L = cN_3N_1$ for some sufficiently small $c>0$, we have
		\begin{align*}
			\int_{\mathbb{R}^2}\partial_x(P_{N_1}Q^S_{\leq L}uP_{N_2}\overline{Q^S_{\leq L}v})P_{N_3}\overline{Q^K_{\leq L}w}~dtdx=0.
		\end{align*}
		By Lemma \ref{highmodu}, \eqref{kdvtran1} and H\"{o}lder inequality, we have
		\begin{align*}
			&\quad\left|\int_{\mathbb{R}^2}\partial_x(Q^S_{> L}P_{N_1}uP_{N_2}\bar{v})P_{N_3}\bar{w}~dtdx\right|\\
			&\leq \|Q^S_{>L}P_{N_1}u\|_{L_{t,x}^2}\|P_{N_2}\bar{v}\partial_xP_{N_3}\bar{w}\|_{L_{t,x}^2}\\
			&\lesssim  N_{1}^{-1}N_3^{1/2}\|P_{N_1}u\|_{U^2_S}\|P_{N_2}v\|_{U^2_S}\|P_{N_3}w\|_{U^2_K}.
		\end{align*}
		By  H\"{o}lder inequality and Lemma \ref{highmodu}, we have
		\begin{align*}
			&\quad \left|\int_{\mathbb{R}^2}\partial_x(P_{N_1}uP_{N_2}\bar{v})P_{N_3}\overline{Q^K_{> L}w}~dtdx\right|\\
			&\leq \|P_{N_3}(P_{N_1}uP_{N_2}\bar{v})\|_{L_{t,x}^2}\|\partial_xP_{N_3}Q^K_{> L}w\|_{L_{t,x}^2}\lesssim N_{1}^{-1/2}\|u\|_{U^2_S}\|v\|_{U^2_S}\|w\|_{U^2_K}.
		\end{align*}
		For $N_1\sim N_2\sim N_3^2$, we claim that for any $t>0$
		\begin{equation}\label{linearwave}
			\begin{aligned}
				&\quad\left|\int_{\mathbb{R}}\int_0^t\partial_x(P_{N_1}S(t')fP_{N_2}\overline{S(t')g})P_{N_3\sim N_1^{1/2}}\overline{K(t')h}~dt'dx\right|\\
				&\lesssim t^{1/2}\|f\|_{L_x^2}\|g\|_{L_x^2}\|h\|_{L_x^2}.
			\end{aligned} 
		\end{equation}
		In fact, following the argument for proving \eqref{sstokdv}, proving \eqref{linearwave} reduces to proving
		\begin{align*}
			&\quad\int_{\mathbb{R}}\frac{|1-e^{it\xi(\xi-2\xi_1-\xi^2)}|\chi_{|\xi|^2\sim|\xi_1|\sim N_1\gg 1}}{|\xi-2\xi_1-\xi^2|}|f(\xi_1)g(\xi-\xi_1)h(\xi)|~d\xi_1d\xi\\
			&\lesssim t^{1/2}\|f\|_{L^2}\|g\|_{L^2}\|h\|_{L^2}.
		\end{align*}
		By Cauchy-Schwarz inequality, one has
		\begin{align*}
			&\quad\int_{\mathbb{R}^2}\frac{|1-e^{it\xi(\xi-2\xi_1-\xi^2)}|\chi_{|\xi|^2\sim|\xi_1|\sim N_1\gg 1}}{|\xi-2\xi_1-\xi^2|}|f(\xi_1)g(\xi-\xi_1)h(\xi)|~d\xi_1d\xi\\
			&\lesssim \int_{\mathbb{R}_{\xi_1}}|f(\xi_1)| \|g(\xi-\xi_1)h(\xi)\|_{L^2_{\xi}}\left\|\frac{\chi_{1\ll |\xi|^2\sim |\xi_1|}\min\{1,t|\xi(\xi-2\xi_1-\xi^2)|\}}{|\xi-2\xi_1-\xi^2|}\right\|_{L^2_\xi}\\
			&\lesssim \|f\|_{L^2}\|g\|_{L^2}\|h\|_{L^2}\left\|\frac{\chi_{1\ll |\xi|^2\sim |\xi_1|}\min\{1,t|\xi_1|^{1/2}|\xi-2\xi_1-\xi^2|\}}{|\xi-2\xi_1-\xi^2|}\right\|_{L^\infty_{\xi_1}L^2_\xi}\\
			&\lesssim t^{1/2}\|f\|_{L^2}\|g\|_{L^2}\|h\|_{L^2}.
		\end{align*}
		Then for $u = \sum_j \chi_{I_{j}}(t)S(t)f_j$, $v = \sum_j\chi_{\tilde{I}_{j}}(t)S(t)g_j$, $w = \sum_j\chi_{\tilde{\tilde{I}}_j}(t)K(t)h_j$, by \eqref{linearwave} we have 
		\begin{align*}
			&\quad\left|\int_{\mathbb{R}}\int_0^T\partial_x(P_{N_1}uP_{N_2}\bar{v})P_{N_3\sim N_1^{1/2}}\bar{w}~dtdx\right|\\
			&\leq \sum_{j,k,l}\left|\int_{\mathbb{R}}\int_{I_{j}\cap \tilde{I}_{k}\cap \tilde{\tilde{I}}_l}\partial_x(P_{N_1}S(t')f_jP_{N_2}\overline{S(t')g_k})P_{N_3\sim N_1^{1/2}}\overline{K(t')h_l}~dt'dx\right|\\
			&\lesssim \sum_{j,k,l}|I_{j}\cap \tilde{I}_{k}\cap \tilde{\tilde{I}}_l|^{1/2}\|f_j\|_{L^2}\|g_k\|_{L^2}\|h_l\|_{L^2}\lesssim T^{1/2}\|f_j\|_{l^2_jL^2}\|g_j\|_{l^2_jL^2}\|h_j\|_{l^2_jL^2}.
		\end{align*}
		By the definition of $U^2_K$, $U^2_S$, we obtain \eqref{timeres}.
	\end{proof}
	
	\begin{lemma}\label{bilinearestioutkdv}
		Let $0<T<1$, $s_1\geq 0$, $s_2 = \min\{4s_1,s_1+1\}$.
		\begin{align*}
			\left\|\mathscr{B}(\chi_T u\bar{v})\right\|_{Y^{s_2}_T}&\lesssim T^{2\varepsilon}\|u\|_{X_{\varepsilon,T}^{s_				1}}\|v\|_{X_{\varepsilon,T}^{s_1}}.
		\end{align*}
	\end{lemma}
	\begin{proof}[\textbf{Proof}]
		By Lemma \ref{kdvlowfre} and duality, we only need to prove
		\begin{align*}
			\left|\int_{\mathbb{R}}\int_0^T\partial_x(J^{-s_1}uJ^{-s_1}\bar{v})J^{s_2}\bar{w}~dtdx\right|\lesssim \|u\|_{U^2_S}\|v\|_{U^{2}_S}\|w\|_{U^2_K}.
		\end{align*}
		Let $N_{\max}$, $N_{\mathrm{med}}$ be the maximal, medium among $N_1,N_2,N_3$. By triangle inequality, one has
		\begin{align*}
			&\quad\left|\int_{\mathbb{R}}\int_0^T\partial_x(J^{-s_1}uJ^{-s_1}\bar{v})J^{s_2}\bar{w}~dtdx\right|\\
			&\leq \sum_{N_1,N_2\nsim N_3~ \mathrm{or}~N_3\lesssim 1}\left|\int_{\mathbb{R}}\int_0^T\partial_x(P_{N_1}J^{-s_1}uP_{N_2}J^{-s_1}\bar{v})P_{N_3}J^{s_2}\bar{w}~dtdx\right|\\
			&\quad+\sum_{N_1\sim N_2\gg 1}\left|\int_{\mathbb{R}}\int_0^T\partial_x(P_{N_1}J^{-s_1}uP_{N_2}J^{-s_1}\bar{v})P_{N_3\sim N_1^{1/2}}J^{s_2}\bar{w}~dtdx\right|.
		\end{align*}			
		By Lemma \ref{frequecylocsbilinear} and $s_2 = \min\{4s_1,s_1+1\}$, we have
		\begin{align*}
			&\quad\left|\int_{\mathbb{R}}\int_0^T\partial_x(J^{-s_1}u J^{-s_1}\bar{v})J^{s_2}\bar{w}~dtdx\right|\\
			&\lesssim \sum_{\substack{N_1,N_2\nsim N_3^2~ \mathrm{or}~N_3\lesssim 1,\\ N_{\max}\sim N_{\mathrm{med}}}}N_3^{-1}\|P_{N_1}J^{-s_1}u\|_{U^2_S}\|P_{N_2}J^{-s_1}v\|_{U^2_S}\|P_{N_3}J^{s_2}w\|_{U^2_K}\\
			&\quad + \sum_{N_1\sim N_2\gg 1}T^{1/2}\|P_{N_1}J^{-s_1}u\|_{U^2_S}\|P_{N_2}J^{-s_1}v\|_{U^2_S}\|J^{s_2}w\|_{U^2_K}\\
			&\lesssim \sum_{\substack{N_1,N_2\nsim N_3^2~ \mathrm{or}~N_3\lesssim 1,\\ N_{\max}\sim N_{\mathrm{med}}}}N_3^{-1+s_2}N_1^{-s_1}N_2^{-s_1}\|P_{N_1}u\|_{U^2_S}\|P_{N_2}v\|_{U^2_S}\|P_{N_3}w\|_{U^2_K}\\
			&\quad+ \sum_{N_1\sim N_2\gg 1}N_1^{-s_1}N_2^{-s_1}N_1^{s_2/2}\|P_{N_1}u\|_{U^2_S}\|P_{N_2}v\|_{U^2_S}\|w\|_{U^2_K}\\
			&\lesssim \|P_{N}u\|_{l^2_N U^2_S}\|P_{N}v\|_{l^2_N U^2_S}\|w\|_{U^2_K}\\
			&\lesssim \|u\|_{U^2_S}\|v\|_{U^2_S}\|w\|_{U^2_K}.
		\end{align*}
		We finish the proof of this lemma.
	\end{proof}
	Combining Lemmas \ref{linearupvp}, \ref{multilinearoutsch}, \ref{kdvupvp}, \ref{bilinearestioutkdv}, we have
	\begin{prop}
		\eqref{model} is local well-posed in $H^{s_1}\times H^{s_2}$ for
		$s_1\geq 0,s_2=\min\{4s_1,s_1+1\}$. 
	\end{prop}
	\begin{proof}[\textbf{Proof}]
		For $T>0$, consider the mapping
		\begin{align*}
			\mathcal{T}:
			\begin{pmatrix}
				u\\
				v
			\end{pmatrix}\mapsto
			\begin{pmatrix}
				\chi_{[0,\infty)}(t)S(t)u_0-i\mathscr{A}(\chi_T(uv+|u|^2u))\\
				\chi_{[0,\infty)}(t)K(t)v_0+\mathscr{B}(\chi_T(|u|^2-v^2/2))
			\end{pmatrix}.
		\end{align*}
		For some sufficiently large $C_2,C_1>0$ (independent to $T$), we define
		\begin{align*}
			D:=\{(u,v)\in X^{s_1}_{\varepsilon,T}\times Y^{s_2}_T:\|u\|_{X^{s_1}_{\varepsilon,T}}\leq C_1T^{-\varepsilon}, \|v\|_{Y^{s_2}_T}\leq C_2\}.
		\end{align*}
		By Lemmas \ref{linearupvp}, \ref{multilinearoutsch}, \ref{kdvupvp}, \ref{bilinearestioutkdv}, we obtain that $\mathcal{T}$ is a contraction mapping from $D$ to $D$ for sufficiently small $T>0$. By standard argument, we conclude the proof.
	\end{proof}
	
	\subsection{The case \texorpdfstring{$s_2 = -3/4, ~0\leq s_1<5/4$}{s2=-3/4,0<=s1<5/4}}\label{borderlinecaselow}
	We need scaling. Following the argument in \cite{guo2010well}, we consider the system
	\begin{equation}\label{scalingeq}
		\left\{
		\begin{aligned}
			&i\partial_t u+\lambda\partial_{xx}u = \lambda uv+ \lambda^{-1}|u|^2u,\\
			&\partial_t v+\partial_{xxx}v+v\partial_x v = \partial_x(|u|^2),\\
			&(u,v)|_{t = 0} = (u_0,v_0)\in H^{s_1}\times H^{s_2}.
		\end{aligned}
		\right.
	\end{equation}
	We assume that $\|v_0\|_{H^{s_2}}\ll 1$, $0<\lambda\ll 1$. Let $$S_\lambda(t) = e^{i\lambda t\partial_{xx}},\quad\mathscr{A}_\lambda(f)(t) = \int_0^tS_\lambda(t-t')f(t')~dt'.$$ Define
	$$\|u\|_{\tilde{X}^{s,b}_\lambda}=\|\langle\xi\rangle^s\langle\tau+\lambda\xi^2\rangle^b\hat{u}(\tau,\xi)\|_{L^2_{\tau,\xi}},~~\|v\|_{Y^{s,b,q}}= \|N^sL^b\|Q_{L}^KP_{N}v\|_{L_{t,x}^2}\|_{l^2_N l^q_L}$$
	and
	$$\|v\|_{F}:=\|P_1v\|_{L_x^2 L_t^\infty}+\|P_{>1}v\|_{Y^{-3/4,1/2,1}}.$$
	Let $F_T$ be the space $F$ restricted on $[0,T]$.
	\begin{lemma}\label{s2-3/4}
		Let $0\leq s< 5/4$. $b>1/2$, $0<\lambda< 1$, $0<T<1$.
		\begin{align*}
			\|uP_{>1}v\|_{\tilde{X}^{s,b-1}_\lambda}\lesssim \lambda{^{-1/2}}\|u\|_{\tilde{X}^{s,b}_\lambda}\|v\|_{Y^{-3/4,1/2,1}}.
		\end{align*}
	\end{lemma}
	\begin{proof}[\textbf{Proof}]
		By duality, we only need to show
		\begin{align*}
			\left|\int_{\mathbb{R}^2}uP_{>1} v\bar{w}~dtdx\right|\lesssim \|u\|_{\tilde{X}^{s,b}_\lambda}\|v\|_{Y^{-3/4,1/2,1}}\|w\|_{\tilde{X}^{-s,1-b}_{\lambda}}.
		\end{align*}
		By the argument for Lemma 3.4 (a) in \cite{guo2010well}, we only need to prove
		\begin{align*}
			\sum_{N_1\ll N_2\sim N}\left|\int_{\mathbb{R}^2}P_{N_1}uP_{N_2} v\overline{P_{N}w}~dtdx\right|\lesssim \|u\|_{\tilde{X}^{s,b}_\lambda}\|v\|_{Y^{-3/4,1/2,1}}\|w\|_{\tilde{X}^{-s,1-b}_{\lambda}}.
		\end{align*}
		By modulation decomposition, one has
		\begin{align*}
			&\quad\sum_{N_1\ll N_2\sim N}\left|\int_{\mathbb{R}^2}P_{N_1}uP_{N_2} v\overline{P_{N}w}~dtdx\right|\\
			&\lesssim \sum_{\substack{N_1\ll N_2\sim N\\ L_1,L_2,L}}\left|\int_{\mathbb{R}^2}Q_{L_1}^{S,\lambda}P_{N_1}uQ_{L_2}^KP_{N_2} v\overline{Q_{L}^{S,\lambda}P_{N}w}~dtdx\right|.
		\end{align*}
		Let $L_{\max} = \max\{L_1,L_2,L\}$. Note that $L_{\max}\gtrsim N^3$. Let $0<\varepsilon\ll 1$. If $L_1 = L_{\max}$, we have
		\begin{align*}
			&\quad\left|\int_{\mathbb{R}^2}Q_{L_1}^{S,\lambda}P_{N_1}uQ_{L_2}^KP_{N_2} v\overline{Q_{L}^{S,\lambda}P_{N}w}~dtdx\right|\\
			&\lesssim \|Q_{L_1}^{S,\lambda}P_{N_1}u\|_{L_{t,x}^2}\|Q_{L_2}^KP_{N_2} v\overline{Q_{L}^{S,\lambda}P_{N}w}\|_{L^2_{t,x}}\\
			&\lesssim N_2^{-1}(L_2L)^{1/2}\|Q_{L_1}^{S,\lambda}P_{N_1}u\|_{L_{t,x}^2}\|Q_{L_2}^KP_{N_2} v\|_{L_{t,x}^2}\|Q_{L}^{S,\lambda}P_{N}w\|_{L^2_{t,x}}.
		\end{align*}
		Then for $-5/2+\varepsilon+s+3/4<0$ ($s<7/4$) one has 
		\begin{align*}
			&\quad\sum_{\substack{N_1\ll N_2\sim N\\ L_1=L_{\max}, L_2,L_3}}\left|\int_{\mathbb{R}^2}Q_{L_1}^{S,\lambda}P_{N_1}uQ_{L_2}^KP_{N_2} v\overline{Q_{L}^{S,\lambda}P_{N}w}~dtdx\right|\\
			&\lesssim \sum_{N_2\geq 2,L_2}N_2^{-1+\epsilon-3/2+s}\|u\|_{\tilde{X}^{s,b}_\lambda}L_2^{1/2}\|Q_{L_2}^KP_{N_2} v\|_{L_{t,x}^2}\|w\|_{\tilde{X}^{-s,1-b}_{\lambda}}\\
			&\lesssim \|u\|_{\tilde{X}^{s,b}_\lambda}\|v\|_{Y^{-3/4,1/2,1}}\|w\|_{\tilde{X}^{-s,1-b}_{\lambda}}.
		\end{align*}
		If $L = L_{\max}$, we have
		\begin{align*}
			&\quad\left|\int_{\mathbb{R}^2}Q_{L_1}^{S,\lambda}P_{N_1}uQ_{L_2}^KP_{N_2} v\overline{Q_{L}^{S,\lambda}P_{N}w}~dtdx\right|\\
			&\lesssim \|Q_{L_1}^{S,\lambda}P_{N_1}uQ_{L_2}^KP_{N_2} v\|_{L_{t,x}^2}\|Q_{L}^{S,\lambda}P_{N}w\|_{L^2_{t,x}}\\
			&\lesssim N_2^{-1}(L_1L_2)^{1/2}\|Q_{L_1}^{S,\lambda}P_{N_1}u\|_{L_{t,x}^2}\|Q_{L_2}^KP_{N_2} v\|_{L_{t,x}^2}\|Q_{L}^{S,\lambda}P_{N}w\|_{L^2_{t,x}}.
		\end{align*}
		Thus for $(-1/4+\varepsilon)+3(b+\varepsilon-1)+s<0$ ($s<7/4$) we have
		\begin{align*}
			&\quad\sum_{\substack{N_1\ll N_2\sim N\\ L=L_{\max},L_1,L_2}}\left|\int_{\mathbb{R}^2}Q_{L_1}^{S,\lambda}P_{N_1}uQ_{L_2}^KP_{N_2} v\overline{Q_{L}^{S,\lambda}P_{N}w}~dtdx\right|\\
			&\lesssim \sum_{L\gtrsim N^3}N^{-1+\varepsilon+3/4} \|u\|_{\tilde{X}^{s,b}_\lambda}\|v\|_{Y^{-3/4,1/2,1}}\|Q_{L}^{S,\lambda}P_{N}w\|_{L^2_{t,x}}\\
			&\lesssim \|u\|_{\tilde{X}^{s,b}_\lambda}\|v\|_{Y^{-3/4,1/2,1}}\|w\|_{\tilde{X}^{-s,1-b}_{\lambda}}.
		\end{align*}
		If $L_2 = L_{\max}$, we have
		\begin{align*}
			&\quad\left|\int_{\mathbb{R}^2}Q_{L_1}^{S,\lambda}P_{N_1}uQ_{L_2}^KP_{N_2} v\overline{Q_{L}^{S,\lambda}P_{N}w}~dtdx\right|\\
			&\lesssim \|Q_{L_1}^{S,\lambda}P_{N_1}u\overline{Q_{L}^{S,\lambda}P_{N}w}\|_{L_{t,x}^2}\|Q_{L_2}^KP_{N_2} v\|_{L^2_{t,x}}\\
			&\lesssim (\lambda N)^{-1/2}(L_1L)^{1/2}\|Q_{L_1}^{S,\lambda}P_{N_1}u\|_{L_{t,x}^2}\|Q_{L_2}^KP_{N_2} v\|_{L_{t,x}^2}\|Q_{L}^{S,\lambda}P_{N}w\|_{L^2_{t,x}}.
		\end{align*}
		Thus for $-1/2+\varepsilon+s+3(b-1)<-3/4$ one has
		\begin{align*}
			&\quad\sum_{\substack{N_1\ll N_2\sim N\\ L_2=L_{\max},L_1,L}}\left|\int_{\mathbb{R}^2}Q_{L_1}^{S,\lambda}P_{N_1}uQ_{L_2}^KP_{N_2} v\overline{Q_{L}^{S,\lambda}P_{N}w}~dtdx\right|\\
			&\lesssim \sum_{L_2\gtrsim N_2^3}(\lambda N_2)^{-1/2}N_2^{\varepsilon+s}L_2^{b-1/2} \|u\|_{\tilde{X}^{s,b}_\lambda}\|Q_{L_2}^KP_{N_2} v\|_{L^2_{t,x}}\|w\|_{\tilde{X}^{-s,1-b}_{\lambda}}\\
			&\lesssim \lambda^{-1/2} \|u\|_{\tilde{X}^{s,b}_\lambda}\|v\|_{Y^{-3/4,1/2,1}}\|w\|_{\tilde{X}^{-s,1-b}_{\lambda}}.
		\end{align*}
		Since one can choose $\varepsilon$ sufficiently small, thus we can obtain the desired inequality when $s<5/4$. 
	\end{proof}
	For other terms, the estimates in \cite{guo2010well} are also effective
	here. Thus we obtain
	\begin{prop}\label{wellpose-3/4}
		Let $s_2 = -3/4$, $0\leq s_1<5/4$. Given $(u_0,v_0)\in H^{s_1}\times H^{s_2}$ with $\|v_0\|_{H^{s_2}}\ll 1$, then the equation \eqref{scalingeq} has a unique solution $(u,v)\in \tilde{X}^{s_1,b}_{\lambda,T}\times F_T$.
	\end{prop}
	By rescaling we obtain the local well-posedness of \eqref{model} with $0\leq s_1<5/4$, $s_2 = -3/4$.
	
	\section{Other regions by using normal form argument}\label{normalform}
	For other regions, the main problem comes from high modulation. Thus we use normal form argument which is a powerful tool to control the high modulation cases.
	\subsection{Upper region}\label{upperreg}
	In this subsection, we show the local well-posedness of  \eqref{model} in $H^{s_1}\times H^{s_2}$, $4/3<s_1+1<s_2\leq\max\{4s_1, s_1+2\}$.
	
	The main problem comes from the term $\partial_x(|u|^2)$. Thus we use normal form argument to separate the high modulation part from this term.
	The integral equation of $v$ is 
	$$v(t) = K(t)v_0+\mathscr{B}(|u|^2)(t)-\mathscr{B}(v^2)(t)/2.$$
	Then,
	\begin{align*}
		\mathscr{B}(|u|^2)(t) &= \sum_{N_1,N_2,N}\mathscr{B}(P_N(P_{N_1}uP_{N_2}\bar{u}))\\
		&=\sum_{N_1\sim N_2\sim N^2}\mathscr{B}(P_N(P_{N_1}uP_{N_2}\bar{u})) + \sum_{N_1\nsim N^2}\mathscr{B}(P_N(P_{N_1}uP_{N_2}\bar{u}))\\
		&:=R(u)(t)+N(u)(t).
	\end{align*}
	Define bilinear operator $T_{M}$ and $\mathscr{B}$ by
	$$T_{M}(f,g)(x) = \mathscr{F}_\xi^{-1}\left(\int_{\xi_1+\xi_2 = \xi} M(\xi,\xi_1)\hat{f}(\xi_1)\hat{\bar{g}}(\xi_2)\mathrm{d}{\xi_1}\right)$$
		$$\mathscr{B}_{M}(u,v)(x) = \partial_x \mathscr{F}_\xi^{-1}\left(\int_{\xi_1+\xi_2 = \xi} \int_0^{t}M(\xi,\xi_1)\hat{u}(t^\prime,\xi_1)\hat{\bar{v}}(t^\prime, \xi_2)\mathrm{d}{t^\prime}\mathrm{d}{\xi_1}\right)$$
	where $ \xi_2 = \xi-\xi_1$ and ($\psi_1 = \varphi$)
	$$M(\xi,\xi_1) := \frac{1}{(2\pi)^{1/2}}\sum_{N_1\nsim N^2}\frac{\psi_{N_1}(\xi_1)\psi_N(\xi)}{\xi^3-\xi_1^2+\xi_2^2}.$$
	Integrating by parts and using the system \eqref{model}, we have
	\begin{align*}
		&\quad N(u)\\
		& = \sum_{N_1\nsim N^2}\frac{1}{(2\pi)^{1/2}} \mathscr{F}^{-1}_\xi \int_{\mathbb{R}}\int_0^te^{-it\xi^3}\xi\psi_N(\xi)\psi_{N_1}(\xi_1)\\
		&\quad\cdot e^{it'\xi_1^2}\hat{u}(t',\xi_1)e^{-it'(\xi_1-\xi)^2}\bar{\hat{u}}(t',\xi_1-\xi)~\frac{de^{it'(\xi^3-\xi_1^2+(\xi_1-\xi)^2)}}{i(\xi^3-\xi_1^2+(\xi_1-\xi)^2)}~d\xi_1\\
		&= -\partial_xT_M(u,u)+\partial_x T_M(u_0,u_0)-i\mathscr{B}(uv+|u|^2u, \bar{u})+i\mathscr{B}(u,\overline{uv+|u|^2u}).
	\end{align*}
	Let 
	\begin{align*}
		&B(u_0):= -\partial_xT_M(u_0,u_0),~~
		C(u,v)(t):=i\mathscr{B}(u,\overline{uv})-i\mathscr{B}(uv,\bar{u}),\\
		&D(u)(t):=i\mathscr{B}(u,|u|^2\bar{u})-i\mathscr{B}(|u|^2u, \bar{u})
	\end{align*}
	and $w(t,x) = v(t,x)-B(u(t))+B(u_0)$. Now system for $(u,w)$ is
	\begin{equation}\label{modifiedeq}
		\left\{
		\begin{aligned}
			u(t) &= S(t)u_0-i\mathscr{A}(uv+|u|^2u)(t),\\
			w(t) &= K(t)v_0+D(u)(t)+R(u)(t)+C(u,v)(t)-\mathscr{B}(v^2/2)(t).
		\end{aligned}
		\right.
	\end{equation}
	We use the following work spaces to solve $(u,w)$.
	\begin{align*}
		\|u\|_{X^{s_1}_T}:=\|J^{s_1}u\|_{U^2_{S,T}},\quad
		\|w\|_{Y^{s_2}_T}:= \|w\|_{C([0,T];H^{s_2})}+\|w\|_{L_x^2L_T^\infty}
	\end{align*}
	where $U^2_{S,T}$ is the space $U^2_S$ restricted on $[0,T]$.
	To establish the local well-posedness, we need the following classical estimates.
	\begin{lemma}[\cite{kenig1991oscillatory}]\label{maximalesti}
		$\|S(t)u_0\|_{L^2}=\|u_0\|_{L^2}$. For $s>1/2$, we have
		\begin{align*}
			\|S(t)u_0\|_{L_x^2L_T^\infty}&\lesssim \langle T\rangle^{1/2}\|u_0\|_{H^s}.
		\end{align*}
		Also, $\|K(t)v_0\|_{L^2}= \|v_0\|_{L^2}$, $\|\partial_x K(t)v_0\|_{L_x^\infty L_t^2}\lesssim \|v_0\|_{L^2}$. For $s>3/4$, 
		we have
		\begin{align*}
			\|K(t)v_0\|_{L_x^2L_T^\infty}&\lesssim \langle T\rangle^{1/2}\|v_0\|_{H^s}.
		\end{align*}
	\end{lemma}
	By duality and Christ-Kiselev lemma, one has the following lemma.
	\begin{lemma}[\cite{molinet2004well}]\label{dual2}
		For $s>3/4$, $T>0$, we have
		\begin{align*}
			\left\|\mathscr{B}(f)\right\|_{C([0,T];H^s)}\lesssim \|J^sf\|_{L_x^1L_T^2},\quad \left\|\mathscr{B}(f)\right\|_{L_x^2L_T^\infty}\lesssim \langle T\rangle^{1/2}\|J^sf\|_{L_x^1L_T^2}.
		\end{align*}
	\end{lemma}
	We also need the following Leibniz-type estimate.
	\begin{lemma}[Theorem 4 in \cite{benea2016multiple}]\label{leibnitz}
		Let $D^s:= \mathscr{F}^{-1}|\xi|^s\mathscr{F}$, $J^s:= \mathscr{F}^{-1}\langle\xi\rangle^s\mathscr{F}$. For $s\geq 0$, $1<q_1,r_1,q_2,r_2\leq \infty$, $1/q_1+1/q_2 = 1/q$, $1/r_1+1/r_2 = 1/r$, $1\leq q,r<\infty$, we have
		\begin{align*}
			\|D^s(uv)\|_{L_x^rL_t^q}&\lesssim \|D^su\|_{L_x^{r_1}L_t^{q_1}}\|v\|_{L_x^{r_2}L_t^{q_2}}+\|D^sv\|_{L_x^{r_1}L_t^{q_1}}\|u\|_{L_x^{r_2}L_t^{q_2}},\\
			\|J^s(uv)\|_{L_x^rL_t^q}&\lesssim \|J^su\|_{L_x^{r_1}L_t^{q_1}}\|v\|_{L_x^{r_2}L_t^{q_2}}+\|J^sv\|_{L_x^{r_1}L_t^{q_1}}\|u\|_{L_x^{r_2}L_t^{q_2}}.
		\end{align*}
	\end{lemma}
	\subsubsection{Multi-linear estimates}
	Firstly, we show the control of boundary term.
	\begin{lemma}\label{boundaryterm}
		Let $s_1\geq 0$, $s_2\leq s_1+2$, $s_2<2s_1+3/2$. Then $
		\|B(u_0)\|_{H^{s_2}}\lesssim \|u_0\|_{H^{s_1}}^2$.
	\end{lemma}
	\begin{proof}[\textbf{Proof}]
		By the definition of $B(u_0)$, we have
		\begin{align*}
			\|B(u_0)\|_{H^{s_2}}\sim\left\|\langle \xi\rangle^{s_2}\sum_{N_1\nsim N^2} \int_{\mathbb{R}}\frac{\psi_N(\xi)\psi_{N_1}(\xi_1)}{\xi^2+\xi-2\xi_1}\hat{u}_0(\xi_1)\bar{\hat{u}}_0(\xi_1-\xi)~d\xi_1\right\|_{L^2_{\xi}}.
		\end{align*}
		Since $N_1\nsim N^2$, we have $|\xi^2+\xi-2\xi|\sim \max\{N_1,N^2\}$ for $|\xi|,|\xi_1|$ in the support of $\psi_N, \psi_{N_1}$. Thus, we have
		\begin{align*}
			\|B(u_0)\|_{H^{s_2}}&\lesssim \left\|\sum_{N_1\nsim N^2} \int_{\mathbb{R}}\frac{N^{s_2}\psi_N(\xi)\psi_{N_1}(\xi_1)}{\max\{N^2,N_1\}}|\hat{u}_0(\xi_1)||{\hat{u}}_0(\xi_1-\xi)|~d\xi_1\right\|_{L^2_{\xi}}\\
			&\lesssim \left\| \int_{\mathbb{R}}\langle \xi\rangle^{s_2-2}|\hat{u}_0(\xi_1)||{\hat{u}}_0(\xi_1-\xi)|~d\xi_1\right\|_{L^2_{\xi}}\\
			&\sim \||\mathscr{F}^{-1}(|\hat{u}_0|)|^2\|_{H^{s_2-2}}.
		\end{align*}
		Then, by the Sobolev multiply estimate (\cite{tao2001multilinear}, page 855), for $s_1\geq 0$, $s_2-2\leq s_1$, $s_2-2<2s_1-1/2$, we have
		$\|B(u_0)\|_{H^{s_2}}\lesssim \|u_0\|_{H^{s_1}}^2$.
	\end{proof}
	\begin{lemma}\label{vv}
		Let $s>3/4$, $0<T\leq 1$. Then,
		\begin{align*}
			\|\mathscr{B}(v_1v_2)\|_{Y^s_T}\lesssim T^{1/2}\|v_1\|_{Y^{s}_T}\|v_2\|_{Y^s_T}.
		\end{align*}
	\end{lemma}
	\begin{lemma}\label{boundmaxi}
		$			\|B(u(t))\|_{L_x^2L_t^\infty}\lesssim \|u\|_{L_x^4L_t^\infty}^2$.
	\end{lemma}
	\begin{proof}[\textbf{Proof}]
		By the definition of $B(u(t))$, we have
		\begin{align*}
			|B(u(t))(x)|&\sim \Bigg|\int_{\mathbb{R}^2_{y,z}}u(t,y)\bar{u}(t,z)\\
			&\qquad\cdot\int_{\mathbb{R}^2_{\xi_1,\xi}}\sum_{N_1\nsim N^2}\frac{\psi_N(\xi)\psi_{N_1}(\xi_1)}{\xi^2+\xi-2\xi_1}e^{i(z-y)\xi_1+i(x-z)\xi}~d\xi_1d\xi dydz\Bigg|.
		\end{align*}
		Denote 
		$$\int_{\mathbb{R}^2_{\xi_1,\xi}}\sum_{N_1\nsim N^2}\frac{\psi_N(\xi)\psi_{N_1}(\xi_1)}{\xi^2+\xi-2\xi_1}e^{i(z-y)\xi_1+i(x-z)\xi}~d\xi_1d\xi$$
		by $\varPhi(z-y,x-z)$.  We will show that there exists $h_1,h_2\in L^1$ such that $|\varPhi(x,y)|\lesssim h_1(x)h_2(y)$. Assume this. Then one has
		\begin{align*}
			\|B(u(t))\|_{L_x^2L_t^\infty}&\lesssim \left\|\int_{\mathbb{R}^2_{y,z}}\|u(t,y)\|_{L_t^\infty}\|u(t,z)\|_{L_t^\infty}h_1(z-y)h_2(x-z)~dydz\right\|_{L_x^2}\\
			&\lesssim \left\|\|u(t,z)\|_{L_t^\infty}\int_{\mathbb{R}_y}\|u(t,y)\|_{L_t^\infty}h_1(z-y)~dy\right\|_{L_z^2}\|h_2\|_{L^1}\\
			&\lesssim \|u\|_{L_x^4L_t^\infty}\left\|\int_{\mathbb{R}_y}\|u(t,y)\|_{L_t^\infty}h_1(z-y)~dy\right\|_{L_z^4}\|h_2\|_{L^1}\\
			&\lesssim \|h_1\|_{L^1}\|h_2\|_{L^1}\|u\|_{L_x^4L_t^\infty}^2.
		\end{align*}
		Now, we estimate $\varPhi(x,y)$. If $|x|\geq 1$, we have
		\begin{align*}
			\varPhi(x,y) &= \int_{\mathbb{R}^2_{\xi_1,\xi}}\sum_{N_1\nsim N^2}\frac{\psi_N(\xi)\psi_{N_1}(\xi_1)}{\xi^2+\xi-2\xi_1}e^{ix\xi_1+iy\xi}~d\xi_1d\xi\\
			&=-\frac{1}{x^2} \int_{\mathbb{R}^2_{\xi_1,\xi}}\sum_{N_1\nsim N^2}\partial_{\xi_1}^2\left(\frac{\psi_N(\xi)\psi_{N_1}(\xi_1)}{\xi^2+\xi-2\xi_1}\right)e^{ix\xi_1+iy\xi}~d\xi_1d\xi.
		\end{align*}
		Thus, 
		\begin{align*}
			|\varPhi(x,y)|&\lesssim x^{-2}\int_{\mathbb{R}^2}\sum_{N_1,N} \frac{\varphi_N(\xi)}{\max\{N^2,N_1\}N_1^2}\chi_{|\xi_1|\lesssim N_1}~d\xi_1d\xi\lesssim x^{-2}.
		\end{align*}
		If $|y|\geq 1$, $|x|\leq 1$, we only consider the case $N_1\gg N^2$. The argument for $N_1\ll N^2$ is simpler. Integrating by parts, we have
		\begin{align*}
			&\quad\left| \int_{\mathbb{R}^2_{\xi_1,\xi}}\sum_{N_1\gg N^2}\frac{\psi_N(\xi)\psi_{N_1}(\xi_1)}{\xi^2+\xi-2\xi_1}e^{ix\xi_1+iy\xi}~d\xi_1d\xi\right|\\
			&=y^{-2}\left| \int_{\mathbb{R}^2_{\xi_1,\xi}}\sum_{N_1\gg N^2}\partial_{\xi}^2\left(\frac{\psi_N(\xi)\psi_{N_1}(\xi_1)}{\xi^2+\xi-2\xi_1}\right)e^{ix\xi_1+iy\xi}~d\xi_1d\xi\right|\\
			&\lesssim y^{-2}\sum_N \int_{\mathbb{R}_\xi}\left|\int_{\mathbb{R}_{\xi_1}}\sum_{N_1\gg N^2}\partial_{\xi}^2\left(\frac{\psi_N(\xi)\psi_{N_1}(\xi_1)}{\xi^2+\xi-2\xi_1}\right)e^{ix\xi_1}~d\xi_1\right|d\xi\\
			&:=y^{-2}M(x).
		\end{align*}
		By Minkowski inequality and Plancherel identity, we have
		\begin{align*}
			\|M\|_{L^2_x}&\lesssim \sum_N\int_{\mathbb{R}_{\xi}}\left\|\sum_{N_1\gg N^2}\partial_{\xi}^2\left(\frac{\psi_N(\xi)\psi_{N_1}(\xi_1)}{\xi^2+\xi-2\xi_1}\right)\right\|_{L^2_{\xi_1}}~d\xi\\
			&\lesssim \sum_N \int_{\mathbb{R}_\xi} \frac{\chi_{|\xi|\lesssim N}}{N^2}~d\xi\lesssim 1.
		\end{align*}
		By similar argument for $|x|\geq 1, |y|\geq 1$, one has $|\varPhi(x,y)|\lesssim |x|^{-2}|y|^{-2}$. 
		
		For $|x|\leq 1$, $|y|\leq 1$, 
		we have
		\begin{align*}
			|\varPhi(x,y)| &=\left|\frac{1}{x} \int_{\mathbb{R}^2_{\xi_1,\xi}}\sum_{N_1\nsim N^2}\partial_{\xi_1}\left(\frac{\psi_N(\xi)\psi_{N_1}(\xi_1)}{\xi^2+\xi-2\xi_1}\right)e^{ix\xi_1+iy\xi}~d\xi_1d\xi\right|\\
			&\lesssim |x|^{-1}\int_{\mathbb{R}^2}\sum_{N_1,N} \frac{\psi_N(\xi)}{\max\{N^2,N_1\}N_1}\chi_{|\xi_1|\lesssim N_1}~d\xi_1d\xi\lesssim |x|^{-1}.
		\end{align*}
		Note that by the former argument, we also have $|\varPhi(x,y)|\leq y^{-2}M(x)$ where $M\in L^2$. $|\varPhi(x,y)|\lesssim |x|^{-3/4}M(x)^{1/4}|y|^{-1/2}$. Note that by H\"{o}lder inequality, we have $|x|^{-3/4}M(x)^{1/4}\chi_{|x|\leq1}\in L^1$. Let 
		$$h_1(x) = \langle x\rangle^{-2}+(|x|^{-3/4}M(x)^{1/4}+M(x))\chi_{|x|\leq 1}, h_2(y) = |y|^{-1/2}\chi_{|y|\leq 1}+\langle y\rangle^{-2}.$$
		We have $|\varPhi(x,y)|\lesssim h_1(x)h_2(y)$, $h_1,h_2\in L^1$. We finish the proof.
	\end{proof}
	\begin{lemma}\label{threefour}
		Let $0<T\leq 1$, $s_1>0$, $3/4<s_2<2s_1+3/2$, $s_2\leq s_1+2$. Then,
		\begin{align*}
			\|C(u,v)\|_{Y^{s_2}_T}\lesssim T^{3/4} \|u\|_{X_T^{s_1}}^2\|v\|_{Y^{s_2}_T},\quad \|D(u)\|_{Y^{s_2}_T}\lesssim T^{1/2}\|u\|_{X^{s_1}_T}^4.
		\end{align*}
	\end{lemma}
	\begin{proof}[\textbf{Proof}]
		By Plancherel identity and maximal function estimate, we have ($s_2>3/4$)
		\begin{align*}
			&\quad\|C(u,v)\|_{Y^{s_2}_T}\\
			&\lesssim \int_0^T\Bigg\|\int_{\mathbb{R}}\sum_{N_1\nsim N^2}\frac{N^{s_2}\psi_N(\xi)\psi_{N_1}(\xi_1)}{\max\{N^2,N_1\}} |\widehat{uv}|(t',\xi_1)|\hat{u}|(t',\xi_1-\xi)~d\xi_1\Bigg\|_{L^2_\xi}dt'.
		\end{align*}
		If $s_1>1/2$, $s_2\leq s_1+2$, we have
		\begin{align*}
			&\quad \Bigg\|\int_{\mathbb{R}}\sum_{N_1\nsim N^2}\frac{N^{s_2}\psi_N(\xi)\psi_{N_1}(\xi_1)}{\max\{N^2,N_1\}} |\widehat{uv}|(t',\xi_1)|\hat{u}|(t',\xi_1-\xi)~d\xi_1\Bigg\|_{L^2_\xi}\\
			&\lesssim \Bigg\|\int_{\mathbb{R}}\sum_{N_1, N}\frac{N^{s_2-2}\psi_N(\xi)\psi_{N_1}(\xi_1)}{N^{s_1}}\\
			&\qquad\cdot(\langle\xi_1\rangle^{s_1}+\langle\xi_1-\xi\rangle^{s_1}) |\widehat{uv}|(t',\xi_1)|\hat{u}|(t',\xi_1-\xi)~d\xi_1\Bigg\|_{L^2_\xi}\\
			&\lesssim \Bigg\|\int_{\mathbb{R}}(\langle\xi_1\rangle^{s_1}+\langle\xi_1-\xi\rangle^{s_1}) |\widehat{uv}|(t',\xi_1)|\hat{u}|(t',\xi_1-\xi)~d\xi_1\Bigg\|_{L^2_\xi}\\
			&\lesssim \|uv\|_{H^{s_1}}\|\hat{u}\|_{L^1}+\|\widehat{uv}\|_{L^1}\|u\|_{H^{s_1}}\lesssim \|u\|^2_{H^{s_1}}\|v\|_{H^{s_2}}.
		\end{align*}
		If $0<s_1\leq 1/2$, $1/2<s_2< 2s_1+3/2$, for $N_1\nsim N$, we have
		\begin{align*}
			&\quad \Bigg\|\int_{\mathbb{R}}\sum_{N_1\nsim N}\frac{N^{s_2}\psi_N(\xi)\psi_{N_1}(\xi_1)}{\max\{N^2,N_1\}} |\widehat{uv}|(t',\xi_1)|\hat{u}|(t',\xi_1-\xi)~d\xi_1\Bigg\|_{L^2_\xi}\\
			&\lesssim \Bigg\|\int_{\mathbb{R}}\sum_{N_1, N}\frac{N^{s_2-2}\psi_N(\xi)\psi_{N_1}(\xi_1)}{N_1^{s_1}N^{s_1}}\\
			&\qquad\cdot\langle\xi_1\rangle^{s_1}\langle\xi_1-\xi\rangle^{s_1} |\widehat{uv}|(t',\xi_1)|\hat{u}|(t',\xi_1-\xi)~d\xi_1\Bigg\|_{L^2_\xi}\\
			&\lesssim \sum_{N_1, N}\frac{N^{s_2-2-s_1}}{N_1^{s_1}}\min\{N,N_1\}^{1/2}\|uv\|_{H^{s_1}}\|u\|_{H^{s_1}}\\
			&\lesssim \|u\|_{H^{s_1}}\|v\|_{H^{s_2}}(\|u\|_{L^\infty}+\|u\|_{H^{s_1}}).
		\end{align*}
		For $N_1\sim N$, $s_2< 2s_1+3/2$, we have
		\begin{align*}
			&\quad \Bigg\|\int_{\mathbb{R}}\sum_{N_1\sim N}\frac{N^{s_2}\psi_N(\xi)\psi_{N_1}(\xi_1)}{\max\{N^2,N_1\}} |\widehat{uv}|(t',\xi_1)|\hat{u}|(t',\xi_1-\xi)~d\xi_1\Bigg\|_{L^2_\xi}\\
			&\lesssim \Bigg\|\sum_{N_1\sim N}N^{s_2-2}\sum_{M\lesssim N}M^{-s_1+{1/2}}\psi_N(\xi)\|\psi_{N_1}(\xi_1)|\widehat{uv}|(t',\xi_1)\|_{L^2}\|u\|_{H^{s_1}}\Bigg\|_{L^2_\xi}\\
			&\lesssim\Bigg\|\sum_{N_1\sim N}\psi_N(\xi)\|\langle\xi_1\rangle^{s_1}\psi_{N_1}(\xi_1)|\widehat{uv}|(t',\xi_1)\|_{L^2}\|u\|_{H^{s_1}}\Bigg\|_{L^2}\\
			&\lesssim \|uv\|_{H^{s_1}}\|u\|_{H^{s_1}}\lesssim \|u\|_{H^{s_1}}\|v\|_{H^{s_2}}(\|u\|_{L^\infty}+\|u\|_{H^{s_1}}).
		\end{align*}
		Overall, we obtain
		\begin{align*}
			\|C(u,v)\|_{Y^{s_2}_T}&\lesssim \int_0^T\|u(t')\|_{H^{s_1}}\|v(t')\|_{H^{s_2}}(\|u(t')\|_{L^\infty}+\|u(t')\|_{H^{s_1}})~dt'\\
			&\lesssim T\|u\|_{X^{s_1}_T}^2\|v\|_{Y^{s_2}_T}+T^{3/4}\|u\|_{X^{s_1}_T}\|u\|_{L_T^4L_x^\infty}\|v\|_{Y^{s_2}_T}\\
			&\lesssim T^{3/4}\|u\|_{X^{s_1}_T}^2\|v\|_{Y^{s_2}_T} .
		\end{align*}
		By the same argument, we have
		\begin{align*}
			\|D(u)\|_{Y^{s_2}_T}&\lesssim \int_0^T \||u|^2u(t')\|_{H^{s_1}}\|u(t')\|_{H^{s_1}}~dt'\\
			&\lesssim \int_0^T \|u(t')\|_{H^{s_1}}^2\|u(t')\|_{L^\infty}^2~dt'\lesssim T^{1/2}\|u\|_{X^{s_1}_T}^4.
		\end{align*}
		We finish the proof of this lemma.
	\end{proof}
	\begin{lemma}\label{resounuuv}
		Let $0<T<1$, $5/4<s_2\leq 4s_1$. Then we have
		$$\|R(u)\|_{Y^{s_2}_T}\lesssim T^{1/2} \|u\|_{X^{s_1}_T}^2.$$
	\end{lemma}
	\begin{proof}[\textbf{Proof}]
		By Lemma \ref{frequecylocsbilinear}, we have
		\begin{align*}
			\|R(u)\|_{C([0,T];H^{s_2})}\lesssim \|J^{s_2}R(u)\|_{V^2_{K,T}}\lesssim T^{1/2} \|J^{s_1}u\|_{U^2_{S,T}}^2\lesssim T^{1/2}\|u\|_{X^{s_1}_T}^2.
		\end{align*}
		By general extension result Proposition 2.16 in \cite{hadac2009well}, we only need to show
		\begin{align*}
			\|R(S(t)J^{-s_1}u_0)\|_{L_x^2L_T^\infty}\lesssim T^{1/2}\|u_0\|_{L^2}^2.
		\end{align*} 			
		Since $\|v\|_{L_{x}^2L_T^\infty}\lesssim \|J^{\tilde{s}} v\|_{U^2_{K,T}}$, $\forall~\tilde{s}>3/4$, $0<T<1$, then for any $\epsilon>0$ we have
		\begin{align*}
			&\|R(S(t)J^{-s}u_0)\|_{L_x^2L_T^\infty}\lesssim \sum_{N\sim 1}\|P_NR(S(t)J^{-s_1}u_0)\|_{L_x^2L_T^\infty}\\
			&\quad + \sum_{N_1\sim N_2\sim N^2\gg 1} N^{3/4+\epsilon}\|\mathscr{B}(P_N(P_{N_1}S(t)J^{-s_1}u_0P_{N_2}S(t)J^{-s_1}\bar{u}_0))\|_{U^2_{K,T}}.
		\end{align*}
		For the first term, it is easy to obtain
		$$\sum_{N\sim 1}\|P_NR(S(t)J^{-s}u_0)\|_{L_x^2L_T^\infty}\lesssim T \|u_0\|_{L^2}^2.$$
		For the second term, by duality we need to estimate
		\begin{align*}
			\int_0^T\int_{\mathbb{R}}\partial_x(P_{N_1}S(t)J^{-s_1}u_0P_{N_2}S(t)J^{-s_1}\bar{u}_0)P_Nv~dxdt.
		\end{align*}
		By Cauchy-Schwarz inequality and \eqref{sstran}, the upper term is controlled by
		\begin{align*}
			&\quad\|\tilde{P}_N(P_{N_1}S(t)J^{-s_1}u_0P_{N_2}S(t)J^{-s_1}\bar{u}_0)\|_{L_{x,T}^2}\|\partial_x P_N v\|_{L^2_{x,T}}\\
			&\lesssim N^{-1/2}\|J^{-s_1}P_{N_1}u_0\|_{L^2}\|J^{-s_1}P_{N_2}u_0\|_{L^2}N\|P_Nv\|_{L_{x,T}^2}\\
			&\lesssim N^{-4s_1+1/2}T^{1/2}\|P_{N_1}u_0\|_{L^2}\|P_{N_2}u_0\|_{L^2}\|v\|_{V^2_K}.
		\end{align*}
		Thus one has
		\begin{align*}
			&\quad\|RS(t)J^{-s}u_0\|_{L_x^2L_T^\infty}\\
			&\lesssim \sum_{N_1\sim N_2}N^{-4s_1+5/4+\epsilon}T^{1/2}\|P_{N_1}u_0\|_{L^2}\|P_{N_2}u_0\|_{L^2}+T\|u_0\|_{L^2}^2\lesssim T^{1/2}\|u_0\|_{L^2}^2.
		\end{align*}
		The last inequality we use the condition $4s_1>5/4$.
	\end{proof}
	\begin{lemma}\label{uesti}
		Let $0<T<1$, $s_2\geq s_1\geq 1/4$. We have
		\begin{align*}
			\|\mathscr{A}(uv)\|_{X^{s_1}_T}\lesssim T^{3/4}\|u\|_{X^{s_1}_T}\|v\|_{Y^{s_2}_T},\quad \|\mathscr{A}(|u|^2u)\|_{X^{s_1}_T}\lesssim T^{1/2}\|u\|_{X^{s_1}_T}^3.
		\end{align*}
	\end{lemma}
	\begin{proof}[\textbf{Proof}]
		By Lemma \ref{multilinearoutsch}, one has $\|\mathscr{A}(|u|^2u)\|_{X^{s_1}_T}\lesssim T^{1/2}\|u\|_{X^{s_1}_T}^3$. We also have
		\begin{align*}
			\|\mathscr{A}(uv)\|_{X^{s_1}_T}&\lesssim T^{3/4}\|J^{s_1}u\|_{L_T^{8}L^4_x}\|J^{s_1}v\|_{L_T^\infty L^2_x}\\
			&\lesssim T^{3/4}\|J^{s_1}u\|_{U^2_{S,T}}\|v\|_{Y^{s_1}_T}\\
			&\lesssim T^{3/4}\|u\|_{X^{s_1}_T}\|v\|_{Y^{s_1}_T}.
		\end{align*}
		We conclude the proof.
	\end{proof}
	\begin{prop}
		Let $s_1>5/16$, $s_1+1<s_2\leq \max\{4s_1,s_1+2\}$. \eqref{model} is local well-posed in $H^{s_1}\times H^{s_2}$.
	\end{prop}
	\begin{proof}[\textbf{Proof}]
		Consider the mapping $\mathcal{T}$. Let $v(t) = w(t)+B(u(t))-B(u_0)$. Define
		\begin{equation*}
			\mathcal{T}:
			\begin{pmatrix}
				u\\
				w
			\end{pmatrix}
			\mapsto
			\begin{pmatrix}
				S(t)u_0-i\mathscr{A}(uv+|u|^2u)(t)\\ K(t)v_0+D(u)(t)+R(u)(t)+C(u,v)(t)-\mathscr{B}(v^2/2)(t)
			\end{pmatrix}.
		\end{equation*}
		Let $0<T<1$. By Lemmas \ref{boundaryterm}--\ref{uesti}, we have
		\begin{align*}
			\|\mathcal{T}(u,w)\|_{X^{s_1}_T\times Y^{s_2}_T}&\lesssim \|u_0\|_{H^{s_1}}+\|v_0\|_{H^{s_2}}+T^{3/4}(\|u\|_{X^{s_1}_T}\\
			&\quad+\|u\|_{X^{s_1}_T}^2)(\|w\|_{Y^{s_2}_T}+\|u\|_{X^{s_1}_T}^2)\\
			&\quad +T^{1/2}(\|u\|_{X^{s_1}_T}^2+\|u\|_{X^{s_1}_T}^3+\|u\|_{X^{s_1}_T}^4+\|w\|_{X^{s_1}_T}^2).
		\end{align*}
		Similarly one has
		\begin{align*}
			&\quad\|\mathcal{T}(u_1,w_1)-\mathcal{T}(u_2,w_2)\|_{X^{s_1}_T\times Y^{s_2}_T}\\
			&\lesssim T^{1/2} \|(u_1-u_2,w_1-w_2)\|_{X^{s_1}_T\times Y^{s_2}_T}\\
			&\qquad\cdot\left(1+\|(u_1,w_1)\|^4_{X^{s_1}_T\times Y^{s_2}_T}+\|(u_2,w_2)\|^4_{X^{s_1}_T\times Y^{s_2}_T}\right)
		\end{align*}
		Thus, for sufficiently small $T$, we obtain a solution of \eqref{modifiedeq}. Then, $(u,w+B(u)-B(u_0))\in C([0,T];H^{s_1})\times C([0,T];H^{s_2})$ is the solution of initial equation \eqref{model} with initial data $(u_0,v_0)$.
	\end{proof}
	\subsection{Lower region}\label{lowerreg}
	In this subsection, we show the local well-posedness of \eqref{model} with $5/4<s_1\leq 9/4$, $s_2 = -3/4$. 
	Consider the equation \eqref{scalingeq} with $\|v_0\|_{H^{-3/4}}\ll 1$, $0<\lambda\ll 1$. 
	
	By Lemma \ref{s2-3/4}, Proposition \ref{wellpose-3/4}, we have a solution of \eqref{scalingeq} $(u,v) \in \tilde{X}_{\lambda,T}^{1,b}\times F_T$ with initial data $(u_0,v_0)\in H^{s_1}\times H^{-3/4}$.
	
	The main problem comes from the term $uv$. We should gain $3$-order derivative for high modulation cases. To achieve this, we use norm form argument. 
	
	Define	bilinear operator $T_{M_\lambda}$ by
	$$T_{M_\lambda}(f,g)(x) = \mathscr{F}_\xi^{-1}\left(\int_{\xi_1+\xi_2 = \xi} M_\lambda(\xi,\xi_1)\hat{f}(\xi_1)\hat{g}(\xi_2)\right)$$
	where $ \xi_2 = \xi-\xi_1$ and
	$$M_\lambda(\xi,\xi_1) := \frac{1}{(2\pi)^{1/2}}\sum_{N\gg 1}\frac{\varphi_{N/8}(\xi_1)\psi_N(\xi)}{\lambda\xi^2-\lambda\xi_1^2-\xi_2^3}.$$
	Recall some notations in Subsection {\ref{borderlinecaselow}}. Integration by parts we have
	\begin{lemma}\label{reguimprolow}
		Let $(u,v)\in \tilde{X}_{\lambda,T}^{1,b}\times F_T$ be a solution of \eqref{scalingeq}. Then for $0\leq t\leq T$, one has
		\begin{align*}
			\mathscr{A}_\lambda(uv)& = B(u(t),v(t))-B(u_0,v_0)+ \lambda C_1(u,v)\\
			&\quad+\lambda^{-1} D_1(u,v)+C_2(u)+C_3(u,v)+R(u,v)
		\end{align*}
		where
		\begin{align*}
			&B(u_0,v_0)
			=-iT_{M_{\lambda}}(u_0,v_0),\\
			&C_1(u,v)=\mathscr{A}_\lambda(T_{M_\lambda}(uv,v)),~~
			C_2(u)=-\mathscr{A}_\lambda(u,\partial_x (|u|^2)),\\
			&C_3(u,v)=-\mathscr{A}_\lambda(u,v\partial_x v),~~
			D_1(u,v)=\mathscr{A}_\lambda(T_{M_\lambda}(|u|^2u,v)),\\
			&R(u,v)=\mathscr{A}_\lambda\left(\sum_{N}P_{N}(P_{>N/8}uv)\right).
		\end{align*}
	\end{lemma}
	Let $\tilde{u}(t)=u(t)-\lambda B(u(t),v(t))$. Then $\tilde{u}$ satisfies the integral equation
	\begin{equation}\label{eqfornorm}
		\begin{aligned}
			\tilde{u}(t)&=u(0)-\lambda B(u(0),v(0)) +\lambda R(\tilde{u}+\lambda B(u,v),v)\\
			&\quad+\lambda^2C_1(u,v)+D_1(u,v)+\lambda C_2(u)+\lambda C_3(u,v)\\
			&\quad+\lambda^{-1}\int_0^t S_\lambda(t-t')(|\tilde{u}+\lambda B(u,v)|^2(\tilde{u}+\lambda B(u,v)))~dt'.
		\end{aligned}
	\end{equation}
	We will solve $\tilde{u}$ by using Bourgain space $\tilde{X}^{s_1,\tilde{b}}_{\lambda,T}$ for some $\tilde{b}>1/2$.
	\begin{lemma}\label{boundterm}
		Let $s_0>1/2$. Then
		$\|B(u_0,v_0)\|_{H^s}\lesssim \|u_0\|_{H^{s_0}}\|v_0\|_{H^{s-3}}$.
	\end{lemma}
	\begin{proof}[\textbf{Proof}]
		By the definition of $M_\lambda$ we have
		\begin{align*}
			\|B(u_0,v_0)\|_{H^s}&\lesssim \left\|N^{s-3}\int_{\mathbb{R}_{\xi_1}}\varphi_N(\xi_1)|\widehat{u_0}(\xi_1)|\psi_N(\xi)|\widehat{v_0}(\xi-\xi_1)|~d\xi_1\right\|_{l^2_{N\gg 1}}\\
			&\lesssim \|u_0\|_{\mathscr{F}L^1}\|v_0\|_{H^{s-3}}\lesssim \|u_0\|_{H^{s_0}}\|v_0\|_{H^{s-3}}.
		\end{align*}
		We finish the proof.
	\end{proof}
	\begin{lemma}\label{reguimpr}
		Let $(u,v) \in \tilde{X}_{\lambda,T}^{1,b}\times F_T$ be a solution of \eqref{scalingeq}. Then there exists $\tilde{b}>1/2$ such that
		$$C_1(u,v),D_1(u,v),C_2(u),C_3(u,v),R(B(u,v),v)\in \tilde{X}^{9/4,\tilde{b}}_{\lambda,T}.$$
	\end{lemma}
	\begin{proof}[\textbf{Proof}]
		We only need to prove
		\begin{align}
			\|T_{M_\lambda}(uv,v)\|_{\tilde{X}_{\lambda}^{9/4,\tilde{b}-1}}&\lesssim \lambda^{-1/6}\|u\|_{\tilde{X}_{\lambda}^{1,b}}\|v\|_{F}^2,\label{uvv}\\
			\|T_{M_\lambda}(|u|^2u,v)\|_{\tilde{X}_{\lambda}^{9/4,\tilde{b}-1}}&\lesssim \|u\|_{\tilde{X}^{1,b}_{\lambda}}^3\|v\|_{F},\label{uuuv}\\
			\|T_{M_\lambda}(u,\partial_x(|u|^2))\|_{\tilde{X}_{\lambda}^{9/4,\tilde{b}-1}}&\lesssim \|u\|_{\tilde{X}^{1,b}_{\lambda}}^3,\label{uuu}\\
			\|T_{M_\lambda}(u,v\partial_xv)\|_{\tilde{X}^{9/4,\tilde{b}-1}_{\lambda}}&\lesssim \|u\|_{\tilde{X}^{1,b}_{\lambda}}\|v\|_{F}^2\label{uvvt}
		\end{align}
		and
		\begin{equation}\label{buvv}
			\left\|\sum_N P_N(P_{>N/8}B(u,v)v)\right\|_{\tilde{X}_{\lambda}^{9/4,\tilde{b}-1}}\lesssim \|u\|_{\tilde{X}^{1,b}_{\lambda}}\|v\|_{F}^2.
		\end{equation}
		For \eqref{uvv}, due to the definition of $M_\lambda$, one has
		\begin{align*}
			\|T_{M_\lambda}(uv,v)\|_{\tilde{X}^{9/4,\tilde{b}-1}_\lambda}\sim \|N^{9/4}\|\tilde{P}_N T_{M_{\lambda}}(uv,P_{N}v)\|_{\tilde{X}^{0,\tilde{b}-1}_\lambda}\|_{l^2_{N\gg 1}}.
		\end{align*}
		For the term $T_{M_\lambda}(uP_{1}v,P_Nv)$, we have
		\begin{align*}
			&\quad \|\tilde{P}_N T_{M_{\lambda}}(uP_{1}v,P_{N}v)\|_{\tilde{X}^{0,\tilde{b}-1}_\lambda}\\
			&\lesssim \sup_{\|w\|_{L^2_{\tau,\xi}}\leq 1}\int_{\mathbb{R}^4}\frac{|M_{\lambda}(\xi,\xi_1)\widehat{uP_1 v}(\tau_1,\xi_1)||\widehat{P_Nv}(\tau-\tau_1,\xi-\xi_1)w(\tau,\xi)|}{\langle\tau+\lambda\xi^2\rangle^{1-\tilde{b}}}\\
			&\lesssim N^{-3}\sup_{\|w\|_{L^2_{\tau,\xi}\leq 1}} \|uP_{1}v\|_{L^2_{t,x}}\|\mathscr{F}^{-1}(|\widehat{P_Nv}|)\|_{L_{t,x}^6}\|\mathscr{F}^{-1}(|w|\langle\tau+\lambda\xi^2\rangle^{\tilde{b}-1})\|_{L^3_{t,x}}
		\end{align*}
		By refined Strichartz estimate for Schr\"{o}dinger equation, for $\tilde{b}\leq 3/4$ one has
		\begin{align*}
			&\quad\|\tilde{P}_N T_{M_{\lambda}}(uP_{1}v,P_{N}v)\|_{\tilde{X}^{0,\tilde{b}-1}_\lambda}\\
			&\lesssim N^{-19/6}\lambda^{-1/12}\|u\|_{L_x^\infty L_t^2}\|P_1 v\|_{L_x^2L_t^\infty}\|P_N v\|_{X^{0,1/2,1}}\\
			&\lesssim N^{-29/12}\lambda^{-1/12}\|u\|_{\tilde{X}^{1,b}_\lambda}\|v\|_{F}^2.
		\end{align*}
		Thus we obtain 
		$$
		\|N^{9/4}\|\tilde{P}_N T_{M_{\lambda}}(uP_1v,P_{N}v)\|_{\tilde{X}^{0,\tilde{b}-1}_\lambda}\|_{l^2_{N\gg 1}}\lesssim \lambda^{-1/12}\|u\|_{\tilde{X}^{1,b}_\lambda}\|v\|_{F}^2.$$
		For the term $T_{M_\lambda}(uP_Kv,P_Nv)$, $K\geq 2$, in fact we would estimate 
		$$\left\|\mathscr{F}^{-1}_{\tau,\xi}\int_{\mathbb{R}^2}|M_{\lambda}(\xi,\xi_1)\widehat{uP_K v}(\tau_1,\xi_1)||\widehat{P_Nv}(\tau-\tau_1,\xi-\xi_1)|~d\xi_1d\tau_1\right\|_{\tilde{X}^{0,\tilde{b}-1}_\lambda}.$$
		Since $\|P_N u\|_{\tilde{X}^{s,b}_{\lambda}} = \|\mathscr{F}^{-1}(|\widehat{P_Nu}|)\|_{\tilde{X}^{s,b}_{\lambda}}$, $\|P_N v\|_{Y^{s,b,q}} = \|\mathscr{F}^{-1}(|\widehat{P_Nv}|)\|_{Y^{s,b,q}}$, without of loss of generality we would assume that $\hat{u},\hat{v}\geq 0$. Thus  for $0<\epsilon<4$ one has
		\begin{align*}
			&\quad\int_{\mathbb{R}^4}\frac{|M_{\lambda}(\xi,\xi_1)\widehat{uP_K v}(\tau_1,\xi_1)||\widehat{P_Nv}(\tau-\tau_1,\xi-\xi_1)w(\tau,\xi)|}{\langle\tau+\lambda\xi^2\rangle^{1-\tilde{b}}}\\
			&\lesssim N^{-3}\|P_{\ll N}(uP_{K}v)\|_{L^{(6-\epsilon)/(5-\epsilon)}_{x}L_t^{2(6-\epsilon)/(4-\epsilon)}}\|P_N v\|_{L_x^\infty L_t^2}\\
			&\quad\cdot\|\mathscr{F}^{-1}(|w|\langle \tau+\lambda\xi^2\rangle^{\tilde{b}-1})\|_{L_{t,x}^{6-\epsilon}}.
		\end{align*}
		For $\tilde{b}\leq (12-\epsilon)/(24-4\epsilon)$, one has $$\|\mathscr{F}^{-1}(|w|\langle \tau+\lambda\xi^2\rangle^{\tilde{b}-1})\|_{L_{t,x}^{6-\epsilon}}\lesssim \lambda^{-1/6}\|w\|_{L^2_{t,x}}.$$
		Thus for $K\lesssim N$ we obtain ($\epsilon\ll 1$)
		\begin{align*}
			&\quad\|T_{M_\lambda}(uP_Kv,P_Nv)\|_{\tilde{X}^{0,\tilde{b}-1}_\lambda}\\
			&\lesssim N^{-3}\lambda^{-1/6}\|P_{\ll N}(uP_{K}v)\|_{L^{(6-\epsilon)/(5-\epsilon)}_{x}L_t^{2(6-\epsilon)/(4-\epsilon)}}\|P_N v\|_{L_x^\infty L_t^2}\\
			&\lesssim N^{-4}\lambda^{-1/6}\|u\|_{L_{x}^{2}L_t^\infty}\|P_Kv\|_{L_{t,x}^{2(6-\epsilon)/(4-\epsilon)}}\|P_Nv\|_{X^{0,1/2,1}}\\
			&\lesssim N^{-4}\lambda^{-1/6}K^{3/4}\|u\|_{\tilde{X}^{1,\tilde{b}}_\lambda}\|v\|_{F}^2.
		\end{align*}
		For $K\gg N$ one has
		\begin{align*}
			&\quad\|T_{M_\lambda}(uP_Kv,P_Nv)\|_{\tilde{X}^{0,\tilde{b}-1}_\lambda}\\
			& = \|T_{M_\lambda}(P_{\sim K}uP_Kv,P_Nv)\|_{\tilde{X}^{0,\tilde{b}-1}_\lambda}\\
			&\lesssim N^{-3}\lambda^{-1/6}\|P_{\ll N}(P_{\sim K}uP_K v)\|_{L_x^{8(6-\epsilon)/(34-7\epsilon)}L_t^{2(6-\epsilon)/(4-\epsilon)}}\|P_N v\|_{L_x^8 L_t^2}\\
			&\lesssim N^{-3}\lambda^{-1/6}\|P_{\sim K}u\|_{L_x^2 L_t^\infty}\|P_K v\|_{L_x^{8(6-\epsilon)/(10-3\epsilon)}L_t^{2(6-\epsilon)/(4-\epsilon)}}\|P_N v\|_{X^{-3/4,1/2,1}}.
		\end{align*}
		By scaling argument and $\|P_Ku\|_{L_x^2L_t^\infty}\lesssim K^{1/2} \|u\|_{X^{0,\tilde{b}}}$, we have $$\|P_{\sim K}u\|_{L_x^2 L_t^\infty}\lesssim \langle \lambda^{1/2}K\rangle^{1/2}\|P_{\sim K}u\|_{X_\lambda^{0,\tilde{b}}}\lesssim K^{-1/2}\|P_{\sim K}u\|_{X^{1,\tilde{b}}_\lambda}.$$
		By H\"{o}lder inequality one has
		\begin{align*}
			&\quad\|P_K v\|_{L_x^{8(6-\epsilon)/(10-3\epsilon)}L_t^{2(6-\epsilon)/(4-\epsilon)}}\\
			&\lesssim \|P_K v\|_{L_{t,x}^2}^{8/(3(6-\epsilon))}\|P_K v\|_{L_{t,x}^8}^{(22-9\epsilon)/(12(6-\epsilon))}\|P_K v\|_{L_{x}^\infty L_t^2}^{1/4}\\
			&\lesssim K^{-1/4}\|P_Kv\|_{X^{0,1/2,1}}.
		\end{align*}
		Thus we have
		\begin{align*}
			&\quad\|T_{M_\lambda}(uP_Kv,P_Nv)\|_{\tilde{X}^{0,\tilde{b}-1}_\lambda}\\
			&\lesssim N^{-3}\lambda^{-1/6}\|P_{\sim K}u\|_{X^{1,\tilde{b}}_\lambda}\|P_K v\|_{X^{-3/4,1/2,1}}\|P_Nv\|_{X^{-3/4,1/2,1}}.
		\end{align*}
		Thus we obtain
		\begin{align*}
			&\quad\|T_{M_\lambda }(uP_{>1}v,P_Nv)\|_{\tilde{X}^{0,\tilde{b}-1}_\lambda}\leq \sum_{K\geq 2} \|T_{M_\lambda }(uP_Kv,P_Nv)\|_{\tilde{X}^{0,\tilde{b}-1}_\lambda}\\
			&\lesssim \sum_{K\lesssim N}N^{-4}\lambda^{-1/6}K^{3/4}\|u\|_{\tilde{X}^{1,\tilde{b}}_\lambda}\|v\|_{F}\\
			&\quad +\sum_{K\gg N} N^{-3}\lambda^{-1/6}\|P_{\sim K}u\|_{X^{1,\tilde{b}}_\lambda}\|P_K v\|_{X^{-3/4,1/2,1}}\|P_Nv\|_{X^{-3/4,1/2,1}}\\
			&\lesssim N^{-13/4}\|u\|_{\tilde{X}^{1,\tilde{b}}_\lambda}\|v\|_{F}^2+N^{-3}\lambda^{-1/6}\|u\|_{\tilde{X}^{1,\tilde{b}}_\lambda}\|v\|_F\|P_Nv\|_{X^{-3/4,1/2,1}}.
		\end{align*}
		Overall one has
		\begin{align*}
			\|T_{M_\lambda}(uv,v)\|_{\tilde{X}^{9/4,\tilde{b}-1}_\lambda}\lesssim \lambda^{-1/6}\|u\|_{\tilde{X}^{1,\tilde{b}}_\lambda}\|v\|_F^2.
		\end{align*}
		For \eqref{uuuv} we have
		\begin{align*}
			\|T_{M_\lambda}(|u|^2u,v)\|_{\tilde{X}^{9/4,\tilde{b}-1}_\lambda}&\sim \|N^{9/4}\|T_{M_\lambda}(|u|^2u,P_Nv)\|_{\tilde{X}^{0,\tilde{b}-1}_\lambda}\|_{l^2_{N\gg 1}}\\
			&\lesssim \|N^{9/4}N^{-3}\|\widehat{u}\|_{L_{\tau,\xi}^1}^3\|P_N v\|_{L_{t,x}^2}\|_{l^2_{N\gg 1}}\\
			&\lesssim \|u\|_{\tilde{X}^{1,\tilde{b}}_{\lambda}}^3\|v\|_F.
		\end{align*}
		For \eqref{uuu} we have
		\begin{align*}
			\|T_{M_\lambda}(u,\partial_x (|u|^2))\|_{\tilde{X}^{9/4,\tilde{b}-1}_\lambda}&\sim \|N^{9/4}\|T_{M_\lambda}(u,P_N\partial_x (|u|^2))\|_{\tilde{X}^{0,\tilde{b}-1}_\lambda}\|_{l^2_{N\gg 1}}\\
			&\lesssim \|N^{9/4}N^{-3}N\|\widehat{u}\|_{L_{\tau,\xi}^1}\|P_N (|u|^2)\|_{L_{t,x}^2}\|_{l^2_{N\gg 1}}\\
			&\lesssim \|u\|_{\tilde{X}^{1,\tilde{b}}_{\lambda}}\||u|^2\|_{\tilde{X}_\lambda^{1,0}}\\
			&\lesssim \|u\|_{\tilde{X}^{1,\tilde{b}}_\lambda}^3.
		\end{align*}
		For \eqref{uvvt} we have
		\begin{align*}
			&\quad\|T_{M_\lambda}(u,v\partial_xv)\|_{\tilde{X}^{9/4,\tilde{b}-1}_{\lambda}}\\
			&\lesssim \|T_{M_\lambda}(u,\partial_x(P_1vP_{\sim N}v))\|_{\tilde{X}^{9/4,\tilde{b}-1}_{\lambda}}+\|T_{M_\lambda}(u,\partial_x(P_{>1}vP_{>1}v))\|_{\tilde{X}^{9/4,\tilde{b}-1}_{\lambda}}.
		\end{align*}
		For the first term one has
		\begin{align*}
			&\quad\|T_{M_\lambda}(u,\partial_x(P_1vP_{\sim N}v))\|_{\tilde{X}^{9/4,\tilde{b}-1}_{\lambda}}\\
			&\lesssim \|N^{9/4}N^{-3}\|P_1vP_{\sim N}v\|_{L_{t,x}^2}\|\hat{u}\|_{L_{\tau,\xi}^1}\|_{l^2_{N\gg 1}}\\
			&\lesssim \|N^{-3/4}\|P_1v\|_{L_x^2 L_t^\infty}\|P_{\sim N}v\|_{L_x^\infty L_t^2}\|_{l^2_{N\gg 1}}\|u\|_{\tilde{X}^{1,\tilde{b}}_\lambda}\\
			&\lesssim \|u\|_{\tilde{X}^{1,\tilde{b}}_\lambda}\|v\|_{F}^2.
		\end{align*}
		For the second term one has
		\begin{align*}
			&\quad\|T_{M_\lambda}(u,\partial_x(P_{>1}vP_{>1}v))\|_{\tilde{X}^{9/4,\tilde{b}-1}_{\lambda}}^2\\
			&\lesssim \sum_{N\gg 1}N^{1/2}\left\|\int_{\mathbb{R}_{\tau_1,\xi_1}^2}\frac{\varphi_N(\xi_1)\psi_N(\xi)|\hat{u}(\tau_1,\xi_1)||\widehat{(P_{>1}v)^2}|(\tau-\tau_1,\xi-\xi_1)}{\langle\tau+\lambda\xi^2\rangle^{1-\tilde{b}}}\right\|_{L^2_{\tau,\xi}}^2.
		\end{align*}
		Without loss of generality we would assume that $\hat{u},\hat{v}\geq 0$. 
		Let $\lambda_1 = \langle\tau_1+\lambda\xi_1^2\rangle$, $\lambda_2 = \langle\tau_2-\xi_2^3\rangle$, $\lambda_3 = \langle\lambda\tau_3-\xi_3^3\rangle$, $\lambda = \langle\tau+\lambda\xi^2\rangle$. Note that for $\xi_1+\xi_2+\xi_3 = \xi$, $\tau_1+\tau_2+\tau_3 = \tau$, $|\xi_1|\ll N$, $|\xi|\sim N$ one has
		$$\lambda_1+\lambda_2+\lambda_3+\lambda\gtrsim |\xi_2^3+\xi_3^3+\lambda\xi_1^2-\lambda\xi^2|\sim N\max\{|\xi_2|,|\xi_3|\}^2.$$
		By symmetry we can assume $|\xi_2|\gtrsim |\xi_3|$. Let $f_1,f_2,g_1,g_2\geq 0$, 
		$\Omega$ be the set $|\xi_1|\ll |\xi|\sim N,\lambda_j\sim L_j,\lambda\sim L$, $|\xi_2|\sim N_2\gtrsim |\xi_3|\sim N_3$. We claim: There exists $\epsilon>0$ such that
		
		\begin{equation}\label{dualfour}
			\begin{aligned}
				&\quad\int_{\substack{\xi_1+\xi_2+\xi_3 = \xi,\\\tau_1+\tau_2+\tau_3 = \tau}}\langle\xi_1\rangle^{-1}\chi_{\Omega}f_1(\tau_1,\xi_1)g_1(\tau_2,\xi_2)g_2(\tau_3,\xi_3)f_2(\tau,\xi)\\
				&\lesssim (L_1L_2L_3)^{\frac{1}{2}}L^{1-\tilde{b}-\epsilon}N^{-\frac{1}{4}-\epsilon}(N_2N_3)^{-\frac{3}{4}}\|f_1\|_{L^2}\|g_1\|_{L^2}\|g_2\|_{L^2}\|f_2\|_{L^2}.
			\end{aligned}
		\end{equation}			
		Using this one has
		\begin{align*}
			&\quad\left\|\int_{\mathbb{R}_{\tau_1,\xi_1}^2}\frac{\varphi_N(\xi_1)\psi_N(\xi)|\hat{u}(\tau_1,\xi_1)||\widehat{(P_{>1}v)^2}|(\tau-\tau_1,\xi-\xi_1)}{\langle\tau+\lambda\xi^2\rangle^{1-\tilde{b}}}\right\|_{L^2_{\tau,\xi}}\\
			&\lesssim N^{-1/4-\epsilon/2} \|u\|_{\tilde{X}^{1,\tilde{b}}_\lambda}\|v\|_{F}^2
		\end{align*}
		and then concludes the proof of \eqref{uvvt}. To show \eqref{dualfour} we divide the integral into several parts. Let $L_{\max} = \max\{L_1,L_2,L_3,L\}$.
		
		If $L = L_{\max}\gtrsim NN_2^2$, $0<\epsilon\ll 1$, $\tilde{b}\leq 5/8-\epsilon$, then one has
		\begin{align*}
			&\quad\int_{\substack{\xi_1+\xi_2+\xi_3 = \xi,\\\tau_1+\tau_2+\tau_3 = \tau}}\langle\xi_1\rangle^{-1}\chi_{\Omega,L=L_{\max}}f_1(\tau_1,\xi_1)g_1(\tau_2,\xi_2)g_2(\tau_3,\xi_3)f_2(\tau,\xi)\\
			&\lesssim \|\mathscr{F}^{-1}(\langle\xi_1\rangle^{-1}\chi_{\lambda_1\sim L_1}f_1)\|_{L_x^4L_t^\infty}\|\mathscr{F}^{-1}(g_1\chi_{\lambda_2\sim L_2,|\xi_2|\sim N_2})\|_{L_x^\infty L_t^2}\\
			&\quad\cdot\|\mathscr{F}^{-1}(g_2\chi_{\lambda_2\sim L_3,|\xi_2|\sim N_3})\|_{L_x^4 L_t^\infty}\|f_2\|_{L^2}\\
			&\lesssim L_1^{1/2}\|f_1\|_{L^2} N_2^{-1}L_2^{1/2}\|g_1\|_{L^2}N_3^{1/4}L_3^{1/2}\|g_2\|_{L^2}\|f_2\|_{L^2}\\
			&\lesssim (L_1L_2L_3)^{1/2}L^{1-\tilde{b}-\epsilon}N^{-1/4-\epsilon}(N_2N_3)^{-3/4}\|f_1\|_{L^2}\|g_1\|_{L^2}\|g_2\|_{L^2}\|f_2\|_{L^2}.
		\end{align*}
		If $L_2 = L_{\max}\gtrsim NN_2^2$, $\tilde{b}\leq 7/12-4\epsilon/3$, one has
		\begin{align*}
			&\quad\int_{\substack{\xi_1+\xi_2+\xi_3 = \xi,\\\tau_1+\tau_2+\tau_3 = \tau}}\langle\xi_1\rangle^{-1}\chi_{\Omega,L=L_{\max}}f_1(\tau_1,\xi_1)g_1(\tau_2,\xi_2)g_2(\tau_3,\xi_3)f_2(\tau,\xi)\\
			&\lesssim \|\mathscr{F}^{-1}(\langle\xi_1\rangle^{-1}\chi_{\lambda_1\sim L_1}f_1)\|_{L^4_{x}L_{t}^\infty}\|g_1\|_{L^2}\\
			&\quad\cdot\|\mathscr{F}^{-1}(g_2\chi_{\lambda_2\sim L_3,|\xi_2|\sim N_3})\|_{L_x^\infty L_t^2}\|\mathscr{F}^{-1}(f_2\chi_{\lambda\sim L,|\xi|\sim N})\|_{L_x^4L_t^\infty}\\
			&\lesssim (L_1L_3L)^{1/2}N_3^{-1}N^{1/4}\|f_1\|_{L^2}\|g_1\|_{L^2}\|g_2\|_{L^2}\|f_2\|_{L^2}\\
			&\lesssim (L_1L_2L_3)^{1/2}L^{1-\tilde{b}-\epsilon}N^{-1/4-\epsilon}(N_2N_3)^{-3/4}\|f_1\|_{L^2}\|g_1\|_{L^2}\|g_2\|_{L^2}\|f_2\|_{L^2}.
		\end{align*}
		The argument for $L_3 = L_{\max}$ is similar. If $L_1 = L_{\max}$, $\tilde{b}\leq 7/12-4\epsilon/3$ one has
		\begin{align*}
			&\quad\int_{\substack{\xi_1+\xi_2+\xi_3 = \xi,\\\tau_1+\tau_2+\tau_3 = \tau}}\langle\xi_1\rangle^{-1}\chi_{\Omega,L=L_{\max}}f_1(\tau_1,\xi_1)g_1(\tau_2,\xi_2)g_2(\tau_3,\xi_3)f_2(\tau,\xi)\\
			&\lesssim \|f_1\|_{L^2}\|\mathscr{F}^{-1}(g_1\chi_{\lambda_2\sim L_2, |\xi_3|\sim N_2})\|_{L_x^\infty L_t^2}\\
			&\quad\cdot\|\mathscr{F}^{-1}(g_2\chi_{\lambda_3\sim L_3,|\xi_3|\sim N_3})\|_{L_x^4 L_t^\infty}\|\mathscr{F}^{-1}(f_2\chi_{\lambda\sim L,|\xi|\sim N})\|_{L_x^4 L_t^\infty}\\
			&\lesssim (L_2L_3L)^{1/2}N_2^{-1}N_3^{1/4}N^{1/4}\|f_1\|_{L^2}\|g_1\|_{L^2}\|g_2\|_{L^2}\|f_2\|_{L^2}\\
			&\lesssim (L_1L_2L_3)^{1/2}L^{1-\tilde{b}-\epsilon}N^{-1/4-\epsilon}(N_2N_3)^{-3/4}\|f_1\|_{L^2}\|g_1\|_{L^2}\|g_2\|_{L^2}\|f_2\|_{L^2}.
		\end{align*}
		Thus we obtain \eqref{dualfour} and then \eqref{uvv}.
		
		\eqref{buvv} is equivalent to $$\|N^{9/4} \|P_N(P_{> N/8}B(u,v)v)\|_{\tilde{X}_{\lambda}^{0,\tilde{b}-1}}\|_{l^2_{N}}\lesssim \|u\|_{\tilde{X}^{1,b}_{\lambda}}\|v\|_{F}^2.$$
		For the term $P_N(P_{\gtrsim N}B(u,v)P_1v) = P_N(B(u,P_{\sim N}v)P_1 v)$ we have
		\begin{align*}
			&\quad\|P_N(P_{\gtrsim N}B(u,v)P_1v)\|_{\tilde{X}^{0,\tilde{b}-1}_\lambda}\\
			&\leq \|(P_{\gtrsim N}B(u,P_{\sim N}v))P_1v\|_{L^2_{t,x}}\\
			&\lesssim N^{-3}\|(\mathscr{F}_x u)(t,\xi)\|_{L_\xi^1L_t^\infty}\|P_{\sim N }v\|_{L_{t,x}^2}\|v\|_F.
		\end{align*}
		Thus 
		\begin{align*}
			\|N^{9/4}(P_{> N/8}B(u,v)P_1v)\|_{\tilde{X}_{\lambda}^{0,\tilde{b}-1}}\|_{l^2_{N\gg 1}}&\lesssim \|\hat{u}\|_{L^1_{\tau,\xi}}\|N^{-3/4}P_{\sim N}v\|_{l^2_{N\gg 1}L_{t,x}^2}\|v\|_F\\&\lesssim \|u\|_{\tilde{X}^{1,\tilde{b}}_\lambda}\|v\|_F^2.
		\end{align*}
		For other cases, we would assume that $\hat{u}, \hat{v}\geq 0$. Let $K\geq 2$.
		\begin{align*}
			&\quad\|P_N(P_{> N/8}B(u,v)P_{K}v)\|_{\tilde{X}_\lambda^{0,\tilde{b}-1}}^2\\
			&\lesssim \left\|\int_{\mathbb{R}^2}\frac{\psi(\xi/N)\varphi_{>N/8}(\xi_1)\widehat{B(u,v)}(\tau_1,\xi_1)\widehat{P_K v}(\tau-\tau_1,\xi-\xi_1)}{\langle\tau+\lambda\xi^2\rangle^{2-2\tilde{b}}}~d\tau_1d\xi_1\right\|_{L^2_{\tau,\xi}}^2.
		\end{align*}
		If $K\lesssim N$ we only need to consider
		$P_{N}(B(P_{\ll N}u,P_{\sim N}v)P_K v)$. This case can reduce to \eqref{dualfour}. We have
		$$\|P_{N}(B(P_{\ll N}u,P_{\sim N}v)P_K v)\|_{\tilde{X}^{0,\tilde{b}-1}_\lambda}\lesssim N^{-13/4-\epsilon}\|u\|_{\tilde{X}^{1,\tilde{b}}_\lambda}\|v\|_F^2.$$
		Thus
		\begin{align*}
			\left\|\sum_{K\lesssim N}N^{9/4}\|P_{N}(B(u,v)P_K v)\|_{\tilde{X}^{0,\tilde{b}-1}_\lambda}\right\|_{l^2_{N\gg 1}}\lesssim \|u\|_{\tilde{X}^{1,\tilde{b}}_\lambda}\|v\|_{F}^2.
		\end{align*}
		If $K\gg N$ we only need to consider $P_N(B(P_{\ll K}u,P_{\sim K}v)P_Kv)$.
		Following the argument for \eqref{dualfour} one has
		\begin{align*}
			&\quad\|P_N(B(P_{\ll K}u,P_{\sim K}v)P_Kv)\|_{\tilde{X}^{0,\tilde{b}-1}_\lambda}\\
			&\lesssim N^{-13/4-\epsilon}\|u\|_{\tilde{X}^{1,\tilde{b}}_\lambda}K^{-3/2}\|P_{\sim K}v\|_{Y^{0,1/2,1}}\|P_K v\|_{Y^{0,1/2,1}}.
		\end{align*}
		Thus we obtain
		\begin{align*}
			\left\|\sum_{K\gg N}N^{9/4}\|P_{N}(B(u,v)P_K v)\|_{\tilde{X}^{0,\tilde{b}-1}_\lambda}\right\|_{l^2_{N\gg 1}}\lesssim \|u\|_{\tilde{X}^{1,\tilde{b}}_\lambda}\|v\|_{F}^2.
		\end{align*}
		We finish the proof of the lemma.
	\end{proof}
	\begin{lemma}\label{Rtildeu}
		Let $s>0, 0<T<1$. There exists $\epsilon_0>0$ such that for any $0<\epsilon\leq \epsilon_0$ one has
		$$\left\|\sum_NP_N(P_{>N/8 }\tilde{u}v)\right\|_{\tilde{X}^{s,-1/2+2\epsilon}_{\lambda,T}}\lesssim \lambda^{-1/2}\|\tilde{u}\|_{\tilde{X}^{s,1/2+\epsilon}_{\lambda,T}}\|v\|_{F_T}.$$
	\end{lemma}
	\begin{proof}[\textbf{Proof}]
		For low frequency of $v$ we have
		\begin{align*}
			\left\|\sum_NP_N(P_{> N/8}\tilde{u}P_1v)\right\|_{\tilde{X}^{s,-1/2+2\epsilon}_{\lambda,T}}
			&\lesssim \|N^sP_{> N/8}\tilde{u}P_1v\|_{L_T^2 l^2_N L_x^2}\\
			&\lesssim T^{1/2}\|\tilde{u}\|_{L_t^\infty H_x^s}\|P_1 v\|_{L_{x,T}^\infty}\\
			&\lesssim T^{1/2} \|\tilde{u}\|_{\tilde{X}^{s,1/2+\epsilon}_{\lambda,T}}\|v\|_{F_T}.
		\end{align*}
		For other parts we only need to show $$\|P_{N}(P_{> N/8}\tilde{u}P_{>1}v)\|_{\tilde{X}^{0,-1/2+2\epsilon}_\lambda}\lesssim \lambda^{-1/2} \|P_{> N/8}\tilde{u}\|_{\tilde{X}^{0,1/2+\epsilon}_\lambda}\|v\|_{Y^{-3/4,1/2,1}}.$$	
		In fact one has 
		$$\|\tilde{u}P_{>1}v\|_{\tilde{X}^{0,-1/2+2\epsilon}_\lambda}\lesssim \lambda^{-1/2} \|\tilde{u}\|_{\tilde{X}^{0,1/2+\epsilon}_\lambda}\|v\|_{Y^{-3/4,1/2,1}}.$$
		See for example Lemma 3.4 (a) in \cite{guo2010well}.
	\end{proof}
	\begin{lemma}\label{cubicinitial}
		Let $s>1/2, T>0$.  $u,v,w\in C_TH_x^s$. Then
		\begin{align*}
			\|u\bar{v}w\|_{\tilde{X}^{s,0}_{\lambda,T}}\lesssim T^{1/2}\|u\|_{C_TH_x^s}\|v\|_{C_TH_x^s}\|w\|_{C_TH_x^s}.
		\end{align*}
	\end{lemma}
	\begin{proof}[\textbf{Proof}]
		By algebraic property of $H^s$, $s>1/2$ we have
		\begin{align*}
			\|u\bar{v}w\|_{\tilde{X}^{s,0}_{\lambda,T}}=\|J^s(u\bar{v}w)\|_{L_{x,T}^2}\lesssim T^{1/2} \|J^su\|_{L_T^\infty L_x^2}\|J^sv\|_{L_T^\infty L_x^2}\|J^sw\|_{L_T^\infty L_x^2}.
		\end{align*}
		We conclude the proof.
	\end{proof}
	By Lemmas \ref{boundterm}--\ref{cubicinitial} and standard argument, we obtain
	\begin{prop}
		Let $(u,v) \in \tilde{X}_{\lambda,T}^{1,b}\times F_T$ be a solution of \eqref{scalingeq} with initial data $u_0\in H^{s_1}$, $5/4<s\leq 9/4$. Then $(u,v)\in C([0,T];H^{s}\times H^{-3/4})$. Also the mapping $H^{s}\times H^{-3/4}\longrightarrow C([0,T];H^{s}\times H^{-3/4}),~(u_0,v_0)\mapsto (u,v)$ is continuous.
	\end{prop}
	\begin{proof}[\textbf{Proof}]
		$(u,v) \in \tilde{X}_{\lambda,T}^{1,b}\times F_T$ be the solution of \eqref{scalingeq}. Then by Lemmas \ref{reguimpr}--  \ref{cubicinitial}, we obtain the unique solution $\tilde{u}\in \tilde{X}^{s,\tilde{b}}_{\lambda,T}$ of equation \eqref{eqfornorm} by standard contraction mapping argument. Then by Lemma \ref{boundterm}, $u = \tilde{u}+\lambda B(u,v)\in C([0,T];H^s)$. Since the mapping $(u_0,v_0)\mapsto B(u,v)$ and $(u_0,v_0)\mapsto \tilde{u}$ are continuous from $H^{s}\times H^{-3/4}$ to $C([0,T];H^s\times H^{-3/4})$, we finish the proof.
	\end{proof}
	\begin{rem}
		For the region $s_2>-3/4$, $s_2+2<s_1\leq s_2+3$, the argument is simpler than the case $s_2 = -3/4$. In fact, for $s_2>-3/4$, let $\lambda = 1$ in the former argument and we substitute $\tilde{X}^{1,b}_{\lambda}$, $F$ by $X^{s_2+2,b_1}$, $Y^{s_2,b_2}$ respectively. Let $(u,v)\in X^{s_2+2,b_1}_T\times Y^{s_2,b_2}_T$ be the solution of \eqref{model} with initial data $(u_0,v_0)\in H^{s_1}\times H^{s_2}$, $s_2+2<s_1\leq s_2+3$. Then we have $B(u,v)\in C_TH^{s_2+3}_x$ and
		\begin{align*}
			C_1(u,v),D_1(u,v),C_2(u),C_3(u,v),R(B(u,v),v)\in X^{s_2+3,\tilde{b}}_{T}
		\end{align*}
		for some $\tilde{b}>1/2$. Combining with Lemmas \ref{Rtildeu}--\ref{cubicinitial}, we obtain the unique solution $\tilde{u}\in X^{s_1,\tilde{b}}_{T}$ of equation \eqref{eqfornorm}. Then by the same argument as the case $s_2 = -3/4$, we know that $(u,v)$ is in $C([0,T];H^{s_1}\times H^{s_2})$ and relies on $(u_0,v_0)\in H^{s_1}\times H^{s_2}$ continuously. 
	\end{rem}
	
	\textbf{Acknowledgments:} The first and second authors are supported in part by the NSFC, grant 12171007. The second author is also supported in part by the NSFC, grant 12301116. The authors would like to thank Professors Boling Guo and Baoxiang Wang for their invaluable support and encouragement.

	\phantomsection 
	\bibliographystyle{amsplain}
	\addcontentsline{toc}{section}{References}
	\bibliography{reference}

	\scriptsize\textsc{Yingzhe Ban: The Graduate School of China Academy of Engineering Physics, Beijing, 100088, P.R. China}
	
	\textit{E-mail address}: \textbf{banyingzhe22@gscaep.ac.cn}
	\vspace{20pt}
	
	\scriptsize\textsc{Jie Chen: School of Sciences, Jimei University, Xiamen 361021, P.R. China}
	
	\textit{E-mail address}: \textbf{jiechern@163.com}
	\vspace{20pt}
	
	\scriptsize\textsc{Ying Zhang: The Graduate School of China Academy of Engineering Physics, Beijing, 100088, P.R. China}
	
	\textit{E-mail address}: \textbf{zhangying21@gscaep.ac.cn}
\end{document}